\documentclass[twoside,11pt]{article}

\usepackage{graphicx, subfigure}
\usepackage[top=1in,bottom=1in,left=1in,right=1in]{geometry}
\usepackage[numbers,sort&compress]{natbib} 
\usepackage{amsfonts, amsmath, amssymb, amsthm, bbm}
\usepackage[normalem]{ulem}
\usepackage{comment}
\usepackage{eufrak}
\usepackage{algorithm}
\usepackage{algorithmic}

\usepackage{color}
\definecolor{darkred}{RGB}{100,0,0}
\definecolor{darkgreen}{RGB}{0,100,0}
\definecolor{darkblue}{RGB}{0,0,150}

\usepackage{hyperref}
\hypersetup{colorlinks=true, linkcolor=darkred, citecolor=darkgreen, urlcolor=darkblue}
\usepackage{url}

\newtheorem{thm}{Theorem}
\newtheorem{prp}{Proposition}
\newtheorem{lem}{Lemma}
\newtheorem{cor}{Corollary}

\def\beq{\begin{equation}}
\def\eeq{\end{equation}}
\def\beqn{\begin{eqnarray*}}
\def\eeqn{\end{eqnarray*}}
\def\bitem{\begin{itemize}}
\def\eitem{\end{itemize}}
\def\benum{\begin{enumerate}}
\def\eenum{\end{enumerate}}
\def\bmult{\begin{multline*}}
\def\emult{\end{multline*}}
\def\bcenter{\begin{center}}
\def\ecenter{\end{center}}

\DeclareMathOperator*{\argmin}{arg\, min}

\DeclareMathOperator{\tr}{tr}

\def\cA{\mathcal{A}}
\def\cB{\mathcal{B}}
\def\cC{\mathcal{C}}

\def\cE{\mathcal{E}}
\def\cF{\mathcal{F}}

\def\cL{\mathcal{L}}

\def\cN{\mathcal{N}}

\def\cS{\mathcal{S}}

\def\cU{\mathcal{U}}

\def\bA{\mathbf{A}}
\def\bB{\mathbf{B}}

\def\bE{\E}

\def\bI{\mathbf{I}}

\def\bN{\mathbf{N}}

\def\bP{\P}
\def\bQ{\mathbf{Q}}
\def\bR{\mathbf{R}}
\def\bS{\mathbf{S}}
\def\bT{\mathbf{T}}
\def\bU{\mathbf{U}}

\def\bX{\mathbf{X}}

\def\bZ{\mathbf{Z}}

\def\ba{\mathbf{a}}

\def\1{{\mathbf 1}}

\newcommand\bSigma{{\boldsymbol\Sigma}}

\newcommand\bGamma{{\boldsymbol\Gamma}}

\def\bbB{\mathbb{B}}

\def\bbR{\mathbb{R}}

\def\bPi{\boldsymbol{\Pi}}

\newcommand{\E}{\operatorname{\mathbb{E}}}
\renewcommand{\P}{\operatorname{\mathbb{P}}}
\newcommand{\Var}{\operatorname{Var}}

\newcommand{\var}[1]{\operatorname{Var}\left(#1\right)}
\newcommand{\cov}[1]{\operatorname{Cov}\left(#1\right)}

\newcommand{\ol}{\overline}

\title{Optimal Sparsity Testing in Linear regression Model}
\author{Alexandra Carpentier and Nicolas Verzelen}

\begin{document}

\maketitle
\begin{abstract}
We consider the problem of sparsity testing in the high-dimensional linear regression model. The problem is to test whether the number of non-zero components (aka the sparsity) of the regression parameter $\theta^*$ is less than or equal to  $k_0$. We pinpoint the minimax separation distances for this problem, which amounts to quantifying how far a $k_1$-sparse vector $\theta^*$ has to be  from the set of $k_0$-sparse vectors so that a test is able to reject the null hypothesis with high probability. 
Two scenarios are considered. In the independent scenario, the covariates are i.i.d. normally~distributed  and the noise level is known. In the general scenario, both the covariance matrix of the covariates and the noise level are unknown. 
Although the minimax separation distances differ in these two scenarios, both of them actually depend on $k_0$ and $k_1$ illustrating that for this composite-composite testing problem both the size of the null and of the alternative hypotheses play a key role. 

\end{abstract}

\section{Introduction}

In the last decade, a lot of effort has been devoted to developing sound statistical methods for high-dimensional data.  Most of the estimation procedures rely on the assumption that the parameter of interest has some possibly unknown structure.  A prominent example is the high-dimensional linear regression problem where it is usually assumed that the regression parameter is sparse~\cite{buhlmann2011statistics}. Despite the pervasiveness of the sparsity assumption in the literature, very few contributions challenge this assumption. 

In this work, we tackle the largely ignored problem of assessing the sparsity of the regression parameter. Henceforth, we consider the random design high-dimensional linear regression model
\beq\label{eq:model_linear_regression}
Y =  \bX\theta^* + \sigma \epsilon\ ,
\eeq
where the unknown parameter $\theta^*$ belongs to $\bbR^p$, the noise vector $\epsilon\in \bbR^n$ follows a standard normal distribution and where the rows of $\bX$ are i.i.d. sampled according to the normal distribution $\cN(0,\bSigma)$. 
For a given integer $k_0$, we study  the problem of testing whether the vector $\theta^*$ has at most $k_0$ non-zero components.

\subsection{Minimax separation distance}

Before discussing our contribution, we first formalize the sparsity testing problem. 
For a vector $\theta$, $\|\theta\|_0$ denotes its number of non-zero entries. Then, given a non-negative integer $k_0\in [0,p]$, write 
$\bbB_0[k_0] = \{\theta \in \mathbb R^p : \|\theta\|_0 \leq k_0\}$ 
for the set of $k_0$-sparse vectors $\theta$. Rephrasing our aim, we want to test whether $\theta^*$ belongs  to $\bbB_0[k_0]$.

In order to assess the quality of a testing procedure, we rely on the framework of minimax separation distances~\cite{ingster_suslina} which is described in the following paragraphs. Let $\|.\|_2$ denote the $l_2$ distance in $\mathbb{R}^p$. 
For any $\theta^*\in \mathbb{R}^p$,  $d_2(\theta^*,\bbB_0[k_0]) := \inf_{u \in \bbB_0[k_0]} \|\theta - u\|_2$ stands for its $l_2$ distance to the set of $k_0$-sparse vectors. Intuitively, any $\alpha$-level test $\phi$ of the null hypothesis $\{\theta^* \in \mathbb B_0[k_0]\}$ cannot reject the null with high probability when $d_2(\theta^*,\bbB_0[k_0])$ is too small. In this work, we aim at characterizing the smallest distance $\rho$, such that there exists a test achieving a small type I error probability and rejecting the null with high probability whenever $d_2(\theta^*,\bbB_0[k_0])$ is larger than $\rho \sigma $. These informal definitions are made precise in the next subsection. In the sequel, $\bP_{\theta^*,\bSigma,\sigma}$ stands for the distribution of $(Y,\bX)$ in \eqref{eq:model_linear_regression}.

In high-dimensional linear regression, the intrinsic difficulty of estimation or testing problems sometimes depends on some specific features such as the knowledge of the noise level $\sigma^2$ or the knowledge of the distribution of the design. In this work, we focus on two emblematic settings. In the {\bf independent} setting, we assume that the covariates are independent ($\bSigma= \bI_p$) and the noise level $\sigma$ is known. In the {\bf general} setting, both the covariance of the covariates and the noise level are unknown. 

\subsubsection{Independent setting}

Fix a positive integer $1\leq\Delta \leq p-k_0$, we consider the alternative hypothesis where $\theta^*$ is $k_0+\Delta$-sparse. Given $\rho>0$ and a test $\phi$, we introduce its risk $R(\phi;k_0,\Delta,\rho)$ as the sum of the type I and type II error probabilities 
\beq\label{eq:risk}
R(\phi;k_0,\Delta,\rho):= \sup_{\theta^* \in \bbB_0[k_0]}\P_{\theta^*,\bI_p,\sigma}[\phi=1] +  \sup_{\theta^* \in \bbB_0[k_0+\Delta],\ d_2(\theta^*; \bbB_0[k_0])\geq  \rho \sigma}\P_{\theta^*,\bI_p,\sigma}[\phi=0]\ , 
\eeq
where we only consider parameters $\theta^*$ in the alternative hypothesis that lie at a distances $d_2$ higher than $\rho\sigma$ from the null. For a fixed (known) $\gamma\in (0,1)$, the separation distance 
$\rho_\gamma[\phi;k_0,\Delta]$ of $\phi$ is the largest $\rho$ such that its risk is higher than $\gamma$, i.e.
$\rho_{\gamma}(\phi;k_0,\Delta):= \sup \left\{\rho>0\ |R(\phi;k_0,\Delta,\rho)>\gamma\right\}$.
Parameters $\theta^*$ lying at a distance larger than $\sigma \rho_{\gamma}(\phi;k_0,\Delta)$ from the null are therefore detected with probability higher than 1-$\gamma$ by $\phi$. 
Finally, the minimax separation distance is
\beq\label{eq:separationkvminmax}
\rho^*_{\gamma}[k_0,\Delta]:= \inf_{\phi}\rho_{\gamma}(\phi;k_0,\Delta )\ ,
\eeq
where the infimum is taken over all tests $\phi$.

\subsubsection{General setting}

In the general case, neither the covariance matrix $\bSigma$ of the covariates, nor the noise level $\sigma$ is known. We only assume that the the eigenvalues of $\bSigma$ are bounded away from zero and from infinity. 
Respectively write $\eta_{\min}(\bSigma)$ and $\eta_{\max}(\bSigma)$ for its smallest and largest eigenvalues.
Given $\eta>1$, define 
\beq\label{eq:definition_Ueta}
\cU(\eta)=\{\bSigma: \eta^{-1}\leq \eta_{\min}(\bSigma)\leq \eta_{\max}(\bSigma)\leq \eta\}\ .
\eeq
Fix $\rho>0$.  In this general model, the risk of a test $\phi$ is now taken as
\[
\bR_{g}(\phi;k_0,\Delta,\rho):= \sup_{\theta^* \in \bbB_0[k_0],\ \bSigma\in \cU(\eta),\ \sigma>0}\P_{\theta^*,\bSigma,\sigma}[\phi=1] +  \sup_{ \sigma >0, \theta^* \in \bbB_0[k_0+\Delta],\ d_2(\theta^*; \bbB_0[k_0])\geq \sigma \rho,\ \bSigma \in \cU[\eta]}\P_{\theta^*,\bSigma,\sigma}[\phi=0]\ .   
\]
Since both $\bSigma$ and $\sigma$ are unknown, we evaluate the type I and type II error probabilities uniformly over all $\sigma>0$ and all $\bSigma\in \cU[\eta]$. The class of covariance matrices is constrained in $\cU[\eta]$ in order to preclude too difficult settings where the eigenvalues of $\bSigma$ differ too much to each other. Then, as in the previous subsection, the  separation distance of a test $\phi$ is 
$\boldsymbol{\rho}_{g,\gamma}(\phi;k_0,\Delta):= \sup \left\{\rho>0\ |\bR_g(\phi;k_0,\Delta,\rho)>\gamma\right\}$ and the minimax separation distance in the general setting is defined by
\beq\label{eq:separationkvminmax_general}
\boldsymbol{\rho}^*_{g,\gamma}[k_0,\Delta]:= \inf_{\phi}\boldsymbol{\rho}_{g,\gamma}(\phi;k_0,\Delta )\ .
\eeq

\medskip 

\noindent 
In this work, we address both independent and general settings. More specifically, 
\begin{enumerate}
 \item[(i)] We characterize the minimax separation distances   in both the independent ($\rho^*_{\gamma}[k_0,\Delta]$) and the general ($\boldsymbol{\rho}^*_{g,\gamma}[k_0,\Delta]$) settings by providing upper and lower bounds that match (up to a polylogarithmic loss in some regimes). 
 
 \item[(ii)] We introduce computationally feasible testing procedures that (almost) simultaneously achieve this minimax separation distance over all $\Delta$. 
\end{enumerate}

\subsection{Previous results and related literature}

Before further describing our contribution, we first discuss related  results in the literature.

\paragraph{Signal detection.} The signal detection problem which amounts to testing whether $\theta^*=0$ is a special instance of the sparsity testing problem (corresponding to $k_0 = 0$). Signal detection in the Gaussian vector model (which corresponds to an orthogonal design) has been extensively studied ~\cite{ingster_suslina,baraud02,jin2004, collier2015minimax,collier2016optimal} in the last fifteen years. More recently, this problem has also been investigated in the random design linear regression model~\cite{2010_EJS_Ingster, 2011_AS_Arias-Castro, carpentier2018minimax}. 

To simplify the discussion, let us consider the high-dimensional setting where  $p\geq n^{1+\zeta}$ for
 some fixed constant $\zeta >0$. Then, one can deduce from \cite{2010_EJS_Ingster} that the minimax separation distance in the independent setting satisfies
 \[
 \rho_{\gamma}^{*2}[0,\Delta] \asymp_{\gamma,\zeta} \min\left[ \frac{\Delta\log\big(p\big)}{n} ,  n^{-1/2}\right]\ ,
 \]
where $f(\Delta,n,p) \asymp_{\gamma,\zeta} g(\Delta,n,p)$ means that there exist  positive  constants $c_{\gamma,\zeta}$ and $c'_{\gamma,\zeta}$ (possibly depending on $\gamma$ and $\zeta$) such that 
$f(\Delta,n,p)\leq c_{\gamma}  g(\Delta,n,p)\leq c'_{\gamma}f(\Delta,n,p)$ for all $\Delta$, $n$, and $p$. For $\Delta \leq \sqrt{n}/\log(p)$, this separation distance is achieved by measuring the raw correlations between the response and the covariates and rejecting when too many of these correlations are unusually large. This can be done through the Higher-Criticism scheme~\cite{2010_EJS_Ingster, 2011_AS_Arias-Castro}. For denser alternatives ($\Delta \geq \sqrt{n}/\log(p)$), we start from the identity $\mathbb E Y_i^2=\sigma^2+ \|\theta^*\|_2^2$ (and the $Y_i$ are i.i.d.). Hence, a test rejecting when the empirical mean of $Y_i^2$ is significantly larger than $\sigma^2$ achieves the optimal squared separation distance of order $n^{-1/2}$~\cite{2010_EJS_Ingster, 2011_AS_Arias-Castro}. In the specific regime where $p$ is of the same order as $n$, and $\Delta$ is close to $\sqrt{n}$, the analysis has to be refined, see \cite{carpentier2018minimax}.

\medskip

In the general setting (unknown $\bSigma$ and unknown $\sigma$), it has been proved in \cite{2010_AS_Verzelen} that, 
\beqn 
\boldsymbol{\rho}_{g,\gamma}^{*2}[0,\Delta] &\asymp_{\gamma,\xi} & \frac{\Delta\log\big(p\big)}{n} \quad \quad \text{ if }\Delta \leq p^{1/2 - \xi}\wedge \frac{n}{\log(p)} \text{ for any fixed $\xi\in (0,1/2)$}  \  ;\\                                                 
\boldsymbol{\rho}_{g,\gamma}^{*2}[0,\Delta] &\geq& c_{\gamma} \frac{\sqrt{p}}{n}\quad \quad    \text{ if } \Delta\geq \sqrt{p} \ .
\eeqn
However, for sparse alternatives,  the corresponding test in~\cite{2010_AS_Verzelen} relies on a $l_0$ type variable selection method and has therefore exponential computationally complexity. For denser alternatives $(\Delta\geq \sqrt{p})$, the lower bound entails that the minimax separation distance is large whenever $p\geq n^2$. Comparing both the independent and the general settings, we observe  that the separation distance is significantly larger in the general setting for dense alternatives $\Delta\geq \sqrt{n}/\log(p)$.

\paragraph{Composite-composite testing problems and related work.}
An important difference between the signal detection ($k_0=0$) problem and the general sparsity testing problem ($k_0>0$) is that, in the latter, the null hypothesis is composite, thereby making the analysis of the problem more challenging.  Up to our knowledge, the analysis of such composite problems has been considered only in a few work~\cite{Juditsky_convexity,baraud2005testing,comminges2013minimax,carpentier2015testing}, although the  problems of constructing adaptive confidence regions (e.g.~\cite{cai2004adaptation,cai2006adaptive, MR2906872,nickl_vandegeer, cai2017confidence,cai2016accuracy}) or of functional estimation (e.g.~\cite{lepski1999estimation, gayraud2005adaptive,cailow2011,MR2382653,MR2589318}) are also related to such testing problems.

In particular, Nickl and Van de Geer~\cite{nickl_vandegeer} consider the problem of constructing adaptive and honest confidence sets for $\theta^*$ in the linear regression model~\eqref{eq:model_linear_regression} with known variance $\sigma^2$. To achieve adaptivity to the unknown sparsity of $\theta^*$, Nickl and van de Geer need to test hypotheses of the form $\|\theta^*\|_0\leq k_0$. Following the so-called ``infimum testing'' principle, described in a systematic way in~\cite{gine2015mathematical}, 
they consider the statistic $\inf_{\theta\in \bbB_0[k_0]}\|Y-\bX\theta\|_2^2/n$. This statistic corresponds to the infimum of the empirical variance when one corrects $Y$ by a $k_0$-sparse vector $\theta$. Under the null, this statistic is not much larger than the noise level $\sigma^2$. This leads them to derive
$$\rho_{\gamma}^{*2}[k_0,\Delta] \leq c_{\gamma}  \left[n^{-1/2} + \frac{k_0\log(p)}{n}\right]\ , $$
for some $c_{\gamma}>0$. Comparing this bound with its counterpart in the signal detection problem ($k_0=0$), we observe an increase by an additive term $\tfrac{k_0\log(p)}{n}$ accounting for the complexity of the null hypothesis.

Up to our knowledge, it is still unknown whether the upper bound of Nickl and van de Geer is optimal (that is whether $\rho_{\gamma}^{*2}[k_0,\Delta]$ actually depends on $k_0\log(p)/n$). In this manuscript,  we answer this open question, this for all $k_0$ and $\Delta$.

\paragraph{Sparsity testing in the Gaussian sequence model.}
The Gaussian sequence model $Y=\theta^*+ \sigma\epsilon$ corresponds to case $p=n$ and $\bX= \bI_p$. In~\cite{carpentier2017adaptive}, we have pinpointed the minimax separation distances for all $k_0$ and $\Delta$ both when $\sigma$ is known and $\sigma$ is unknown.  In particular, the optimal separation distance actually depends on the size $k_0$ of the null hypothesis for large $k_0$ but is significantly smaller than what is obtained by infimum tests strategies such as those in~\cite{gine2015mathematical}.

Generally speaking,~\cite{carpentier2017adaptive} is closely related to the aims and results of this paper, but there is a significant challenge in adapting the results in~\cite{carpentier2017adaptive} which are available for the Gaussian sequence setting, to the linear regression setting.

Related to this problem, some authors~\cite{MR2382653,MR2420411, MR2325113,MR2589318} have considered the problem of estimating $\|\theta^*\|_0$ in the Gaussian sequence model in a Bayesian framework where all $\theta^*_i$'s are sampled according to some mixture distribution. Although some of the ideas can be borrowed from their work, this Bayesian setting is quite different (see~\cite{carpentier2017adaptive} for a discussion).

\subsection{Our results}

In this paper, we characterize the minimax separation distances $\rho^{*}_\gamma[k_0,\Delta]$ and $\boldsymbol{\rho}_{g,\gamma}^{*2}[0,\Delta]$. To alleviate the discussion, we restrict ourselves throughout this paper to the high dimensional regime $p \geq n^{1+\zeta}$ where $\zeta>0$ is an arbitrarily small absolute constant.

\paragraph{Independent setting.} We establish matching (up to a multiplicative constants depending on $\gamma$) upper and lower bounds for  $\rho^{*}_\gamma[k_0,\Delta]$, this for almost all values of $k_0$ and $\Delta$; see Table~\ref{Tab:KV} for a summary of these results. An aggregated test is also shown to simultaneously achieve the optimal separation distance for all $\Delta>0$, entailing that adaptation to the sparsity is possible for this problem. In our exhaustive picture of $\rho^{*}_\gamma[k_0,\Delta]$, some of the regimes in $k_0$ and $\Delta$ are addressed by simple extensions of signal detection tests. However, other regimes turn out to be more challenging and require novel ideas. In what follows, we briefly mention these original aspects.
\begin{itemize}
\item We prove that, when $k_0 \geq c n/\log(p)$, then the testing problem
becomes extremely difficult, in the sense that the separation distance $\rho^*_{\gamma}[k_0,\Delta]$ is very large. For  $k_0 \geq n$, this separation distance is even infinite.  This is not unexpected since identifiability problems arise in this regime.
\item For moderate $k_0\in [\frac{\sqrt{n}}{\log(p)}, p^{1/2 - \zeta}]$ and large $\Delta$, we prove that the upper bound of Nickl and Van de Geer~\cite{nickl_vandegeer} turns out to be optimal, i.e.~the squared minimax separation distance is achieved  by their infimum test and is of the order of $\frac{k_0 \log(p)}{n}$. The general idea is to reduce the problem of sparsity testing (with known variance) to a detection problem with unknown variance. 

\item For larger $k_0 \in [p^{1/2 - \zeta}, cn/\log(p)]$ (where $\zeta>0$ is an arbitrarily small absolute constant, and where $c>0$ is an absolute constant), then both upper and lower bounds are new. The lower bound is based  on moment matching strategies and best polynomial approximation akin to those of  \cite{carpentier2017adaptive} in the Gaussian model. But the derivation is significantly more involved in the regression setting. For small $\Delta$ ($\Delta\le k_0$), an optimal test is built using any estimator of $\theta^*$ achieving a small $l_{\infty}$ error (see e.g. \cite{javanmard2014confidence, zhang2014confidence,van2014asymptotically,javanmard2018debiasing}). The test simply rejects when this estimator has more than $k_0$ unusually large entries. For denser alternatives ($\Delta \geq k_0$), the approach is quite different. We build a statistic based on the empirical Fourier transform of some correction of the raw correlations between the covariates and the responses $Y$. This approach is reminiscent of sparsity estimators in~\cite{MR2420411,carpentier2017adaptive} in the Gaussian sequence model.

\end{itemize}

 \begin{table}
\caption{Square minimax separation distances $\rho^{*2}_\gamma[k_0,\Delta]$ in the independent setting  for $k_0 \in [1,p-1]$ and $\Delta\in [1,p-k_0]$ when $p\geq n^{1+\zeta}$ with a fixed $\zeta>0$. Separation distances are given up to constants that may depend on $\gamma$ and $\zeta$. }
\label{Tab:KV}
 \begin{center}
 \begin{tabular}{c|c|c}
 $k_0$& $\Delta$ &  $\rho^{*2}_\gamma[k_0,\Delta]$\\ \hline \hline
$k_0 \leq p^{1/2-\zeta}$ & $1\leq \Delta \leq k_0 + \frac{\sqrt{n}}{\log(p)}$ & $ \frac{\Delta \log(p)}{n} $ \\ 
&  $k_0+\frac{\sqrt{n}}{\log(p)} \leq \Delta \leq p-k_0$  & $\frac{1}{\sqrt{n}} + \frac{k_0\log(p)}{n}$ \\ 
\hline
$p^{1/2+\zeta} \leq k_0 \leq c_\gamma \frac{n}{\log(p)}$ &  $1 \leq \Delta \leq k_0p^{-\zeta}$
&  $\frac{\Delta \log(p)}{n}$\\
 & 
$k_0\leq  \Delta \leq p-k_0$ & $\frac{k_0}{n\log(p)}$\\
 \end{tabular}
 \end{center}
\end{table}

\paragraph{General setting.}   We derive lower and upper bounds of the minimax separation distance $\boldsymbol{\rho}^*_{g,\gamma}[k_0,\Delta]$. These bounds match except  in the large $k_0$ and $\Delta$ regime, where there is a $\log^2(p)$ mismatch. See  Table~\ref{fig:UV} for a summary of the results. As in the independent setting, we emphasize below the most novel ingredient of our analysis. 
\begin{itemize}
\item Achieving the optimal squared distance $\Delta\log(p)/n$ could be easily done if one has access to an estimator whose $l_{\infty}$ distance to $\theta^*$ is less than $\sigma\sqrt{\log(p)/n}$ with high probability. However, such an estimator is unknown for general covariance matrices $\bSigma \in \cU(\eta)$. For $\|\theta^*\|_0\geq \sqrt{n}$, it is even proved that no such estimator exists~\cite{cai2017confidence}. Here, we first select a reasonable candidate for the support of $\theta^*$ by relying on the non-convex penalized least-square estimator MCP~\cite{zhang2010nearly}. Then, a test based on the restricted least-squares estimator applied to the selected subset is shown to achieve the desired separation distance. We also introduce an alternative test based on an iterative application of a projected version of the square-root Lasso.
\end{itemize}

\begin{table}
\caption{Square minimax separation distances in the general setting (in the $\asymp_{\gamma,\eta}$ sense, see Subsection~\ref{ss:not}). We report in this table only the case where $n^{1+\zeta} \leq p \leq n^{2 - \zeta}$, where $\zeta\in (0,1)$ can be chosen arbitrarily small. 
LB stands for Lower bound and UB stands for upper bound. }
\label{fig:UV}
 \begin{center}
 \begin{tabular}{c|c|c}
 $k_0$& $\Delta$ &  $\boldsymbol{\rho}_{g,\gamma}^{*2}[k_0,\Delta]$\\ \hline \hline
$k_0 \leq p^{1/2-\zeta} $ & $1\leq \Delta \leq p^{1/2-\zeta} \land k_0$ & $ \frac{\Delta \log(p)}{n}$ \\ 
& $ p^{1/2+\zeta} \land k_0 \leq \Delta \leq p-k_0$ & $\frac{\sqrt{p}}{n}$ \\ 
\hline
$p^{1/2+\zeta} \leq  k_0 \leq c_\gamma \frac{n}{\log(p)}$ &  $1 \leq \Delta \leq k_0p^{-\zeta}$
&  $\frac{\Delta \log(p)}{n}$\\
 & 
$k_0\leq \Delta \leq p-k_0$ & LB : $\frac{k_0}{n\log(p)}$\\
 & 
 & UB : $\frac{k_0\log(p)}{n}$\\
 \end{tabular}
 \end{center}
\end{table}

\subsection{Other related work}

Two recent works~\cite{zhu2017projection,javanmard2017flexible} have among other things consider general testing problems that encompass the sparsity testing problem. These two contributions assess the quality of their tests according to the $l_\infty$ separation distance (instead of $l_2$ as we do here) to the null hypothesis, i.e.~$d_\infty(\theta^*;\bbB_0[k_0])= \inf_{\theta \in \mathbb B_0[k_0]} \|\theta^* - \theta\|_\infty$.
In their setting, the covariance $\bSigma$ of the covariates is unknown but its inverse $\bSigma^{-1}$ is assumed to be sparse (each row of $\bSigma^{-1}$ has at most than $n/\log(p)$ non-zero entries) so that it can be reasonably well estimated. 
In that setting, the computationally feasible test in~\cite{zhu2017projection} has a small type II error probability when $k_0 \log(p)$ is much smaller than $n^{1/4}$ and when $d_\infty(\theta^*;\bbB_0[k_0])\geq c \sigma n^{-1/4}$. 

In~\cite{javanmard2017flexible}, Javanmard and Lee use a test based on the debiased Lasso. It achieves a small type I error probability. Whenever $(k_0+\Delta) \log(p)$ is much smaller than $\sqrt{n}$, and also $d_\infty(\theta^*;\bbB_0[k_0])\geq c \sigma \sqrt{\log(p)/n}$, its type II error probability is also small.
Translating these results in the $l_2$ separation distance setting, we observe that this test achieves a squared separation distance $\Delta \log(p)/n$ which, in view of Table~\ref{fig:UV}, is optimal for small $\Delta$. Their approach could be used instead of ours in their setting. However, we stress out that they achieve this bound to the price of considering  a much more restricted class of covariance matrices than $\cU(\eta)$ - they need that each row of $\bSigma^{-1}$ is at most $n/\log(p)$ sparse, while $\cU(\eta)$ contains all matrices $\bSigma$ that have their spectrum contained in $[\eta^{-1}, \eta]$. 

\medskip

A recent line of work has focused on testing the nullity of a given subset of coordinates of $\theta^*$  (e.g.~\cite{zhu2018linear, zhu2018significance, bradic2018testability, javanmard2018debiasing, van2014asymptotically, zhang2014confidence,cai2017confidence}), but both the settings and the methodology are quite different.

\subsection{Notation}\label{ss:not}

For any positive integer $d$ and $u \in \mathbb R^d$, we write $\cS(u)= \{i: u_i\neq 0\}$ for the support of a vector $u$.
For $u\in \mathbb{R}^d$ and $S\subset [d]$, we write $u_{S}= (u_i \1_{i\in S})_i$ for the vector in $\mathbb{R}^d$ whose values outside $S$ have been set to $0$. For a vector $\gamma$, $\gamma_{(i)}$ stands for its $i$-th  largest (in absolute value) entry. Given $S\subset \{1,\ldots,p\}$, $\overline{S}$ stands for its complement.

In the sequel, $c$, $c_1$, $c'$ denote numerical positive constants that may vary from line to line. Given some quantity $\delta$,  $c_{\delta}$ stands for a positive constant possibly depending on $\delta$ that may vary from line to line. Underlined constant such as $\underline{c}$, $\underline{c}^{(1)}$ do not vary in the paper.

Let $a,b\in \mathbb R$ be two functions that may depend on several quantities such as $n,p,\Delta, k_0$ and let $u\in \mathbb{R}$.  We write $a\lesssim_u b$ (resp.~$a\approx_u b$) if there exists a constant $c_u>0$ that depends only on $u$ (resp.~two constants $c_u^+, c_u^->0$ that depend only on $u$) such that $a\leq c_u b$ (resp.~such that $c_u^- b\leq a \leq c_u^+ b$).

For $x>0$, $\lfloor x\rfloor$ (resp. $\lceil x\rceil$) stands for the largest (resp. smallest) integer which is less (resp. greater) or equal to $x$. Also, $\log_2$ stands for the binary logarithm. Finally, $\overline{\Phi}$ stands for the tail distribution function of a standard normal distribution.

\section{Independent setting} \label{sec:independent_design}\label{sec:ides}
To simplify the notation, we denote $\P_{\theta^*,\sigma}$ the distribution of the data when $\bSigma$ is the identity matrix. Recall that we are especially interested in the high-dimensional setting. This is why we shall sometimes assume that $p\geq n$ or even $p\geq n^{1+\zeta}$ for some $\zeta>0$ arbitrarily small.

\subsection{Minimax lower bound}

As a starting point, we prove that, when the size $k_0$ of the null hypothesis is too large, consistent testing is impossible. Indeed, assume that $k_0\geq n$. Then, for any $(Y,\bX)\in \mathbb{R}^{n}\times \mathbb{R}^{n\times p}$ such that $\mathrm{Rank}(\bX)\geq n$, there exists $\theta\in \bbB_0[k_0]$  that perfectly fits this sample ($Y=\bX\theta$)  and it is therefore impossible to decipher whether $\theta^*$ is $k_0$-sparse or not. The following proposition formalizes this observation.

\begin{prp}\label{prp:k0large}
If $k_0 \geq n$, then, for any $\gamma<1/2$, and  $1\leq \Delta\leq p-k_0$, we have 
$\rho^*_{\gamma}[k_0,\Delta] = \infty$.
\end{prp}

In the sequel, we therefore restrict ourselves  to the case where $k_0 <  n$. The next theorem provides a lower bound for the minimax separation distance of the sparsity testing problem.

\begin{thm}\label{thm:lbkvkd}
Assume that $p\geq 2n$. 
There exist positive numerical constants $c_1$--$c_5$ such that the following holds for all $\gamma\leq 0.06$ and for all $p\geq c_1$.  For $1\leq \Delta\leq p-k_0$, one has
\beq\label{eq:lower_lbkvkd}
\rho_{\gamma}^{*2}[k_0,\Delta] \geq c_1 \left\{
\begin{array}{ccc}
\min\Big[ \frac{1}{\sqrt{n}} + \frac{k_0}{n}\log\big[1+ \frac{\sqrt{p}}{k_0}\big] , \frac{\Delta}{n} \log(1+\frac{\sqrt{p}}{\Delta}) \Big]\ & \text{ if } & 0\leq k_0\leq \sqrt{p}\wedge n \ ;\\
\frac{\Delta}{n } \frac{\log^2\big[1+ \sqrt{\frac{k_0}{\Delta}}\big]}{\log(p)}& \text{ if } & \sqrt{p}< k_0< n\ .
\end{array}
\right.
\eeq
Furthermore, if $p\geq c_2 n^2$ and $k_0\geq c_3 n /\log(\sqrt{p}/n)$, then 
\beq\label{eq:lower_ultra_high_known_variance}
\rho_{\gamma}^{*2}[k_0,\Delta] \geq c_4\frac{\Delta \wedge k_0}{n} \log\left(2\vee   \frac{\sqrt{p}}{k_0}\right)e^{c_5\frac{k_0}{n}\log(1+\frac{\sqrt{p}}{k_0})}\ ,
\eeq
for all $1\leq \Delta\leq p-k_0$. 
\end{thm}

In particular, \eqref{eq:lower_ultra_high_known_variance} entails that the sparsity testing problem turns out to be extremely difficult in the regime $n/\log(p)\lesssim k_0\lesssim n$ (at least when $p\geq n^2$). 

The different regimes in \eqref{eq:lower_lbkvkd} will be discussed together with the upper bounds at the end of the section. Let us shortly comment on the proof of Theorem \ref{thm:lbkvkd}. The functional $\rho^*_\gamma[k_0,\Delta]$  is (almost) nondecreasing with respect to $k_0$. As a consequence, the lower bound $\frac{\Delta}{n} \log(1+\frac{\sqrt{p}}{\Delta})$ is a straightforward consequence of the analysis of the detection problem e.g.~in~\cite{2010_EJS_Ingster}. 

The two lower bounds $\frac{1}{\sqrt{n}} + \frac{k_0}{n}\log\big[1+ \frac{\sqrt{p}}{k_0}\big]$ and \eqref{eq:lower_ultra_high_known_variance} are based on a reduction argument. The proof stems from the fact it is impossible to decipher between two sets of hypothesis if these two sets of hypotheses are almost indistinguishable from a third party hypothesis. Here, the third party hypothesis corresponds to $\theta^*=0$ and a tailored noise variance $\sigma' > \sigma$. Plugging minimax lower bounds for detection with {\it unknown} variance allows us to get the desired rate. See the proof for more details.  

In fact, it is most challenging to prove the minimax lower bound in the regime $\Delta >k_0 > \sqrt{p}$ as we cannot apply any reduction technique to signal detection problem and we need to take into account that both the null and the alternative hypotheses are composite. As for the Gaussian sequence model~\cite{carpentier2017adaptive}, we use a general moment matching technique~\cite{lepski1999estimation}, but the non-orthogonal design matrix $\bX$ makes the computations more tricky.

\subsection{Testing procedures}

In this subsection, we fix $\alpha$ and $\delta \in (0,1)$. 
We now introduce three testing procedures whose combination leads to matching the previous minimax lower bound.

Without loss of generality, we assume that $n$ is divisible by $3$ and we divide the sample $(Y,\bX)$ into three subsamples $(Y^{(1)},\bX^{(1)})$ and  
$(Y^{(2)},\bX^{(2)})$ and $(Y^{(3)},\bX^{(3)})$ of equal size $m=n/3$. For $i=1,2,3$, we write $\P_{\theta^*,\sigma}^{(i)}$ for the probability according to the $i$-th sub-sample. 
In fact, some of the tests introduced below only use the first two subsamples. Nevertheless, we use three subsamples throughout the paper to simplify the presentation.

To characterize the performances of the testing procedures, we shall control the type I error probability uniformly over the null hypothesis and control the type II error probability on some 'large'  parameter subset of the alternative. To simplify the statements of the results we shall refer to these two properties as {(\bf P1}) and {(\bf P2}) as defined below.

\smallskip

\noindent 
{\bf Property  P1}. A test $\phi$ satisfies {(\bf P1}[$\alpha$]) if its type I error probability is less than or equal to $\alpha$, that is 
$
 \sup_{\theta^* \in \bbB_0[k_0]}\P_{\theta^* ,\sigma}[\phi=1]\leq \alpha$

 \smallskip

\noindent 
{\bf Property  P2}. A test $\phi$ satisfies ({\bf P2}[$\beta$]) on a set  $\Theta$ if its type II error probability is uniformly less than or equal to $\beta$, uniformly on $\Theta$, that is 
$ \inf_{\theta^* \in \Theta}\P_{\theta^* ,\sigma}[\phi=1]\geq 1-\beta$

\medskip

Following the discussion in the previous subsection, we restrict our attention to sparsities $k_0$ that are less than $n/\log(p)$. This is formalized in the following condition ($\bA[\alpha]$) where $\underline{c}^{(\bf A)}, \underline{c}^{(\bf A)'}$ are numerical constants (respectively small enough for $\underline{c}^{(\bf A)}$ and large enough for $\underline{c}^{(\bf A)'}$) whose values are constrained in Propositions \ref{prp:t}--\ref{prp:analyse_Z_i}. 

\medskip 

\noindent 
($\bA[\alpha]$)\centerline{$(k_0\vee 1) \log(\frac{p}{\alpha})+ \log^2\big(\frac{p}{\alpha }\big) \leq \underline{c}^{(\bf A)} n$\,\text{ and }\, $p\geq \underline{c}^{(\bA)'}\ .$\quad }

\subsubsection{Test $\phi^{(t)}$  based on a $l_{\infty}$ estimation of $\theta^*$}\label{sec:test_l_infty}

The first test aims at detecting whether $\theta^*$ contains at least $k_0+1$ 'large' entries. In order to do so, we need to build a reasonable $l_{\infty}$ estimator of $\theta^*$. Note that estimators based on the debiased Lasso have already been proved to achieve such a property (see e.g.~\cite{javanmard2018debiasing}) in some settings. For the sake of completeness and as a gentle introduction to more challenging settings, we introduce here a slightly different estimator.

As a first step, we rely on a square-root Lasso~\cite{squarerootLasso} estimator based on the first subsample. From the design matrix $\bX^{(1)}$, we build its column normalized modification $\bT^{(1)}$ by
$$\mathbf T^{(1)} = \Big(\bX^{(1)}_{.,1}/\|\bX^{(1)}_{.,1}\|_2, \ldots, \bX^{(1)}_{.,p}/\|\bX^{(1)}_{.,p}\|_2 \Big)\ .$$
Set $\lambda = 2\sqrt{\overline{\Phi}^{-1}(\delta/(4p))}$. The square-root Lasso estimator is then defined by 
\beq\label{eq:def_SL_1}
\widehat{\theta}_{SL,N} \in \arg \min \|Y^{(1)}-\bT^{(1)}\theta\|_2 + \lambda \|\theta\|_1 \ ; \quad (\widehat{\theta}_{SL})_i= (\widehat{\theta}_{SL,N})_i/\|\bX^{(1)}_{.,i}\|_2, \ i=1,\ldots, p\ . 
\eeq
In this section, we could replace the square-root Lasso estimator by a classical Lasso estimator since the noise level $\sigma$ is known. Also, the design is normalized for the purpose of simplifying some proof arguments, but the results remain valid (with slightly different constants) with the unnormalized design matrix $\bX^{(1)}$.

 Then, given $\widehat{\theta}_{SL}$, we use the second sample to improve the estimation  of $\theta^*$. The estimator $\widetilde{\theta}_{\bI}$ is based on the empirical raw correlations between the covariates and the residuals.
\beq
  \widetilde{\theta}_{\bI} = \frac{1}{m}\bX^{(2)T}\big(Y^{(2)}- \bX^{(2)}\widehat{\theta}_{SL} \big)+  \widehat{\theta}_{SL}\  . \label{eq:phi_proj}
  \eeq
  Since the design is independent, $\widetilde{\theta}_{\bI}$ is an unbiased estimator of $\theta^*$. It is not hard to show (see the proof of the next proposition) that, under weak assumptions, this estimator satisfies has $\|\widetilde{\theta}_{\bI}-\theta^*\|_{\infty}\lesssim c\sigma \sqrt{\log(p)/n}$ with high probability. 
  This is why we define the test $\phi^{(t)}
  $ rejecting the null if $\big|(\widetilde{\theta}_{\bI})_{(k_0+1)}\big|\geq \underline{c}^{(t)} \sigma\sqrt{\log(p/\alpha)/n}$, where a suitable value for the constant $\underline{c}^{(t)}$ is defined in the proof of Proposition \ref{prp:t} below. This test is powerful when $\theta^*$ contains at least $k_0+1$ large entries. This is formalized in the following proposition.

\begin{prp}\label{prp:t}
There exist numerical constants $\underline{c}^{(t)}$,  $c$ and $c'$  such that the following holds under Condition ($\bA[\alpha\wedge \beta\wedge \delta]$). The test $\phi^{(t)}$ satisfies \emph{({\bf P1}[$\alpha+\delta$])} and \emph{({\bf P2}[$\beta+\delta$])} on the collections
\beq\label{eq:cond_P_separation}
\bbB_0[k_0+\Delta] \bigcap \Big\{\theta^*,\quad    |\theta^*_{(k_0+1)}|\geq c \sigma \sqrt{\frac{1}{n}\log\big(\frac{p}{\alpha\wedge \beta}\big)}\Big\}\ ,
\eeq
with $1\leq \Delta \leq c'n/\log(p/\delta)$.
 \end{prp}

Again, we emphasize that similar performances are achieved by the debiased Lasso test of Javanmard and Lee~\cite{javanmard2017flexible}.

\subsubsection{Test $\phi^{(\chi)}$ based on the $l_2$ norm of the residuals}\label{ss:phi_chi}

The second test is also simple. We heavily rely on the knowledge of the noise level $\sigma$. In the detection setting $(k_0=0)$, \cite{2010_EJS_Ingster,2011_AS_Arias-Castro} consider a test rejecting the null when the squared norm $\|Y\|_2^2/(n\sigma^2)$ is large compared to one. Indeed, in expectation, $\|Y\|_2^2/(n\sigma^2)$ is equal to $\|\theta^*\|_2^2/\sigma^2+1$. Here, we have to adapt this statistic as $\|\theta^*\|_2^2$ is unknown under the null. 

First, we project the square-root Lasso estimator $\widehat{\theta}_{SL}$ onto the parameter set corresponding to the null hypothesis. More precisely, we define $\widetilde{\theta}_{SL,k_0}= \arg\min_{\theta\in \bbB_0[k_0]} \|\widehat{\theta}_{SL}-\theta\|_2^2$. In other words, $\widetilde{\theta}_{SL,k_0}$ is obtained from $\widehat{\theta}_{SL}$ by thresholding its  $(p-k_0)$ smallest entries to zero.  Then, given $\widetilde{\theta}_{SL,k_0}$, we use the second sample to assess whether $\theta^*$ is significantly different from $\widetilde{\theta}_{SL,k_0}$. Define the residuals vectors $\widehat{R}_{k_0}= Y^{(2)}- \bX^{(2)}\widetilde{\theta}_{SL,k_0}$ and, for $R\in \mathbb{R}^m$,   the statistic
$ Z_{\chi}[R] = \frac{\|R\|_2^2  }{m \sigma^2} -1$.

 Take the threshold $v_{\alpha,\delta}^{(\chi)}= \sqrt{\frac{\log(1/\alpha)}{m}}  + \frac{(k_0\vee 1)\log(p/\delta)}{m}$, we consider the test $\phi^{(\chi)}
 $ rejecting the null hypothesis when  $Z_{\chi}[\widehat{R}_{k_0}]>
 \underline{c}^{(\chi)} v_{\alpha,\delta}^{(\chi)}$, where the numerical constant $\underline{c}^{(\chi)}$ is introduced in the proof of the following proposition.

\begin{prp}\label{prp:chi}
There exist numerical constants $\underline{c}^{(\chi)}$ and $c$ and  such that the following holds under Condition ($\bA[\alpha\wedge \beta\wedge\delta]$). The test $\phi^{(\chi)}$ satisfies  \emph{({\bf P1}[$\alpha+\delta$])} and \emph{({\bf P2}[$\beta$])} on the collection
\beq\label{eq:condition_distance_chi}
\Big\{\theta^*,\quad d^2_2\big[\theta^*;\bbB_0[k_0]\big]\geq c \sigma^2  \Big[ \frac{k_0\vee 1}{n}\log(p/\delta) + \sqrt{\frac{\log(2/(\alpha\wedge \beta))}{n}}\Big]\Big\}\ .
\eeq
\end{prp}

It turns out that a combination of  $\phi^{(t)}$ and $\phi^{(\chi)}$ is matching the minimax lower bound of Theorem \ref{thm:lbkvkd} when $k_0\leq \sqrt{p}$. For larger null hypotheses, we need to rely on more intricate tests that are discussed in the next section.

\subsubsection{Test $\phi^{(f)}$ based on the empirical Fourier transform of the raw covariances}\label{ss:fourier}

In the Gaussian sequence framework ($p=n$ and $\bX=\bI_p$), \cite{carpentier2017adaptive} have recovered the optimal separation distance using test based on the empirical Fourier transform of the data. In this section, we adapt this approach in the linear regression model.

Conditionally to $Y$, it is shown in the proof of Proposition \ref{prp:analyse_Z_f} below that the normalized raw covariances $\bX^T Y /\|Y\|_2$ follow a normal distribution with mean $\theta^* \|Y\|_2/[\sigma^2+\|\theta^*\|_2^2]$ and variance $\bI_p - \theta^*\theta^{*T}/[\sigma^2+\|\theta^*\|_2^2]$. Since $\|Y\|_2^2$ is concentrated around $n[\sigma^2+\|\theta^*\|_2^2]$ and assuming that $\|\theta^*\|_2^2$ is small compared to $\sigma^2$, this implies that the raw covariances are almost distributed as a normal distribution with mean $\sqrt{n}\theta^*/\sigma$ and covariance $\bI_p$. This observation leads us to adapt the Fourier tests of \cite{carpentier2017adaptive} in our setting through raw covariances. 

\medskip 

The purpose of the empirical  Fourier transform statistic considered in \cite{carpentier2017adaptive} (but see also \cite{MR2420411,MR2325113} for previous work), is to approximate the discontinuous function $\sum_{i=1}^p \1_{\theta^*_i\neq 0}$. First, introduce, for $s>0$, the function 
\beq\label{phi_t}
\varphi(s;x)= \int_{-1}^{+1}(1-|\xi|) \cos\big(\xi s x\big)e^{ \xi^2 s^2/2}d\xi\ . 
\eeq
For $Z\sim \cN(a, 1)$, standard computations lead to $\E[\varphi(s;Z)]= 2 \frac{1-\cos(sa)}{(sa)^2}=:g(sa)$. In particular, the function $g$ takes values in $[0,1]$ with $g(0)=0$ and $\lim_{|a|\rightarrow \infty} g(sa)= 1$ (see \cite{carpentier2017adaptive}). In some way, $g(sa)$ is a smooth approximation to $\1_{a\neq 0}$. The larger $s$ is, the closer $g(sa)$ is to the indicator function. However, $\varphi(s; Z)$ exhibit a higher variance for large $s$.

\medskip

The conditional distribution of $\bX^T Y /\|Y\|_2$ is close to a normal distribution with mean $\sqrt{n}\theta^*/\sigma$ and variance-covariance matrix $\bI_p$, provided that $\|\theta^*\|_2^2$ is small compared to $\sigma^2$. Hence,  it would be tempting to use a statistic of the form $\sum_{i=1}^p\varphi(s; (\bX^T Y)_i/\|Y\|_2)$, which in expectation would be close to $\sum_{i=1}^p g(s\theta^*_i)$, which in turn would approximate $\|\theta^*\|_0$. Unfortunately, large coordinates $|\theta^*_{i}|$ may perturb the concentration of the statistic since the true conditional covariance of $\bX^T Y /\|Y\|_2$ is $\bI_p- \theta^*\theta^{*T}/[\sigma^2+\|\theta^*\|_2^2]$. To address this technical issue, we first correct $\theta^*$ by removing its large coefficients.

As in Subsection~\ref{sec:test_l_infty}, the first two samples are respectively dedicated to building the Lasso estimator $\widehat{\theta}_{SL}$ and the debiased  estimator $\widetilde{\theta}_{\bI}$. If $|[\widetilde{\theta}_{\bI}]_{(k_0+1)}|>\underline{c}^{(t)}\sigma \sqrt{\log(2p/\alpha)/n}$, then the test $\phi^{(f)}$ introduced below rejects the null hypothesis, otherwise we define $\overline{\theta}_{\bI}$ as 
\beq\label{eq:definition_theta_bar}
\overline{\theta}_{\bI,i} =\widetilde{\theta}_{\bI,i} \1_{\{|\widetilde{\theta}_{\bI,i}|>  \underline{c}^{(t)} \sigma \sqrt{\frac{\log(2p/\alpha)}{n}}\}}\ . 
\eeq
In Subsection~\ref{sec:test_l_infty}, we argued that, with high probability,  $\|\widetilde{\theta}_{\bI}-\theta^*\|_{\infty}\leq\underline{c}^{(t)}\sigma \sqrt{\log(2p/\alpha)/n}$. As a consequence, $\|\overline{\theta}_{\bI}-\theta^*\|_{\infty}\leq 2\underline{c}^{(t)} \sigma \sqrt{\frac{\log(2p/\alpha)}{n}}$  and the support of $\overline{\theta}_{\bI}$ is included in that of $\theta^*$

\medskip 

Finally, we use the third subsample to compute the corrected raw covariances  $W_j= {\bf X}^{(3)T}_j \ol{Y}^{(3)}$ with $\ol{Y}^{(3)}=  Y^{(3)}-\bX^{(3)}\overline{\theta}_{\bI}$  relative to the linear regression model with parameter $\theta^* - \overline{\theta}_{\bI}$.
Then, following the above heuristic explanation,  we consider the statistic 
\[
Z_f : = \sum_{j=1}^p \1_{(\overline{\theta}_{\bI})_j=0}\varphi(s;\frac{W_j}{\|\ol{Y}^{(3)}\|_2})+\1_{(\overline{\theta}_{\bI})_j\neq 0}\ , 
\]
with tuning parameter $s = \sqrt{\log(e\frac{k_0}{\sqrt{p}})}\lor 1$. For $j$ in the support of $\overline{\theta}_{\bI}$, we are already confident that $\theta^*_j$ is non zero and we do not have to rely on $\varphi$. Finally, the test $\phi^{(f)}
$ rejects the null when $Z_f\geq k_0 + v^{(f)}_{\alpha}$ with $v^{(f)}_{\alpha}= s^2/5 + se^{s^2/2}\sqrt{2p\log(2/\alpha)}$. 

In comparison to the original statistic of \cite{carpentier2017adaptive} for the Gaussian sequence model, we use here a slightly smaller tuning parameter $s$ and the threshold  $v^{(f)}_{\alpha}$ has an additional term $s^2/5$.

\begin{prp}\label{prp:analyse_Z_f}
There exist constants $c$, $c_{\alpha}$, $c'_{\alpha}$, and  $c''_{\alpha}$ such that the following holds under Condition ($\bA[\alpha\wedge\delta ]$).  The test $\phi^{(f)}$ satisfies \emph{({\bf P1}$[\alpha+\delta]$)} and \emph{({\bf P2}[$\alpha+\delta+ e^{-n/27}$])} on the collection of parameters $\theta^*$ satisfying
$\|\theta^*\|_{0}\leq c n/\log(p/\delta)$,  $d^2_2\big[\theta^*;\bbB_0[k_0]\big] \leq   \sigma^2$ and at least one of the two following conditions.
\begin{eqnarray}
 |\theta^*_{(k_0+q)}| &\geq&  c_{\alpha}\sigma \sqrt{\frac{k_0}{qn\log(1+k_0/\sqrt{p})}}\ ,\, \text{ for some } q\geq  c'_{\alpha}\Big( (k_0^{3/4} p^{1/8})\lor \sqrt{p}\Big)
\ ; \label{eq:separation_log}
\\
 \label{eq:separation_log_l2}
 \sum_{i=1}^p \Big[\theta_i^{*2}&\wedge&  \frac{\sigma^2}{n\log(1+k_0/\sqrt{p})}\Big]\geq c''_{\alpha}\sigma^2 \frac{k_0}{n\log(1+k_0/\sqrt{p})}\ .
\end{eqnarray}

\end{prp}

 The test $\phi^{(f)}$ rejects the null hypothesis when there are many small non-zero coefficients in $\theta^*$. In particular, if $\theta^*$ contains $2k_0>2\sqrt{p}$ coefficients of order $\sigma  (n\log(p))^{-1/2}$, then the null hypothesis is rejected with high probability. Note that $\sigma (n\log(p))^{-1/2}$ is much smaller than the value needed to recover the position of these non-zero coefficients, which is of the order $\sigma \sqrt{\log(p)/n}$.  This behavior is reminiscent of the minimax lower bound in Theorem \ref{thm:lbkvkd}, where the squared separation distance is proven to be at least of the order $k_0/[n\log(p)]$ for $\Delta\geq k_0\geq \sqrt{p}$.

  When there are a few entries in $\theta^*$ that are neither large nor small - see below for more precisions, it turns out that $\phi^{(f)}$  only matches the minimax lower bound up to some $\log\log(p)$ multiplicative factor. To address this issue we need to introduce an additional test $\phi^{(i)}$.

\subsubsection{Intermediary regime: Test $\phi^{(i)}$ based on the empirical Fourier transform of the raw covariance}\label{ss:fourier_intermediary}

In this subsection, we focus on entries $\theta^*_i$ that are neither large (with respect to $\sigma \sqrt{\log(p)/n}$) as in the analysis $\phi^{(t)}$ nor small (with respect to $\sigma \sqrt{1/(n\log(p))}$) as in the analysis of $\phi^{(f)}$. This setting turns out to be  relevant for large $k_0$ only and we assume henceforth that $k_0\geq 2^{11}\sqrt{p}$. As in the previous section, 
we adapt a test from~\cite{carpentier2017adaptive} in the Gaussian sequence setting by applying the empirical Fourier transform to the raw covariances. 

Given two tuning parameters $r$ and $l$, define the function 
\beq\label{eq:def_eta}
\eta_{r,w}(x)= \frac{r}{(1-2\overline{\Phi}(r))}\int_{-1}^1 \frac{e^{-r^2\xi^2/2}}{\sqrt{2\pi}} e^{\xi^2 w^2/2 }\cos(\xi  w x)d\xi\ .
\eeq
and the statistic
\[V(r,w)= \sum_{j=1}^p \1_{(\overline{\theta}_I)_j= 0} \big[1 - \eta_{r,w}(W_j/\|\ol{Y}^{(3)} \|_2)\big]+ \1_{(\overline{\theta}_I)_j\neq 0}\ .\]

In order to get a grasp of this statistic let us consider the expectation of $\eta_{r,w}(X)$ for $X\sim \cN(x,1)$. Simple computations (see \cite{carpentier2017adaptive}) lead to $
\E[1- \eta_{r,w} (X)]=  1- \frac{1}{1-2\overline{\Phi}(r)}\int_{-r}^{r} \phi(\xi)\cos(\xi x \frac{w}{r} )d\xi$, which for large $r$, is close to $1- \exp(-x^2\tfrac{w^2}{2r^2})$. Thus, in contrast to the population function $g$ introduced in the previous subsection, which converges to $1$ at a quadratic rate, this function converges to one at an exponential rate, thereby better handling  moderate values of $\theta^*_i$. The downside of using this statistic is that $\E[1- \eta_{r,w} (X)]$ does not lie in $[0,1]$.

The test $\phi^{(i)}$ is an aggregation of multiple tests based on the statistics $V(r,w)$ for different tuning parameters $r$ and $w$. Define
 $l_0= \lceil k_0^{4/5}p^{1/10}\rceil$ and the dyadic collection $\cL_{0}=\{l_0, 2l_0,4l_0, \ldots, l_{\max}\}$ where $l_{\max}=  2^{\lfloor \log_2 (k_0/l_0)\rfloor} l_0/4 \leq k_0/4$. Note that $\cL_{0}$ is not empty if $k_0\geq 2^{11} \sqrt{p}$ and $p$ is large enough.
Given any $l\in \cL_{0}$, define 
\beq \label{eq:param}
 r_{l} =  \sqrt{2\log(\tfrac{k_0}{l})} \ ; \quad \quad w_l = \sqrt{\log(\tfrac{l}{\sqrt{p}})}\ .
\eeq
 Then, the test $\phi^{(i)}$ rejects the null hypothesis if, for some $l \in \cL_{0}$, 
\beq\label{eq:rejection_intermediary}
  V(r_{l},w_l) \geq k_0+ l+ v^{i}_{\alpha,l} \, \quad \quad\text{ where } \quad v^{i}_{\alpha,l}=\frac{e^{1/2}}{2}\omega_l^2+ \sqrt{2 ln^{1/2}\log\Big(\frac{\pi^2 [1+\log_2(l/l_0)]^2}{6\alpha}\Big)}\ .
\eeq

In comparison to the test in~\cite{carpentier2017adaptive}, the collection of tuning parameters $\cL_0$ is slightly narrower and the threshold $v^{i}_{\alpha,l}$ has an additional corrective term of the  order of $\omega_l^2$.

\begin{prp}\label{prp:analyse_Z_i}
There exist positive constants   $c, c_{\alpha}, c_\alpha'$ such that the following holds under Condition ($\bA[\alpha\wedge\delta ]$).  The test $\phi^{(i)}$ satisfies \emph{({\bf P1}$[\alpha+\delta]$)} and \emph{({\bf P2}[$\alpha+\delta+ e^{-n/27}$])} on the collection of parameters $\theta^*$ satisfying
$\|\theta^*\|_{0}\leq c n/\log(p/\delta)$,  $d^2_2\big[\theta^*;\bbB_0[k_0]\big] \leq   \sigma^2$ and 
\beq
 |\theta^*_{(k_0+q)}| \geq  c_{\alpha}\sigma \frac{1+\log\left(\frac{k_0}{q\wedge k_0}\right)}{\sqrt{n \log\left(1+\frac{k_0}{\sqrt{p}}\right)}}\ ,\, \text{ for some } q\geq  c'_{\alpha} k_0^{4/5} p^{1/10}
\ . \label{eq:separation_log_intermediairy}
\eeq
\end{prp}

In Comparison to Condition~\eqref{eq:separation_log} for Proposition \ref{prp:analyse_Z_f}, $|\theta^*_{(k_0+q)}|$ is possibly much smaller than for $\phi^{(f)}$ in the regime where $k_{0}^{4/5}p^{1/10}\lesssim_\alpha q \lesssim_\alpha \sqrt{p}$.

\subsubsection{Aggregated test}

To conclude this section, we evaluate the performances of the combination of all the previous tests. In fact, $\phi^{(i)}$ is only defined in the large $k_0$ regime. We take the convention that $\phi^{(i)}$ is a trivial test that always accepts the null hypothesis in the small $k_0$ regime. Consider the aggregated  test 
\[
\phi^{(ag)}= \max(\phi^{(t)},\phi^{(\chi)}, \phi^{(f)},\phi^{(i)},\1\{d_2^2(\widehat{\theta}_{SL},\bbB_0[k_0])\geq \sigma^2/2\})\ .
\]
The last test  $\1\{d_2^2(\widehat{\theta}_{SL},\bbB_0[k_0])\geq \sigma^2/2\}$ is introduced for technical purpose to handle very dense alternatives 
($\|\theta^*\|_0\geq c n/\log(p/\delta)$).

\begin{thm}\label{thm:ubkvkd}
Let $\delta\in (0,1)$ and $\varsigma\in (0,1)$. There exists positive constants  $c_{\varsigma}$ and $c_{\varsigma,\delta}$ such that the following holds. 
Assume  that $p \geq c_\varsigma$  and that Condition ($\bA[\delta \wedge \delta]$) is satisfied. 
Define 
\beq\label{eq:upper_lbkvkd}
\rho^2_{k_0,\Delta,\varsigma}
= \left\{
\begin{array}{ccc}
\min\Big[ \frac{\Delta}{n} \log(p),\frac{1}{\sqrt{n}} + \frac{k_0}{n}\log(p) \Big]\ & \text{ if } & 0\leq k_0\leq p^{1/2-\varsigma}
\ ;\\
\min[\frac{\Delta\log(p)}{n},\frac{k_0}{n\log(p) }\big]& \text{ if } &  k_0 > p^{1/2+\varsigma} \ .\\
\end{array}\right.
\eeq
The test $\phi^{(ag)}$ satisfies \emph{({\bf P1}$[\delta+ 4\alpha]$)} and \emph{({\bf P2}[$\delta+\alpha+ e^{-n/27}$])} on the collection of parameters 
\[
\bbB_0[k_0+\Delta]\cap \big\{ \theta^*, \ d^2_2\big[\theta^*;\bbB_0[k_0]\big]\geq c_{\delta,\varsigma} \rho^2_{k_0,\Delta,\varsigma}\big\}
\]
with $1\leq \Delta \leq p-k_0$.
\end{thm}
The case $k_0+\Delta \leq c n/\log(p/\delta)$ is a simple corollary of the previous results, whereas the dense case $k_0+\Delta > c n/\log(p/\delta)$ requires further work.

To further compare this result with the minimax lower bound of Theorem \ref{thm:lbkvkd}, we assume that $p\geq n^{1+\zeta}$ for some $\zeta>0$. Recall that we also suppose $k_0\leq c n/\log(p)$. From Theorems \ref{thm:lbkvkd} and \ref{thm:ubkvkd}, we deduce that 
\medskip

\noindent 
 {\bf Case 1}: $k_0\leq p^{1/2 - \kappa}$ with an arbitrary    $\kappa\in (0,1/2)$. 
 \[
  \rho_{\gamma}^{*2}[k_0,\Delta]\asymp_{\gamma,\kappa,\varsigma}\left\{
  \begin{array}{cc}
  \frac{\Delta}{n}\log(p) & \text{ if } \Delta \leq \frac{\sqrt{n}}{\log(p)}+ k_0\ ; \\
 \frac{1}{\sqrt{n}}+  \frac{k_0\log(p)}{n} & \text{ if } \Delta > \frac{\sqrt{n}}{\log(p)}+ k_0\ .
                                                                    \end{array}
  \right.
 \]

 \noindent 
{\bf Case 2}: $k_0\geq p^{1/2 + \kappa}$ with  an arbitrary $\kappa\in (0,1/2)$. For any $\varsigma \in (0,1/2)$ arbitrarily small, we have
\[\rho_{\gamma}^{*2}[k_0,\Delta]\asymp_{\gamma,\kappa,\zeta,\varsigma} \left\{\begin{array}{cc}
  \frac{\Delta}{n}\log(p) & \text{ if } \Delta < k_0p^{ -\varsigma} \ ;\\
  \frac{k_0}{n\log(p)} & \text{ if } \Delta \geq  k_0\ , 
                                                                    \end{array}\right.
  \  \]
and that all these bounds are simultaneously achieved by the test $\phi^{(ag)}$. As a consequence, $\phi^{(ag)}$ is simultaneous minimax over all $k_0$ and all $\Delta$ except in the regimes when $k_0$ is close to $\sqrt{p}$ or when $\Delta$ is close to $k_0$, in which case, there is possibly a polylogarithmic  difference between the minimax lower and upper bounds.

\begin{proof}
If $k_0\leq p^{1/2-\kappa}$ and $p\geq n^{1+\zeta}$, then  $\log(1+\sqrt{p}/k_0)\asymp _{\kappa}\log(p)$. Hence \eqref{eq:lower_lbkvkd} in Theorem~\ref{thm:lbkvkd} ensures that the square minimax  separation distance is at least  of the order of $\min(\frac{\Delta}{n}\log(1+\frac{\sqrt{p}}{\Delta}), \frac{1}{\sqrt{n}}+  \frac{k_0\log(p)}{n})$. The first term is (up to numerical constants) larger than second one when $\Delta\geq k_0\vee \sqrt{n}$. For $\Delta \leq k_0\vee \sqrt{n}$, we have  $\log(\sqrt{p}/\Delta)\asymp_{\kappa,\zeta}\log(p)$. Hence, the square minimax separation distance 
is at least of the order of  $\min(\frac{\Delta}{n}\log(p), \frac{1}{\sqrt{n}}+  \frac{k_0\log(p)}{n})$
which matches the upper bound of Theorem~\ref{thm:ubkvkd}. If $k_0\geq p^{1/2+\kappa}$, $p\geq n^{1+\zeta}$, and $\Delta\leq k_0p^{-\varsigma}$, then $\log(1+\sqrt{k_0/\Delta})\asymp_{\kappa,\zeta,\varsigma}\log(p)$ and Theorem~\ref{eq:lower_lbkvkd} ensures that the square minimax  separation distance is at least of the order of $\Delta\log(p)/n$, matching again Theorem~\ref{thm:ubkvkd}. When $\Delta\geq k_0$, $\Delta\log^2(1+\sqrt{k_0/\Delta})\geq c k_0$, and we deduce from Theorems~\ref{thm:lbkvkd} and \ref{thm:ubkvkd} that the square minimax separation distance is of order of $k_0/[n\log(p)]$. 
\end{proof}

\medskip

\noindent Let us summarize the different regimes
\begin{itemize}
\item If $\Delta$ is small - first result in Cases 1 and 2 - then the squared  minimax separation distance ($\Delta\log(p)/n$) is the same as for signal detection ($k_0=0$). The upper bound can be achieved using any  $\sqrt{\log(p)/n}$ $l_\infty$-consistent estimator of $\theta^*$ and simply counting the number of its large entries. In the independent setting such estimator is easily built using the raw correlation ($\widetilde{\theta}_{\bI}$) between the variables and the response. Alternatively, one could use the debiased Lasso~\cite{javanmard2014confidence, zhang2014confidence,van2014asymptotically,javanmard2018debiasing} which is valid for a wider class of $\bSigma$.

\item If $\Delta$ is large and $k_0$ is small - second result in Case 1 - then the squared minimax separation distance can be understood as the sum of the quantity $n^{-1/2}$ arising in signal detection and the complexity $k_0\log(p)/n$ of the null hypothesis. The matching upper bound is achieved by computing the $l_2$ norm of the residuals when plugging a suitable estimator of $\theta^*$. The upper bound was already obtained in~\cite{nickl_vandegeer} (for a computationally inefficient method) but the matching minimax lower bound is new.

\item Finally, if $\Delta$ is large and $k_0$ is large - second result in Case 2 - then the minimax separation distance is highly non standard and depends on the complexity of the null hypothesis. Both the lower and upper bound are new. In some way, they  both draw inspiration from the analysis~\cite{carpentier2017adaptive} of the same problem in the Gaussian sequence framework. 
\end{itemize}

In this paper, we focused on recovering the minimax separation distance in the the high-dimensional setting, namely we require $p \geq n^{1+\zeta}$, where $\zeta>0$ is an arbitrarily small universal constant.
Aside from this restriction, there are two gaps in our analysis: 
\begin{itemize}
\item Fist, when $p\geq c_2 n^2$ and $k_0\in (n/\log(p), n)$, our minimax lower bounds in Theorem~\ref{thm:lbkvkd} imply that the testing problem is almost impossible. However, for $p\leq c_2n^2$, we did not manage to prove similar lower bounds. We conjecture that, for $p\leq c_2 n^2$ and $k_0\in (n/\log(p), n)$, $\rho^*_{\gamma}[k_0,p]$ is huge, but we did not manage to prove it. 
\item Some poly-log terms mismatch between the upper and lower bounds arise when  $k_0$ is close to $\sqrt{p}$ - e.g.~$k_0\in [\sqrt{p} \log^{-\zeta}(p), \sqrt{p} \log^{\zeta}(p)]$ for some $\zeta>0$ and when $\Delta$ gets close to $k_0$ from below - i.e.~$k_0 p^{-\zeta} \leq \Delta \leq k_0$ for some arbitrarily small universal constant $\zeta>0$. In that regime, we could improve our upper bounds by adapting some higher-criticism~\cite{jin2004} procedures as it was done in the sequence model~\cite{carpentier2017adaptive}. However, even with this new procedure this would not completely close the gap. We conjecture that our minimax bound~\eqref{eq:lower_lbkvkd} is not completely sharp in that regime (see its proof for a tentative explanation).

\end{itemize}

\section{General Setting}\label{sec:general}

 In this section, we focus on the general setting where $\bSigma$ is unknown and is only assumed to belong to some class $\cU(\eta)$~\eqref{eq:definition_Ueta} for some $\eta>1$. The noise variance $\sigma^2$ is also assumed to be unknown.

\subsection{Minimax lower bound}

 Obviously, $\boldsymbol{\rho}^{*}_{g,\gamma}[k_0,\Delta]$ is at least as large as $\rho^{*}_{\gamma}[k_0,\Delta]$ since the covariance matrix $\bSigma$ is unknown and $\bI_p$ belongs to $\cU[\eta]$. Therefore, Theorem~\ref{thm:lbkvkd} in the previous section provides a lower bound on $\boldsymbol{\rho}^{*}_{g,\gamma}[k_0,\Delta]$. It turns out that  that this lower bound is sometimes loose and that the general setting is actually more challenging in some regimes as shown by  the following proposition. 

\begin{prp}\label{prp:lbuvkd}
Assume that $p\geq 2n$. 
There exist positive numerical constants $c>0$ and $c'>0$ such that for $p\geq c_3$ and for all $\gamma\leq 0.06$, one has 
\beq\label{eq:lower_lbuvkd}
\boldsymbol{\rho}_{g,\gamma}^{*2}[k_0,\Delta] \geq c \left\{
\begin{array}{ccc}
 \frac{\Delta}{n}\log\left(1+\frac{\sqrt{p}}{\Delta}\right)\exp\left[c'\frac{\Delta}{n}\log\left(1+\frac{\sqrt{p}}{\Delta}\right)  \right] &\text{ if }& 0\leq k_0  \leq \sqrt{p}\wedge n\ ;
 \\
\frac{\Delta}{n } \frac{\log^2\big[1+ \sqrt{\frac{k_0}{\Delta}}\big]}{\log(p)}& \text{ if } & \sqrt{p} \leq k_0\leq n\ ,
\end{array}
\right.
\eeq
with $1\leq \Delta\leq p-k_0$.
\end{prp}
In fact, this result is a combination of  Theorem~\ref{thm:lbkvkd} together with known minimax lower bounds for the detection problem ($k_0=0$) with unknown variance \cite{2010_AS_Verzelen,verzelen_minimax}. 

In comparison to the independent setting, one cannot achieve anymore the rate $1/\sqrt{n}+ k_0\log(1+ \sqrt{p}/k_0)/n$. Most importantly, the testing problem becomes almost impossible for dense alternative ($\Delta \gtrsim n/\log(p)$) in the high-dimensional regime $p\geq n^2$.

 \subsection{Testing procedures}

We cannot rely anymore on the test $\phi^{(\chi)}$ as the noise level is unknown nor on $\phi^{(t)}$ and $\phi^{(f)}$ as their reconstruction relies on the independence of the covariates.

As in the previous sections we introduce two properties {({\bf gP1}}) and {({\bf gP2}}) characterizing the type I and II error probabilities in this setting where the noise level $\sigma$ and the covariance matrix $\bSigma$ are unknown. 

\noindent 
{\bf Property  gP1}. A test $\phi$ satisfies {(\bf gP1}[$\alpha$]) if its type I error probability is less or equal to $\alpha$, that is 
$
 \sup_{\theta^*\in \bbB_0[k_0]}\sup_{\sigma>0}\sup_{\bSigma\in \cU[\eta]}\P_{\theta^*,\sigma,\bSigma}[\phi=1]\leq \alpha$. 

 \medskip 

\noindent 
{\bf Property  gP2}. A test $\phi$ satisfies ({\bf gP2}[$\beta$]) on the collection  $\Theta$  of parameters if its type II error probability is uniformly less or equal to $\beta$ that is 
$ \inf_{\sigma>0,\bSigma\in \cU[\eta]}\inf_{\theta^* \in \Theta}\P_{\sigma \theta^*,\sigma,\bSigma}[\phi=1]\geq 1-\beta$. 

\medskip

\noindent 
Note that that in the above bound $\theta^*$ is rescaled by $\sigma$ for homogeneity purpose. As in Section \ref{sec:independent_design}, we restrict our attention to sparsities $k_0$ that are less than $n/\log(p)$. The numerical constants $\underline{c}^{({\bf B})}$ and  $\underline{c}^{({\bf B})'}$ in the following condition are introduced in the proof of Proposition \ref{prp:U} and Corollary \ref{cor:test_MCP}.

\noindent 
([$\bB[\alpha]$)\centerline{$(k_0\vee 1) \big[1+\log(p/\alpha)\big]+ \log^3\big(\frac{1}{\alpha}\big)+ \log(p)\log(\frac{1}{\alpha}) \leq \underline{c}_{\eta}^{(\bf B)} n$\,\text{ and }\, $p\geq \underline{c}^{(\bB)'}.$\quad}

\bigskip 

In this section, we divide the sample in two subsamples $(Y^{(0)},\bX^{(0)})$ and  $(Y^{(1)},\bX^{(1)})$ of equal size $m=n/2$. As previously, we shall combine several tests to match the minimax lower bounds.

\subsubsection{Test $\phi^{(u)}$ based on a $U$-statistic.} \label{sec:U-stat}

The first test is specific to the moderate regime $p\leq n^2$. For known $\sigma$, we introduced in the previous section a statistic relying on the observation that $\|Y\|_2^2/n-\sigma^2$ estimates well $\|\theta^*\|_2^2$. Then, relying on a good $k_0$-sparse estimator $\widetilde{\theta}_{SL,k_0}$ of $\theta^*$ and computing the square norm of the residuals, we estimate $\|\theta^*-\widetilde{\theta}_{SL,k_0}\|_2^2$, which under the null, should be small. Here, we follow the same strategy by considering an estimator of the signal strength, still valid for unknown $\sigma$.

\medskip 
In~\cite{dicker_variance}, Dicker tackled the problem of estimating the signal strength $\|\theta^*\|_2^2$ in the setting where $\bSigma=\bI_p$ and  $\sigma^2$ is unknown. This led him to introduce
the $U$-statistic $N= \frac{1}{n^2}[Y^T [\bX \bX^{T}-\frac{1}{n}\mathrm{tr}[\bX\bX^{T}]\bI_{n}]Y]$, which is unbiased and $\sqrt{p}/n$ consistent. For general $\bSigma$, this statistic has later been shown to be  concentrated around the quadratic form $\theta^{*T} \bSigma^2\theta^*$ (see \cite[Sect.2.1]{verzelen2016adaptive}). 
As a consequence, one can rely on it to test the nullity of $\theta^*$.

For composite null hypotheses, we use $(Y^{(1)},\bX^{(1)})$ to build $\widetilde{\theta}_{SL,k_0}$ as in Subsection \ref{ss:phi_chi} and then compute the residuals $\widehat{R}_{SL}$ with respect to the the second sample, $\widehat{R}_{SL}= Y^{(0)}-\bX^{(0)}\widetilde{\theta}_{SL,k_0}$. Finally, we define the normalized $U$-statistic $Z^{(u)}$ by 
\beq 
  Z^{(u)} = \frac{\widehat{R}_{SL}^T \Big[\bX^{(0)}\bX^{(0)T}-\frac{1}{m}\mathrm{tr}[\bX^{(0)}\bX^{(0)T}]\bI_{m}\Big] \widehat{R}_{SL}}{\|\widehat{R}_{SL}\|_2^2(m+1)}\ .\label{eq:phi_moment}
  \eeq
Conditionally to $\widetilde{\theta}_{SL,k_0}$, $\widehat{R}_{SL}$ is the response  of a linear regression model with parameter $(\theta^*- \widetilde{\theta}_{SL,k_0})$, variance $\sigma^2\bI_m$, and random design $\cN(0,\bSigma)$ . Hence, the second moment of each entry of $\widehat{R}_{SL}$ equals $\sigma^2+ (\theta^* - \widetilde{\theta}_{SL,k_0})^T \bSigma (\theta^* - \widetilde{\theta}_{SL,k_0})$ and $\|\widehat{R}_{SL}\|_2^2/m$ is therefore close to $\sigma^2+ (\theta^* - \widetilde{\theta}_{SL,k_0})^T \bSigma (\theta^* - \widetilde{\theta}_{SL,k_0})$. Intuitively, the statistic  $Z^{(u)}$ is therefore expected to be close to 
\[
\frac{(\theta^* - \widetilde{\theta}_{SL,k_0})^T \bSigma^2 (\theta^* - \widetilde{\theta}_{SL,k_0})}{\sigma^2 + (\theta^* - \widetilde{\theta}_{SL,k_0})^T \bSigma (\theta^* - \widetilde{\theta}_{SL,k_0})}\ ,
\]
so that a large value for $Z^{(u)}$ suggests that $\theta^*$ is significantly different from a $k_0$ sparse vector. 
Setting the threshold 
\beq\label{eq:u_M}
v^{(u)}_{\alpha}=  \frac{(k_0\vee 1)\log(p/\delta)}{m}+  \frac{\sqrt{p\log(2/\alpha)}}{m}\ ,
\eeq
we consider  the test $\phi^{(u)}$ rejecting the null hypothesis when $Z^{(u)}>\underline{c}^{(u)}_{\eta} v^{(u)}_{\alpha}$.

\begin{prp}\label{prp:U}
 There exist three constants $\underline{c}^{(u)}_{\eta}$, $c_{\eta}$ and $c'_{\eta}$ such that the following holds under Condition 
\emph{({\bf B}($\alpha\wedge \beta\wedge \delta$)} and if $ 2n\leq p\leq c_{\eta} n^2\log^{-1}\big(\frac{2}{\alpha \wedge \beta}\big)$.
The test $\phi^{(u)}$ satisfies \emph{({\bf gP1}[$\alpha+\delta$])}
and \emph{({\bf gP2}[$\beta$])} on  the collection 
\beq \label{eq:cond_M_separation}
  \Big\{\theta^*\ , \quad d^2_2\big[\theta^*;\bbB_0(k_0)\big]\geq c'_{\eta} \Big[\frac{\sqrt{p\log(2/(\alpha\wedge \beta))}}{n}+\frac{(k_0\vee 1)}{n}\log\big(\frac{p}{\delta}\big)  \Big]\Big\}\ .
\eeq
\end{prp}

\subsubsection{Recovering the $\Delta\log(p)/n$ rate with variable selection}

To achieve the $\Delta\log(p)/n$ rate, it would suffice to estimate $\theta^*$ at the $l_{\infty}$ rate $\sigma\sqrt{\log(p)/n}$ as we did for the test $\phi^{(t)}$ in the previous section. However, we are unaware of any estimator achieving this rate uniformly over the class $\cU(\eta)$ of covariance matrices $\bSigma$. For $k_0\geq \sqrt{n}$, it is even proved that no such estimator exists~\cite{cai2017confidence}.

Here, we adopt another strategy. We shall first estimate the support $\cS(\theta^*)$ of $\theta^*$ and count the number of large entries of the least-squares estimator of $\theta^*$ restricted to the estimated support $\widehat{S}$. Of course, if $\widehat{S}= \cS(\theta^*)$ with high probability, then the restricted least-squares estimator $\widehat{\theta}_{\widehat{S}}$ (see below for a definition) will be close to $\theta^*$ in $l_{\infty}$ norm. Unfortunately, it is impossible for an estimator $\widehat{S}$ to estimate exactly the support $\cS(\theta^*)$, especially when $\theta^*$ contains arbitrarily small coordinates. 

This is why we shall require that the estimator 
$\widehat{S}$ satisfies a weaker property.  Given $a>0$, let $M(\ba_1,\frac{\theta^*}{\sigma})= |\{i, 0< \frac{|\theta^*_i|}{\sigma}\leq \ba_1 \sqrt{\log(p)/m}\}|$ be the number of small but non zero coefficients of $\theta^*$. Below $\ba_1$, $\ba_2$, $\ba_3$ refer to three positive quantities. Recall that $\overline{S}$ is the complement of $S$.

 \medskip 
 
 \noindent 
{\bf Property  ($\bS[\ba_1,\ba_2,\ba_3]$)}. A (possibly random) set $S$ is said to satisfy this property if  
\beq\label{eq:upper_CS}
|S| \leq \ba_2\|\theta^*\|_0\ ;\quad \quad 
\|\theta_{\overline{S}}^*\|_2^2 \leq \ba_3^2\sigma^2 M(\ba_1,\frac{\theta^*}{\sigma})\frac{ \log(p)}{ m}\ .
 \eeq

In other words, the cardinal of $S$ is not too large compared to the sparsity of $\theta^*$ and the square norm of $\theta^*$ outside $S$ is at most as large as that of the small entries of $\theta^*$. Observe that the large entries of $\theta^*$ are not required to belong to $S$.

Then, given a set $S$, we consider the restricted least-square estimator and the plug-in variance estimators 
\beq\label{eq:definition_restricted_least_squares}
\widehat{\theta}_{ls,S} = \argmin_{\theta \ : \, \cS(\theta)\subset S} \| Y^{(0)} - \bX^{(0)}\theta\|_2^2\ ;\quad \quad  \widehat{\sigma}_S^2= \frac{1}{m}\| Y^{(0)} - \bX^{(0)}\widehat{\theta}_{ls,S}\|_2^2\ . 
\eeq

For a vector $u\in \mathbb{R}^p$ and $c>0$, $N[c;u]= |\{i: |u_i|\geq c\sqrt{\log(p)/m}\}|$  is the number of entries of $u$ larger or equal (in absolute value) than $c\sqrt{\log(p)/m}$. Then, we define the  test 
 $\phi^{(th)}[S; c]$ rejecting the null if and only if $N[\underline{c};\widehat{\theta}_{ls,S}/\widehat{\sigma}_{S} ]\geq k_0+1$, which means  that  $\widehat{\theta}_{ls,S}$ contains at least  $k_0+1$ large entries.

\begin{thm}\label{thm:puissance_thres}
There exist constants $c$ and $c'_{\eta}$  such  that the following holds for any $p\geq 3$. 
Consider any $\theta^*$, $\sigma$, $\bSigma\in \cU[\eta]$,  $\delta\in (0,1)$, $(\ba_1,\ba_2, \ba_3)>0$ satisfying
$c\big[\ba_2\|\theta^*\|_0+ \log\left(\frac{4}{\delta}\right)\big]\leq m$, and  $S$ satisfying  $\bS[\ba_1,\ba_2,\ba_3]$. Taking 
\beq\label{eq:definition_c_*}
\underline{c}_*\geq  \sqrt{2}\ba_1+ 11\eta^2(\ba_3\vee 1)\sqrt{\log\left(\frac{4e}{\delta}\right)}\ . 
\eeq
 we have $\mathbb{P}^{(0)}_{\theta^*,\sigma,\bSigma}\big[\phi^{(th)}[S,\underline{c}_*]=1\big]\leq \delta$ if $\|\theta^*\|_0\leq k_0$.
Besides, $\mathbb{P}^{(0)}_{\theta^*,\sigma,\bSigma}\big[\phi^{(th)}[S,\underline{c}_*]=1\big]\geq 1- \delta$ if $\|\theta^*\|_0> k_0$ and  \[
d^2_2(\theta_*;\bbB_0[k_0])\geq c_{\eta}\big[1+ \ba_3^2 \|\theta^*\|_0\frac{\log(p)}{m}\big]\left[\ba_3^2\sqrt{\log(4/\delta)}+ \underline{c}_*^2\right]\sigma^2 \frac{[\|\theta^*\|_0-k_0]\log(p)}{m}\ .\]

\end{thm}

If $S$ satisfies ($\bS[\ba_1,\ba_2,\ba_3]$), then the test $\phi^{(th)}[S; c]$ with a suitable tuning parameter $c$ has a controlled type I error probability. Besides, its square separation distance over $\bbB_0[k_0+\Delta]$ is (up to constants depending on $\ba_1$ and $\ba_3$) of the order of $\Delta\log(p)/n$.

In view of this general result, it suffices to build an estimator $\widehat{S}$ of the support based on $(Y^{(1)},\bX^{(1)})$ that satisfies ($\bS[\ba_1,\ba_2,\ba_3]$) for small $\ba_1,\ba_2,\ba_3$ to get the squared separation distance $\Delta\log(p)/n$.

\bigskip

Unfortunately, the support of the Lasso estimator is only proved to satisfy the first part of property ($\bS[\ba_1,\ba_2,\ba_3]$). Its number of false positives is at most of the order of $\|\theta^*\|_0$, see~\cite{zhang2008sparsity}. It turns out that the second part of the property has only recently been proved to be achieved by  non-convex penalized estimators, see~\cite{feng2017sorted} such as MCP estimator~\cite{zhang2010nearly}.

As in the previous section, we consider the column normalized design $\bT^{(1)}$. Given $\beta\in \mathbb{R}^p$ and two tuning parameters $b>0$, $\lambda>0$, the MCP criterion is defined by 
\beq\label{eq:critere_MCP}
 L(\theta):= \|Y^{(1)}-\bT^{(1)}\theta\|_2^2+ \sum_{i=1}^p\rho(|\theta_i|; \lambda)\ ; \quad \quad \rho(t; \lambda)= \lambda\int_{0}^t (1- x/(\kappa \lambda))_+ dx \ ,
\eeq
where $x_+=\max(x,0)$.  Local minimizers of the MCP criterion can be efficiently computed using the PLUS Algorithm from~\cite{zhang2010nearly} or by approximate regularization path by~\cite{wang2014optimal}. It turns out that non-convex penalized estimators suffer from less bias than Lasso estimators.

Consider the square-root Lasso estimator~\eqref{eq:def_SL_1} $\widehat{\theta}_{SL}$ with $\delta=1/p$ and the plug-in variance estimator $\widehat{\sigma}_{SL}=\|Y-\bX\widehat{\theta}_{SL}\|_2/\sqrt{m}$. Define the tuning parameters
\[
 \widehat{\lambda}_{MCP}= \underline{c}^{(MCP)}_{\eta} \widehat{\sigma}_{SL}\sqrt{\log(p)}\ ;\quad\quad \kappa= \underline{c}^{'(MCP)}_{\eta}
\]
for some constants $\underline{c}^{(MCP)}_{\eta}$ and  $\underline{c}^{'(MCP)}_{\eta}$ whose range of possible values follows from~\cite{zhang2010nearly} and~\cite{feng2017sorted}. 
The following proposition is a consequence of Corollary 1 in~\cite{feng2017sorted} together with Theorem 6 in~\cite{zhang2010nearly}. 
\begin{prp}{\label{prp:estimator_MCP}}
There exist constants $\underline{c}^{(MCP)}_{\eta}$, $c$, $c^{(1)}_{\eta}$-- $c^{(4)}_{\eta}$  such that the following holds for any $\theta^*$ with $ c^{(1)}_{\eta}\|\theta^*\|_0\leq m/\log(p)$. 
 With probability higher than $1-cp^{-1}$, the support $\widehat{S}_{MCP}$ of any stationary point of the criterion~\eqref{eq:critere_MCP} satisfies $\bS[c^{(2)}_{\eta},c^{(3)}_{\eta},c^{(4)}_{\eta}]$.
\end{prp}

A similar result holds if we use the non-convex SCAD penalty instead of MCP from~\cite{feng2017sorted}.

Now, we can plug the support estimator $\widehat{S}_{MCP}$ into the test $\phi^{(th)}$ with a suitable constant $\underline{c}_*$.
The following result is a straightforward consequence of Theorems \ref{thm:puissance_thres} and Proposition \ref{prp:estimator_MCP} and its proof is therefore omitted.

\begin{cor}\label{cor:test_MCP}
There exist  constants  $c$, $c_{\eta,\delta}$, and $c'_{\eta,\delta}$ such that the following holds under Condition  \emph{({\bf B}($\delta$))}. The test $\phi^{(th)}[\widehat{S}_{MCP}; \underline{c}^{(MCP),*}_{\eta}]$ satisfies \emph{({\bf gP1}[$ \frac{c}{p^2}+\delta$])} and  \emph{({\bf gP2}[$ \frac{c}{p^2}+\delta$])}  over the collections
\beq \label{eq:cond_MCP_separation}
 \bbB_{0}[k_0+\Delta]\bigcap \Big\{ \theta^*\, ,\ d^2_2\big[\theta^*;\bbB_0[k_0]\big]\geq c_{\eta,\delta}\sigma^2\frac{\Delta \log(p)}{n}\Big\}\ ,
\eeq\
for all  $1\leq \Delta \leq c'_{\eta,\delta} n/\log(p)$. 
\end{cor}

\subsubsection{Aggregated tests and summary}

Consider some $\delta>0$. Since the performances of the test $\phi^{(u)}$ are only assessed in the regime $p\leq c_{\eta} n^2\log^{-1}\big(\frac{2}{\delta}\big)$ ($c_{\eta}$ is introduced in Proposition \ref{prp:U}), we combine the tests $\phi^{(u)}$ and $\phi^{(th)}[\widehat{S}_{MCP}; \underline{c}^{(MCP),*}_{\eta}]$ only in that regime. For larger $p$, we solely use $\phi^{(th)}[\widehat{S}_{MCP}; \underline{c}^{(MCP),*}_{\eta}]$. Combining Proposition \ref{prp:U} and Corollary \ref{cor:test_MCP} to evaluate the separation distance of the aggregated test and comparing them with the minimax lower bounds of Proposition \ref{prp:lbuvkd}, we obtain the following characterization - note that we assume here that we are in the high dimensional regime, i.e.~ $p \geq n^{1+\zeta}$ where $\zeta>0$ is an arbitrarily small absolute constant.

 \medskip

\noindent 
 {\bf Case 1}: $p\leq n^{2-\kappa}$ with an arbitrary but fixed  $\kappa\in (0,1/2)$ and $k_0\leq \sqrt{p}p^{-\varsigma}$. 
 \[
  \boldsymbol{\rho}_{g,\gamma}^{*2}[k_0,\Delta]\asymp_{\gamma,\eta, \kappa,\zeta,\varsigma}\left\{
  \begin{array}{cc}
  \frac{\Delta}{n}\log(p) & \text{ if } \Delta \leq \min(\sqrt{p},k_0)p^{-\zeta}\ ; \\
 \frac{\sqrt{p}}{n} & \text{ if } \Delta \geq \sqrt{p}\ ,
                                                                    \end{array}
  \right.
 \]
for $1\leq \Delta\leq p-k_0$.

 \noindent 
 {\bf Case 2}: $p\leq n^{2-\kappa}$ with an arbitrary but fixed $\kappa\in (0,1/2)$ and $k_0\geq   \sqrt{p}$. 
 \beqn 
  \boldsymbol{\rho}_{g,\gamma}^{*2}[k_0,\Delta]&\asymp_{\gamma,\eta, \kappa,\zeta}&\frac{\Delta}{n}\log(p) \hspace{2cm}  \text{ if }\Delta \leq k_0p^{-\zeta}\ ; \\
   \frac{k_0}{n\log(p)}\lesssim_{\gamma,\eta, \kappa} \boldsymbol{\rho}_{g,\gamma}^{*2}[k_0,\Delta]&\lesssim_{\gamma,\eta, \kappa} & \frac{k_0\log(p)}{n}
  \hspace{2cm}  \text{ if } 
  \Delta \geq k_0\ , 
\eeqn 
for $1\leq \Delta\leq p-k_0$.

 \noindent 
{\bf Case 3: $p\geq n^2$.} For any $k_0$ and $\Delta$ smaller than  $c_{\eta}n/\log(p)$, we have
\[\boldsymbol{\rho}_{g,\gamma}^{*2}[k_0,\Delta]
\asymp_{\gamma,\eta} \frac{\Delta \log(p)}{n}\ , 
\]
whereas the problem become much more difficult for larger $\Delta$ or $k_0$.

\medskip

In conclusion, the aggregated test achieves the minimax separation distance except in the regime where $\sqrt{p}\leq k_0\leq \Delta \leq n$ where there is $\log^2(p)$ gap between the two squared rates.

Summing up our findings, we observe that 
\begin{itemize}
\item For sparse alternatives (small $\Delta$) - first result in Cases 1 and 2 and result in Case 3 - then the minimax separation distance is analogous to that of signal detection ($k_0=0$), i.e.~of order $\tfrac{\Delta \log(p)}{n}$. It would be straightforward to achieve this distance if we had at our disposal a $\sqrt{\log(p)/n}$  $l_\infty$-consistent estimator of $\theta^*$. However, this is not possible over the class of $\bSigma\in \cU(\eta)$ ($\bSigma$ unknown in this class)~\cite{cai2017confidence,javanmard2018debiasing}. This is why we use a slightly different approach that focuses on selecting most of the relevant features (as in Property
 $\bS[\ba_1,\ba_2,\ba_3]$~\eqref{eq:upper_CS}).
 
\item If $\Delta$ is large and $k_0$ is small - second result in Case 1 - then the squared minimax separation distance is of the order of $\sqrt{p}/n$ and is the same as for  signal detection ($k_0=0$). It is achieved by a $U$-statistic originally introduced for estimating $\|\theta^*\|_2^2$ when $\bSigma=\bI_p$~\cite{dicker_variance,verzelen2016adaptive}.

\item If $\Delta$ is large and $k_0$ is large - second result in Case 2 - then the lower bound on the minimax separation distance reflects the complexity of the null hypothesis. The lower bound is the same as in the independent setting, described in the previous section. The upper bound is based on the same $U$-statistic as in the previous case. Unfortunately the upper and lower bounds only match up to $\log^2(p)$ factor. In this general setting, we doubt that adapting the Fourier statistic of the previous section is possible and we conjecture that the squared separation distance is actually of the order $k_0\log(p)/n$. 

\item Finally, we emphasize that, for $\Delta$ large compared to $n/\log(p)$ and $p\geq n^2$, the optimal separation distance is huge (Proposition \ref{prp:lbuvkd}). Without further assumptions, it is therefore almost impossible to test whether $\theta^*$ is $k_0$-sparse or if $\theta^*$ is a dense vector when $p\geq n^2$. This result is in sharp contrast with the independent setting. 
\end{itemize}

\subsubsection{An alternative variable selection procedure}

In the previous section, we established that the test $\phi^{(th)}$ applied to the support $\widehat{S}_{MCP}$ estimated by the MCP estimator achieves the square separation rate $\Delta\log(p)/n$. Here, we introduce an alternative to the concave penalized estimator MCP based on simple iterations of the thresholded square-root Lasso.

 Starting from $\widehat{S}_0 = \emptyset$, the algorithm builds a subset  $\widehat{S}_t$  of variables iteratively from a subset $\widehat{S}_{t-1}$ of variables. It is done by applying a thresholded square-root Lasso to the data projected on the orthogonal of the variables  in $\widehat{S}_{t-1}$. Then, $\widehat{S}_t$ is the concatenation of  $\widehat{S}_{t-1}$ and the variables selected by the thresholded square-root Lasso. The procedure stops after approximately  $\log(n)$ iterations, and returns the current subset. The general idea is to iteratively remove  non-zero coordinates of $\theta^*$ so that the projected square-root Lasso estimator is less perturbed by large coordinates of $\theta^*$.

 \medskip 

 We need to introduce some notation. Define $T= \lfloor \log_2(n)\rfloor+1$. Assume without loss of generality that $m/T=n/(2T)$ is an integer. We divide the sample $(Y^{(1)},\bX^{(1)})$ into $T$ subsamples $\{( \underline{Y}^{(t)},\underline{\bX}^{(t)})\}$ of size $m/T$. 
Given a $r \times d$ matrix $\mathbf M$ and  some subset $S \subset \{1,\ldots,p\}$, we write $\mathbf M_{S}$ for the $r\times d$  matrix defined by $(\mathbf M_{S})_{i,j} = \mathbf M_{i,j} \1_{\{j \in S\}}$. Given $S$ and any $1\leq t \leq T$, define the subspace $V[S,\underline{\bX}^{(t)}]= \mathrm{vect}(\underline{\bX}_{S}^{(t)})$ of $\mathbb R^{m/T}$ and an $(m/T-\mathrm{dim}(V[S,\underline{\bX}^{(t)}]))\times m/T$ matrix  $\underline{\bPi}^{\perp}_{t,S}$ (measurable with  respect to $\bX^{(t)}_{S}$) whose corresponding linear application is null on $V[S,\bX^{(t)}]$ and maps isometrically the orthogonal of  $V[S,\bX^{(t)}]$ to $\mathbb{R}^{m/T-\mathrm{dim}(V[S,\underline{\bX}^{(t)}])}$.

\medskip

Next, we define the {\bf Thresholded square-root Lasso estimator}. Let $\underline{m}>0$ and let $\delta >0$.
Given a $\underline{m}\times p$ matrix $\underline{\bX}$ and a size  $\underline{m}$ vector $\underline{Y}$, we write $\underline{\bX}_c$ for the subdesign matrix of $\underline{\bX}$ where its null rows have been removed. Then, $\widehat{\theta}_{SL}(\underline{\bX},\underline{Y})$ stands for the square-root Lasso estimator  (see Equation~\eqref{eq:def_SL_1}) of $(\underline{\bX}_c,\underline{Y})$ with parameter $\lambda = 2 \sqrt{\overline{\Phi}^{-1}(\delta/(2p)}$. For the purpose of notation, we consider that $\widehat{\theta}_{SL}(\underline{\bX},\underline{Y})\in \mathbb{R}^p$ and that its entries $\widehat{\theta}_{SL}(\underline{\bX},\underline{Y})$ corresponding to null rows of $\underline{\bX}$ are equal to zero. 
Using the plug-in variance estimator 
$\hat \sigma^2 =  \|\underline{Y} -  \underline{\bX} [\widehat{\theta}_{SL}(\underline{\bX}, \underline{Y})] \|_2^2/\underline{m}$, 
we define the thresholding modification $\widehat{\theta}_{SL,t}(\underline{\bX},\underline{Y})$ of $\widehat{\theta}_{SL}(\underline{\bX}, \underline{Y})$ such that 
$$\big[\widehat{\theta}_{SL,t}(\underline{\bX},\underline{Y})\big]_i =\big[\widehat{\theta}_{SL}(\underline{\bX}, \underline{Y})\big]_i \mathbf 1\Big\{\big|[\widehat{\theta}_{SL}(\underline{\bX}, \underline{Y})]_i\big| \geq  \underline c_\eta^{(SL)} \hat \sigma \frac{8}{3}\sqrt{\frac{\log(p/\delta)}{\underline{m}}}\Big\},\quad \quad i=1,\ldots, p\ ,$$
where the constant $\underline c_\eta^{(SL)}$ is introduced in Lemma~\ref{lem:square_root_Lasso}.

\medskip

The set $\widehat{S}^{(ith)}$ is constructed as follows. We start with the empty support $\widehat{S}_0=\emptyset$. At each step $t=1,\ldots, T$, we project  both $\underline{\bX}^{(t)}$ and $\underline{Y}^{(t)}$ along the space $V[\widehat{S}_{t-1},\underline{\bX}^{(t)}]$ spanned by the variables in $\widehat{S}_{t-1}$. Then, we apply thresholded square-root Lasso to these projected data to select new variables. Finally, $\widehat{S}^{(ith)}$ is the last set $\widehat{S}_T$.

\begin{algorithm}[H]\label{algo:1}
\caption{Iterative construction of a set of relevant coordinates}
\label{EWA}
\begin{algorithmic}[1]
\REQUIRE  $\eta, \bf X, Y$   \COMMENT{$\eta$ is required to compute $\underline c_\eta^{(SL)}$ in the definition of $\widehat{\theta}_{(SL,t)}}$
\STATE $\widehat{S}_0=\emptyset$ \COMMENT{Initialization of the set of relevant coordinates}
	\FOR{$t = 1, \ldots, T, \do $}
		\STATE $\widehat{S}_{t} = \widehat{S}_{t-1} \cup \cS(\widehat{\theta}_{(SL,t)}[\underline{\bPi}^{\perp}_{t,\widehat{S}_{t-1}}\underline{\bX}^{(t)}, \underline{\bPi}^{\perp}_{t,\widehat{S}_{t-1}}\underline{Y}^{(t)}])$ \COMMENT{We re-compute at each step the thresholded square-root Lasso ) on data that have been projected on the orthogonal of previous supports}
	\ENDFOR
	\ENSURE $\widehat{S}^{(ith)}= \widehat{S}_{T}$ 
\end{algorithmic}

\end{algorithm}

\begin{thm}\label{thm:support}
There exist constants $\underline{c}^{(ith)}_{\eta}$ and $c'_{\eta}$ such that the following holds for any $\sigma>0$, $\bSigma\in \cU(\eta)$ and any $\theta^*$ satisfying
\beq\label{eq:boundspa}
c'_{\eta}\|\theta^*\|_0 \leq \frac{n}{\log^2(n)\log(p)}\ .
\eeq
With probability higher than $1- Tp^{-2}$, the estimator $\widehat{S}^{(ith)}$ satisfies 
$\bS[\underline{c}^{(ith)}_{\eta}\sqrt{T}, 2T,\underline{c}^{(ith)}_{\eta}\sqrt{T/2} ]$.
\end{thm}

It turns out that $\widehat{S}^{(ith)}$ satisfies the desired property $\cS[\ba_1,\ba_2,\ba_3]$ but with $\ba_1$ and $\ba_3$ that  are logarithmically large. As a consequence, the squared separation distance of the corresponding test $\phi^{(th)}$ with $\widehat{S}^{(ith)}$ is of order $\Delta\log(p)\log(n)/n$, which is optimal up to  an additional $\log(n)$ term in the regime $\Delta\leq k_0\wedge \sqrt{p}$.

\section{Discussion}\label{disc}
 
 In this section, we briefly discuss several related problems.

 \subsection{Low-dimensional problems}

Although some of our results are valid in a low-dimensional setting, we focused our attention on pinpointing the minimax separation distance in a high-dimensional regime $p\geq n^{1+\zeta}$ which is arguably the most interesting one. Let us briefly discuss the low dimensional regime $p\leq n/2$. In the independent setting, the main difference is that the $n^{-1/2}$ rate can be improved to $\sqrt{p}/n +k_0\log(p)/n$ by considering the ordinary least-square estimator and computing its $l_2$-norm when its $k_0$ largest entries are removed. In the general setting, we can recover similar upper bounds as in Section \ref{sec:general}, but with much simpler procedures based on the ordinary least-squares estimator.

Between these two regimes, the medium-dimensional case where $p$ and $n$ are of the same order is technically challenging. Our upper bounds and lower bounds only match up to polylogarithmic factor. Deriving the sharp minimax separation distance requires further work.

\subsection{Sparse inverse covariance matrices $\bSigma^{-1}$ and debiased Lasso}

Consider an intermediary setting where both $\sigma$ and $\bSigma$ are unknown but $\bSigma^{-1}$ is also restricted to have less than $\sqrt{n}/\log(p)$ non-zero entries on each rows. In this setting, the minimax lower bounds of Proposition~\ref{prp:lbuvkd} in the general setting turn out to be still valid. Indeed, the proof of Proposition~\ref{prp:lbuvkd} holds in the simpler setting where $\bSigma=\bI_p$ and $\sigma$ is unknown. As in the general setting,  the upper bound $k_0\log(p)/n + \sqrt{p}/n$ is achieved by the polynomial time $U$-statistic of Section \ref{sec:U-stat}. In contrast, achieving the $\sqrt{\Delta\log(p)/n}$ separation distance in the small $\Delta$ regime is now much easier than in the general setting. Whereas we introduced a refitted least-square estimator combined with the non-convex MCP regularized estimator,  one can now alternatively rely on the debiased Lasso method~\cite{javanmard2014confidence, zhang2014confidence, van2014asymptotically,javanmard2018debiasing} to obtain a $\sqrt{\log(p)/n}$ $l_{\infty}$-consistent  estimator of $\theta^*$ and then simply count  the number of its large entries. This was already done in \cite{javanmard2017flexible} as discussed previously.

\subsection{Know $\bSigma$ and unknown $\sigma^2$.}  Consider the intermediate scenario where  $\bSigma=\bI_p$ is the identity matrix, but $\sigma^2$ is unknown. As explained in the previous subsection, the minimax lower bounds of Proposition \ref{prp:lbuvkd} stated for the general setting  are still valid in this intermediate scenario. Obviously, we can also apply the testing procedures of Section~\ref{sec:general}.

 However, in the general scenario, our lower and upper bounds are only matching up to a $\log^2(p)$ factor in the large $k_0$, large $\Delta$ setting. More specifically, when $k_0\geq p^{1/2+\zeta}$ (for some $\zeta>0$) and $\Delta\geq k_0$, the lower bound of Proposition~\ref{prp:lbuvkd} is of order $k_0/[n\log(p)]$ whereas Proposition~\ref{prp:U} provides an upper bound of the  order of $k_0\log(p)$. 
 
 It turns out that, in this intermediate scenario, the gap is easily closed by 
adapting the Fourier-based test $\phi^{(f)}$ and $\phi^{(i)}$ introduced in Section \ref{sec:independent_design}. Indeed, the only place where the knowledge of $\sigma$ is necessary in these two tests is in the definition of the pre-estimator $\overline{\theta}_{\bI}$ which is a thresholded version of $\widetilde{\theta}_{\bI}$~\eqref{eq:phi_proj}. If we replace $\sigma$ in this threshold by the plug-in estimator of the variance based on the square-root Lasso and if we increase some constants, this modification of the tests $\phi^{(f)}$ and $\phi^{(i)}$ does not depend anymore on $\sigma$. Besides, one can easily check that (up to some changes in the numerical constants) Propositions \ref{prp:analyse_Z_f} and \ref{prp:analyse_Z_i} are still valid for these tests.

 \subsection{Unknown $\bSigma$ and known $\sigma^2$.} 

 In this case, we can improve the upper bounds of the general case by adapting the test $\phi^{(\chi)}$ from Section~\ref{sec:independent_design}. Indeed, the statistic $Z_\chi(\hat R_{k_0})$ is now  centered on $\|\bSigma^{1/2}(\theta^* - \tilde \theta_{SL, k_0)})\|_2^2 \geq \eta^{-1}\|\theta^* - \tilde \theta_{SL, k_0)}\|_2^2$ on the class $\mathcal U(\eta)$ of $\bSigma$. Hence, the corresponding test achieves a squared separation distance of the order of $n^{-1/2}+ k_0\log(p)/n$. The main difference with the independent case is that we are not able to adapt the Fourier-based test $\phi^{(f)}$ and $\phi^{(i)}$ to unknown $\bSigma$. In regimes where $p^{1/2+\kappa} \leq k_0 \leq c_\gamma \frac{n}{\log(p)}$ and $\Delta \geq k_0$, there is therefore  a $\log^2(p)$ gap between our upper and lower bounds.

\section{Proofs of the minimax upper bounds}

\subsection{Some results on the square-root Lasso and a simple debiased Lasso}

We start with a few probability bounds for the square-root Lasso $\widehat{\theta}_{SL}$ and its thresholded modification $\widetilde{\theta}_{SL,k_0}$ where only the $k_0$ largest values of $\widehat{\theta}_{SL}$ are not set to $0$. They almost follow straightforwardly from earlier results~\cite{squarerootLasso,kolt11,2012_Sun,10:RWG_restricted}. As we shall apply this lemma in different contexts, we reintroduce the setting here. We consider a $m\times q$ linear regression model $Y=\bX \theta^* + \sigma \epsilon$ with $\epsilon\sim \cN(0,\bI_m)$ and where the rows of $\bX$ are independent and follow a centered normal distribution with common covariance matrix $\bSigma$. The $m\times q$ matrix $\bT$ is the column normalized version of $\bX$ i.e.
$$\mathbf T = \Big(\bX_{.,1}/\|\bX_{.,1}\|_2, \ldots, \bX_{.,q}/\|\bX_{.,q}\|_2 \Big).$$
  We take
 \beq\label{eq:choic_lambda_generic}
  \lambda_{\delta,m,q} = 2\sqrt{\overline{\Phi}^{-1}(\delta/(4q))}\ , 
  \eeq
and consider the square-root Lasso estimator~\cite{squarerootLasso,2012_Sun}\ , 
\beq\label{eq:def_SL_generic}
\widehat{\theta}_{SL,N} \in \arg \min \|Y-\bT\theta\|_2 + \lambda_{\delta,m,q} \|\theta\|_1\ ,
\eeq
Then, define $(\widehat{\theta}_{SL})$ as $(\widehat{\theta}_{SL})_i= (\widehat{\theta}_{SL,N})_i/\|\bX_{.,i}\|_2$ for any $i=1,\ldots, q$ and $\widetilde{\theta}_{SL,k_0}= \argmin_{\theta \in \mathbb B_0[k_0]}\|\theta-\widehat{\theta}_{SL}\|_{2}^2$. The plug-in variance estimator is $\widehat{\sigma}_{SL}= \|Y-\bX \widehat{\theta}_{SL}\|_2/ \sqrt{m}$.

\begin{lem}\label{lem:square_root_Lasso}
 Fix any $0<\delta\leq 1/2$ and any $\eta \geq 1$. There exist constants $\underline c^{(SL)}_{\eta}$ and  $\underline c^{(SL),2}_{\eta}$ such that the following holds. Let $k_{\max}$ be the largest integer such that 
 \[
\underline{c}^{(SL),2}_{\eta} \big[k_{\max}\log(q/\delta)+ \log(1/\delta)\log(q)\big] \leq m\ ,
\]
For any  $\sigma>0$, $\bSigma\in \cU[\eta]$ and $\theta^*$ with $\|\theta^*\|_0\leq k_{\max}$,
there exists an event $\cE$ of probability higher than $1-\delta$, such that
 \beqn 
  \|\widehat{\theta}_{SL}- \theta^*\|_2^2&\leq& \underline c^{(SL)}_{\eta} \sigma^2  \frac{\|\theta^*\|_0 \vee 1}{m} \log\left(\frac{q}{\delta}\right)\ ;\\
 \|\widetilde{\theta}_{SL,k_0}- \theta^*\|_2^2&\leq&  \underline c^{(SL)}_{\eta}\Big[\sigma^2 \frac{\|\theta^*\|_0 \vee 1}{m} \log\left(\frac{q}{\delta}\right) + d^2_2\big[\theta^*;\bbB_0[k_0]\big] \Big]\ ;\\
 \frac{3}{4} &\leq& \frac{\widehat{\sigma}_{SL}}{\sigma}\leq \frac{5}{4}\ . 
 \eeqn 
\end{lem}

\begin{proof}[Proof of Lemma \ref{lem:square_root_Lasso}]
We first argue that the design matrix $\bT$ satisfies  the compatibility  property (see \cite{kolt11,2012_Sun}) with any set of size less than $k_{\max}$ and constant depending on $\eta$. Indeed, Corollary 1 in~\cite{10:RWG_restricted} enforces that this property holds with probability higher than $1-qe^{-cm}\geq 1-\delta/2$. Then, we are in position to apply Theorem 1 in~\cite{2012_Sun}, which implies that $\widehat{\sigma}_{SL}/\sigma$ belongs to $[3/4,5/4]$ and that
\[
\|\widehat{\theta}_{SL}- \theta^*\|_2^2\leq \underline c^{(SL)}_{\eta} \sigma^2  \frac{\|\theta^*\|_0 \vee 1}{m} \log\left(\frac{q}{\delta}\right)\ .
\]
 
 The second result of the lemma is a consequence of the first result. Denote $S_1$ (resp. $S_2$) the subset of the $k_0$ largest entries of $\widehat{\theta}_{SL}$ (resp. $\theta^*$). 
 From the definition of 
 $\widetilde{\theta}_{SL,k_0}$, we deduce that 
 \beqn
 \|\widetilde{\theta}_{SL,k_0}-\theta^*\|_2^2&\leq&  
 2  \|\widehat{\theta}_{SL} - \theta^*\|_2^2 +  2\sum_{i\in \overline{S_1}}\big(\widehat{\theta}_{SL}\big)^2_{i} \\
 &\leq & 2\|\widehat{\theta}_{SL} - \theta^*\|_2^2 +  2\sum_{i\in \ol{S_2}}\big(\widehat{\theta}_{SL}\big)^2_{i} \\
& \leq & 6 \|\widehat{\theta}_{SL} - \theta^*\|_2^2 + 4 \sum_{i\in \ol{S_2}}\theta^{*2}_{i}\ = 6 \|\widehat{\theta}_{SL} - \theta^*\|_2^2 + 4 d_2^2[\theta^*; \bbB_0[k_0]]\ . 
\eeqn 
The result follows.
\end{proof}

 \subsection{Analysis of the tests $\phi^{(t)}$, $\phi^{(\chi)}$, and $\phi^{(u)}$}

 \subsubsection{Proof of Proposition~\ref{prp:t} (Test $\phi^{(t)}$)}

We start with a $l_{\infty}$ error bound on $\widetilde{\theta}_{\bI}$. 
\begin{lem}\label{lem:thresh}
 There exists a constant $c$ such that the following holds under ($\bA[\alpha\wedge \delta]$). For any  $\theta^*\in \mathbb R^p$ with $\|\theta^*\|_0\leq k_{\max}$ (with $k_{\max}$ as in Lemma \ref{lem:square_root_Lasso}), we have 
 \[ \|\theta^*-\widetilde{\theta}_{\bI}\|_\infty \leq c \sigma \sqrt{\frac{\log(p/\alpha)}{n}} \ , \]
 with  probability higher  than $1-\delta-\alpha$. Besides, for any $\theta^*\in \mathbb{R}^p$, we have 
 \[ \|\theta^*-\widetilde{\theta}_{\bI}\|_\infty  \1_{\|\widehat{\theta}_{SL}-\theta^*\|_2^2\leq 2 \sigma^2 }\leq c \sigma \sqrt{\frac{\log(p/\alpha)}{n}} \ , \]
 with probability higher than $1-\alpha$. 
 \end{lem}

From Lemma~\ref{lem:thresh}, we derive that with probability higher than $1-\delta-\alpha$, we have 
  $\|\theta^*-\widetilde{\theta}_{\bI}\|_\infty \leq c \sigma \sqrt{\frac{\log(p/\alpha)}{n}}$. Setting $\underline{c}^{(t)}$ as $c$ in Lemma~\ref{lem:thresh}, we derive that the test $\phi^{(t)}$ has a type I error probability less or equal to $\alpha+\delta$. 
Now consider a vector $\theta^*\in \bbB_0[k_0+\Delta]$ such that  $|\theta^*_{(k_0+1)}|\geq 2.1\underline{c}^{(t)} \sigma\sqrt{\log[p/(\alpha\wedge \beta)]/n}$. From Lemma \ref{lem:thresh}, we deduce that, with probability higher than $1-\alpha-\beta$,
\[
|(\widetilde{\theta}_{\bI})_{(k_0+1)}|\geq |\theta^*_{(k_0+1)}|- \|(\widetilde{\theta}_{\bI})-\theta^*\|_{\infty}  \geq  \underline{c}^{(t)}\sigma(2.1\sqrt{\log(p/(\alpha \wedge \beta)} - \sqrt{\log(p/\beta)}) \frac{\sigma}{\sqrt{n}} > \underline{c}^{(t)} \sigma \sqrt{\frac{\log(p/\alpha)}{n}}\ , 
\]
and the test $\phi^{(t)}$ therefore  rejects the null hypothesis, which concludes the proof.

\begin{proof}[Proof of Lemma~\ref{lem:thresh}]
 Set $\widehat{\gamma} = \theta^*-\widehat{\theta}_{SL}$. If $\|\theta^*\|_0\leq k_{\max}$, the conditions of Lemma \ref{lem:square_root_Lasso} are satisfied and it follows from this lemma  that 
\begin{equation}\label{eq:det:2}
\|\widehat{\gamma} \|_2^2 \leq c \sigma^2  (\|\theta^*\|_0  \lor 1)\frac{\log(p/\delta)}{m}  \leq c'\sigma^2\ ,
\end{equation}
with $\P^{(1)}_{\theta^*,\sigma}$ probability higher than $1-\delta$.  In the second result of Lemma \ref{lem:thresh},  we restrict ourselves to the case $\|\widehat{\gamma}\|_2^2 \leq 2 \sigma^2$. Hence, it suffices to prove that, conditionally to $\widehat{\gamma}$ satisfying  $\|\widehat{\gamma}\|_2^2 \leq (2\vee c') \sigma^2$, we have $\|\theta^*-\widetilde{\theta}_{\bI}\|_\infty \leq c_1 \sigma \sqrt{\frac{\log(p/\alpha)}{n}}$ with probability higher than $1-\alpha$.

Write $Z$ the statistic defined by $Z= \theta^* - \widehat{\theta}_{SL}  - \frac{1}{m}\bX^{(2)T} (Y^{(2)}- \bX^{(2)}\widehat{\theta}_{SL})$. Also, define $\widehat{\bSigma}= \frac{1}{m}\bX^{(2)T} \bX^{(2)}$ and $\widehat{\bGamma}$ its diagonal part. We have
\beqn 
\|Z\|_{\infty}&= &\|\widehat{\gamma}- \frac{1}{m}\bX^{(2)T}\bX^{(2)}\widehat{\gamma} - \frac{\sigma}{m}\bX^{(2)T}\epsilon^{(2)}\|_{\infty} \\
&\leq & \|\widehat{\gamma}- \frac{1}{m}\bX^{(2)T}\bX^{(2)}\widehat{\gamma}\|_{\infty} +  \| \frac{\sigma}{m}\bX^{(2)T}\epsilon^{(2)}\|_{\infty}\\
&\leq & \|(\bI_m - \widehat{\bGamma}) \widehat{\gamma}\|_{\infty}+ \|(\widehat{\bSigma}- \widehat{\bGamma}) \widehat{\gamma}\|_{\infty} +  \| \frac{\sigma}{m}\bX^{(2)T}\epsilon^{(2)}\|_{\infty}= A_1+ A_2 + A_3.
\eeqn 
We control each of these three quantities independently.

\begin{lem}\label{lem:conchi}
Let $\bQ$ be a $d\times d$ symmetric matrix and  let $G\sim \cN(0,\bI_d)$. Define $S=G^T \bQ G$. For any $t>0$, one has 
\[
 S\leq Tr(\bQ) + 2\|\bQ\|_F t + 2 \|\bQ\|_{op} t, 
\]
with probability higher than $1-e^{-t}$. Here, $\|\bQ\|_F$ and $\|\bQ\|_{op}$ respectively correspond to the Frobenius and operator norm of $\bQ$. 
\end{lem}
This result is a slight extension of Lemma 1 in \cite{Laurent00} (that requires $\bQ$ to be positive). The extension to general symmetric  matrices proceeds from the same arguments and we omit the proof.

Let us first control $A_3$. Each of the $p$ entries of $\sigma m^{-1}\bX^{(2)T}\epsilon^{(2)}$ is distributed as a quadratic form of $2m$ standard normal random variables. The corresponding matrix $\bQ$ satisfies $\tr(\bQ)=0$ and $\|\bQ\|_F^2 = \sigma^2/(2m)$. Since $\|\bQ\|_{op}\leq \|\bQ\|_{F}$, it follows from the above lemma together with an union bound that
\beq\label{eq:upper:A3}
A_3 \leq  4\sigma \sqrt{\frac{\log(6p/\alpha)}{2 m}}\ .
\eeq
with $\P^{(2)}_{\theta^*,\sigma}$ probability higher than $1-\alpha/3$. As for $A_1$ and $A_2$, we first work conditionally to $\widehat{\gamma}$. Fix $i\in \{1,\ldots,p\}$, $\widehat{\bGamma}_{ii}$ is distributed as quadratic form of $m$ standard normal variable and the corresponding matrix $\bQ$ satisfies $tr(\bQ)= 1$, $\|\bQ\|_F= m^{-1/2}$ and $\|\bQ\|_{op}\leq 1/m$. It then follows from Lemma \ref{lem:conchi}, that conditionally to $\widehat{\gamma}$, 
\beq\label{eq:upper:A1}
A_1 \leq  2 (1+ m^{-1/2})\sqrt{\frac{\log(6p/\alpha)}{m}}\|\widehat{\gamma}\|_{\infty}\ ,
\eeq
with $\P^{(2)}_{\theta^*,\sigma}$ probability higher than $1-\alpha/3$. As for $A_2$, observe that, conditionally to $\widehat{\gamma}$, $[(\widehat{\bSigma}-\widehat{\bGamma})\widehat{\gamma}]_{j} =  \frac{1}{m}\sum_i \bX_{i,j} \sum_{j'\neq j}\bX_{i,j'} \hat \gamma_{j'}$ is distributed as $(\sum_{j'\neq j}\widehat{\gamma}_{j'}^2)^{1/2}/m \sum_{q=1}^m U_q U'_q$ where the $U_q$'s and $U'_q$'s are independent standard normal random variables. Again, we deduce from Lemma \ref{lem:conchi}
 that, conditionally to $\widehat{\gamma}$, 
\beq\label{eq:upper:A2}
A_2 \leq  4 \sqrt{\frac{\log(6p/\alpha)}{2 m}} \max_{i}(\sum_{j\neq i}\widehat{\gamma}_j^2)^{1/2} \leq 4\|\widehat{\gamma}\|_2 \sqrt{\frac{\log(6p/\alpha)}{2 m}}\  ,
\eeq
with $\P^{(2)}_{\theta^*,\sigma}$ probability higher than $1-\alpha/6$. Finally, we gather \eqref{eq:det:2} with (\ref{eq:upper:A3}--\ref{eq:upper:A2}) to conclude that there exists $c>0$ such that 
\[
 \|Z\|_{\infty}\leq c \sigma\sqrt{\frac{\log(p/\alpha)}{m}}\ ,  
\]
with probability larger than $1-\delta-\alpha$.

\end{proof}

\subsubsection{Proof of Proposition~\ref{prp:chi} (Test $\phi^{(\chi)}$)}

We first state the following lemma that characterizes the deviations of $Z_{\chi}[R]$.

\begin{lem}\label{lem:t_comp}
For any $t>0$, any $\theta^*\in \mathbb R^p$, any  $\sigma>0$, and any fixed $\theta$, we have for, $\widehat{R}= Y^{(2)}-\bX^{(2)}\theta$, 
\begin{align}
  \big|Z_{\chi}(\widehat{R}) - \frac{\|\theta^*-\theta \|_2^2}{\sigma^2}\big| \leq 2 \Big[1 +\frac{\|\theta^*-\theta\|_2^2}{\sigma^2}\Big] \big[\sqrt{\frac{t}{m}} +\frac{t}{m}\big]\ ,
 \end{align}
 with $\P^{(2)}_{\theta^*,\sigma}$ probability higher than $1-e^{-t}$.
\end{lem}
\begin{proof}[Proof of Lemma \ref{lem:t_comp}]
We have
\begin{align*}
\widehat{R}_i &= (Y^{(2)}-\mathbf{X}^{(2)}\theta)_i = \epsilon^{(2)}_i + (\mathbf{X}^{(2)}(\theta^* - \theta))_i\\
&=  \epsilon^{(2)}_i + \sum_j \bX^{(2)}_{i,j}(\theta^*_j - \theta_j)\ .
\end{align*}
Hence, $\widehat R_i \sim \mathcal N(0, \sigma^2 + \|\theta^* - \theta\|_2^2)$ and these variables are independent from each other.
 So the random variable  
 $\|\widehat{R}\|_2^2[\sigma^2+ \|\theta^*-\theta\|_2^2]^{-1}$ follows a $\chi^2$ distribution with $n$ degrees of freedom. To prove the result, we only have to apply  Lemma~\ref{lem:conchi} with $\mathbf{Q} = \mathbf{I}_p$.
\end{proof}

First assume that $\theta^*$ belongs to $\bbB_0[k_0]$.  With $\P^{(2)}_{\theta^*,\sigma}$ probability higher than $1-\alpha$, we have 
\beqn
 Z_{\chi}(\widehat{R}_{k_0})&\leq & \frac{\|\theta^*-\widetilde{\theta}_{SL,k_0}\|_2^2}{\sigma^2}+ 2\Big[1 +\frac{\|\theta^*-\widetilde{\theta}_{SL,k_0}\|_2^2}{\sigma^2}\Big] \Big[\sqrt{\frac{\log(1/\alpha)}{m}} +\frac{\log(1/\alpha)}{m}\Big] \\
 &\leq &  \frac{\|\theta^*-\widetilde{\theta}_{SL,k_0}\|_2^2}{\sigma^2}+ c\Big[1 +\frac{\|\theta^*-\widetilde{\theta}_{SL,k_0}\|_2^2}{\sigma^2}\Big]\sqrt{\frac{\log(1/\alpha)}{m}} \ ,
\eeqn 
where we used Condition ($\bA[\alpha\wedge \delta$]). Then, we apply Lemma \ref{lem:square_root_Lasso} to control $\|\theta^*-\widetilde{\theta}_{SL,k_0}\|_2^2$ with probability higher than $1-\delta$. With probability higher than $1-\alpha-\delta$, we get
\[
  Z_{\chi}(\widehat{R}_{k_0})\leq c\Big[ (k_0\vee 1) \frac{\log(p/\delta)}{m}+ \sqrt{\frac{\log(1/\alpha)}{m}}\Big]\ ,
\]
so that choosing the constant $\underline{c}^{(\chi)}$ large enough leads to $({\bf P1}[\alpha+\delta])$.

Now assume that $d_2(\theta^*;\bbB_0[k_0])>0$. Since $\widetilde{\theta}_{SL,k_0}$ is $k_0$-sparse, it follows that  $\|\theta^*-\widetilde{\theta}_{SL,k_0}\|_2^2\geq d^2_2\big[\theta^*;\bbB_0[k_0]\big]$. 
Then, Lemma \ref{lem:t_comp} enforces that, for $\log(1/\beta)$ small enough compared to $n$ (which is ensured by Condition $({\bf A}[\alpha\wedge \beta\wedge \delta])$), one has 
\beqn 
Z_{\chi}(\widehat{R}_{k_0})&\geq& \frac{d^2_2\big[\theta^*;\bbB_0[k_0]\big]}{\sigma^2}\Big[1 - 4 \sqrt{\frac{\log(1/\beta)}{m}}  \Big]-  4 \sqrt{\frac{\log(1/\beta)}{m}}\\
&\geq & \frac{d^2_2\big[\theta^*;\bbB_0[k_0]\big]}{2\sigma^2}- 4 \sqrt{\frac{\log(1/\beta)}{m}}\ ,
\eeqn 
with $\P^{(2)}_{\theta^*,\sigma}$ probability larger than $1-\beta$. As a consequence, under condition \eqref{eq:condition_distance_chi} with a constant $c$ large enough, the type II error probability is less than $\beta$.

\subsubsection{Proof of Proposition~\ref{prp:U} (test $\phi^{(u)}$)}

The following lemma is borrowed from Theorem 2.1  in~\cite{verzelen2016adaptive}.
\begin{lem}\label{lem:deviation_moment_statistic}  
  There exist numerical constants $c>0$ and $c'>0$ such that the following holds. Assume that $p\geq m$. Consider any $\theta^*\in \mathbb{R}^p$, any $\sigma>0$, and  any  $\bSigma\in \cU(\eta)$. Given any estimator $\widehat{\theta}$ based on the subsample $(Y^{(1)},\bX^{(1)})$,  define  $\widehat{R}= Y^{(0)}-\bX^{(0)}\widehat{\theta}$.  We have, conditionally on $(\bX^{(1)}, Y^{(1)})$,
\beq\label{eq:concentration_M_dense}
\P_{\theta^*,\sigma,\bSigma}^{(0)}\left[\Big|Z^{(u)} - \frac{(\theta^*-\widehat{\theta})^{T} \bSigma^2 (\theta^*-\widehat{\theta})}{\sigma^2 + (\theta^*-\widehat{\theta})^{T}\bSigma(\theta^*-\widehat{\theta})}\Big|\leq c\eta\frac{\sqrt{pt}}{m}\right]\geq 1- c' e^{-t}\ ,
\eeq
for  all $t\leq n^{1/3}$. 
 \end{lem}

First, assume that $\theta^*$ belongs to $\bbB_0[k_0]$. Since Condition ({\bf B}[$\alpha,\beta$]) is satisfied with a constant $\underline{c}_{\bB}^{\eta}$ large enough, we can apply \eqref{eq:concentration_M_dense}. With probability higher than $1-\alpha$, one has 
\beqn 
Z^{(u)}&\leq&  \frac{(\theta^*-\widetilde{\theta}_{SL,k_0})^{T} \bSigma^2 (\theta^*-\widetilde{\theta}_{SL,k_0})}{\sigma^2 + (\theta^*-\widetilde{\theta}_{SL,k_0})^{T}\bSigma(\theta^*-\widetilde{\theta}_{SL,k_0})}+ c\eta\frac{\sqrt{p\log(c'/\alpha)}}{m}\\
&\leq & \eta^2 \frac{\|\theta^* -  \widetilde{\theta}_{SL,k_0}\|_2^2}{\sigma^2} + c\eta\frac{\sqrt{p\log(c'/\alpha)}}{m}\ .
\eeqn 
Then, we use Lemma \ref{lem:square_root_Lasso}  to conclude that 
\[
 Z^{(u)}\leq c_{\eta}\Big[(k_0\vee 1)\frac{\log(p/\delta)}{m}+ \frac{\sqrt{p\log(c'/\alpha)}}{m}  \Big]\ , 
\]
with probability higher than $1-\alpha -\delta$. Setting the constant $\underline{c}_{\eta}^{(u)}$ small enough, we conclude that the type I error probability of $\phi^{(u)}$ is less than $\alpha+\delta$.

Now assume that $d_2(\theta^*;\bbB_0[k_0])>0$. Since $\widetilde{\theta}_{SL,k_0}$ is $k_0$-sparse, it follows that  $\|\theta^*-\widetilde{\theta}_{SL,k_0}\|_2^2\geq d^2_2\big[\theta^*;\bbB_0[k_0]\big]$. 
Then, Lemma \ref{lem:deviation_moment_statistic} enforces that, for $\log(1/\beta)$ small enough compared to $n$,  with probability higher than $1-\beta$, one has 
\beqn 
Z^{(u)}&\geq&  \frac{(\theta^*-\widetilde{\theta}_{SL,k_0})^{T} \bSigma^2 (\theta^*-\widetilde{\theta}_{SL,k_0})}{\sigma^2 + (\theta^*-\widetilde{\theta}_{SL,k_0})^{T}\bSigma(\theta^*-\widetilde{\theta}_{SL,k_0})}- c\eta\frac{\sqrt{p\log(c'/\beta)}}{m}\\
&\geq & \eta^{-1}\frac{(\theta^*-\widetilde{\theta}_{SL,k_0})^{T} \bSigma (\theta^*-\widetilde{\theta}_{SL,k_0})}{\sigma^2 + (\theta^*-\widetilde{\theta}_{SL,k_0})^{T}\bSigma(\theta^*-\widetilde{\theta}_{SL,k_0})}- c\eta\frac{\sqrt{p\log(c'/\beta)}}{m}\\
&\geq & \frac{1}{2\eta}\Big[ \frac{(\theta^*-\widetilde{\theta}_{SL,k_0})^{T} \bSigma (\theta^*-\widetilde{\theta}_{SL,k_0})}{\sigma^2} \wedge 1 \Big] - c\eta\frac{\sqrt{p\log(c'/\beta)}}{m}\\
&\geq &\frac{1}{2\eta^2}\Big[ \frac{d^2_2\big[\theta^*;\bbB_0[k_0]\big]}{\sigma^2} \wedge 1 \Big] - c\eta\frac{\sqrt{p\log(c'/\beta)}}{m}\ , 
\eeqn 
where we used in the second and fourth line that $\bSigma$ belongs $\cU(\eta)$. Now assume that $d_2(\theta^*;\bbB_0[k_0])$ is large enough so that Condition \eqref{eq:cond_M_separation} is satisfied. Choosing the constant $c'_{\eta}$ in \eqref{eq:cond_M_separation} large enough and the constant  $c_{\eta}$ small enough, it then 
follows that the type II error probability is smaller than $\beta$.

\subsection{Analysis of $\phi^{(f)}$ and $\phi^{(i)}$ (Propositions \ref{prp:analyse_Z_f} and \ref{prp:analyse_Z_i})}

In the proofs of this subsection, we set
$$\theta = \theta^* - \overline{\theta}_{\bI}.$$
To alleviate the notation, and since $\overline{\theta}_{\bI}$ only depends on the first two subsamples, $\theta$ can be considered as fixed when we condition to these  subsamples. To simplify the notation, we respectively write henceforth $Y$ and $\bX$ for $\overline{Y}^{(3)}$  and $\bX^{(3)}$  and work conditionally to $\theta$. For any $1\leq i\leq m$ and $1\leq j\leq p$, we have  $\var{Y_i}=\sigma^2+\|\theta\|_2^2$ and $\cov{Y_i,\bX_{i,j}}= \theta_j$. Hence, we have \begin{equation*}
\bX_{i,.}|(Y, \theta) \sim \mathcal N\left(\frac{Y_i \theta}{\sigma^2 +\|\theta\|_2^2}, \left(\mathbf I_p - \frac{\theta\theta^T}{\sigma^2 +\|\theta\|_2^2}\right) \right)\ ,
\end{equation*}
and since the $\bX_{i,.}|(Y, \theta)$ are independent, we have
\begin{equation*}
\bX^T Y |(Y, \theta) \sim \mathcal N\left(\frac{\|Y\|_2^2 \theta}{\sigma^2 +\|\theta\|_2^2}, \|Y\|_2\left(\mathbf I_p - \frac{\theta\theta^T}{\sigma^2 +\|\theta\|_2^2}\right) \right).
\end{equation*}
 For 
$W = \bX^T Y$, it holds that 
\beq\label{eq:definition_v}
\frac{W}{\|Y\|_2}\Big|(\|Y\|_2, \theta)\sim \cN\left(v, \bI_p - \frac{\theta \theta^T}{\sigma^2 + \|\theta\|_2^2}\right)\ ; \quad \text{ with } v=\frac{\theta \|Y\|_2}{\sigma^2 +\|\theta\|_2^2}.
\eeq
As a consequence, given $\|Y\|_2$ and $\theta$, $\frac{W}{\|Y\|_2}$ behaves almost like a standard Gaussian vector. We shall prove that, under the condition of the propositions, the term  $\frac{\theta \theta^T}{\sigma^2 + \|\theta\|_2^2}$ in the covariance turns out to be negligible, whereas $\theta\|Y\|_2/[\sigma^2+\|\theta\|_2^2]$ is closely related to $\sqrt{n}\theta^*/\sigma$. 
The following lemma states that the conditional expectations of $Z_f$ and $V(r_l,\omega_l)$  are almost the same as if the conditional covariance of $W/\|Y\|_2$ was the identity matrix. Recall the function $g$ introduced in Section \ref{ss:fourier}. Define the function $\Psi_l(x)$ by $\Psi_l(x)= \E[\eta_{r_l,\omega_l}(X)]$ where $X\sim \cN(x,1)$. As explained in Section \ref{ss:fourier_intermediary} and proved in \cite{carpentier2017adaptive} (Section C.2.3), $\Psi_l(x)= \frac{1}{1-2\overline{\Phi}(r_l)}\int_{-r_l}^{r_l}\phi(\xi) \cos(\xi x \tfrac{\omega_l}{r_l})d\xi$. Obviously, we have $\Psi_l(0)=1$. Besides it is also shown in \cite{carpentier2017adaptive} (Section C.2.3) that $-\frac{l}{k_0}\leq\Psi_l(x)\leq \frac{l}{k_0}+ 2 \exp\left(- \frac{\omega_l^2x^2}{r_l^2}\right)$.

\begin{lem}\label{lem:esperance_fourier_statistics}
 If  $s\|\theta\|_{\infty}\leq \sigma$, we have
 \beq
 \label{eq:control_E_Z_f}
\Big|\E_{\theta^*,\sigma}^{(3)}[Z_f|(\|Y\|_2,\theta)] -  \sum_{i=1}^{p}\left[\1_{(\overline{\theta}_I)_i= 0}g(s v_i) + \1_{(\overline{\theta}_I)_i\neq 0}\right]\Big|\leq  \frac{s^2}{5} \ .
\eeq
Consider any $l\in \cL_0$.
If $\omega_l \|\theta\|_{\infty}\leq \sigma$, we have 
 \beq \label{eq:control_E_V}
\Big|\E_{\theta^*,\sigma}^{(3)}[V(r_l,\omega_l)|(\|Y\|_2,\theta)] -  \sum_{i=1}^{p}\left[ \1_{(\overline{\theta}_I)_i= 0}\left(1-\Psi_{l}(v_j)\right)+  \1_{(\overline{\theta}_I)_i\neq 0}\right]\Big|\leq  
\frac{e^{1/2}}{2}\omega_l^2\ . 
\eeq
\end{lem}

Also, the next lemma enforces that the deviations of the statistics $Z_f$ and $V(r_l,\omega_l)$ are almost the same as if the conditional covariance of $W/\|Y\|_2$ was the identity matrix. 
\begin{lem}\label{lem:concentration_fourier_statistics}
Assume that $\|\theta\|_{\infty}\leq [\sigma^2 +\|\theta\|_2^2]^{1/2}/s$. For any $t>0$, one has
\begin{eqnarray}\label{eq:concentation_Z_f}
\P_{\theta^*,\sigma}^{(3)}[Z_f- \E^{(3)}_{\theta^*,\sigma}[Z_f|(\|Y\|_2,\theta)]\geq se^{s^2/2}\sqrt{2pt}\big| (\|Y\|_2,\theta)\big]&\leq& e^{-t}  \ ;\\
\P_{\theta^*,\sigma}^{(3)}[Z_f- \E^{(3)}_{\theta^*,\sigma}[Z_f|(\|Y\|_2,\theta)]\leq -se^{s^2/2}\sqrt{2pt}\big| (\|Y\|_2,\theta)\big]&\leq& e^{-t}  \ .\nonumber
\end{eqnarray}
Besides, for any $l\in \cL_{0}$ and any $t>0$, one has 
\begin{eqnarray}
\P_{\theta^*,\sigma}^{(3)}\left[V(r_l,\omega_l)- \E^{(3)}_{\theta^*,\sigma}[V(r_l,\omega_l)|(\|Y\|_2,\theta)]\geq 
\sqrt{2lp^{1/2}t}
\Big| (\|Y\|_2,\theta)\right]&\leq& e^{-t}  \ ;\label{eq:concentation_V}\\
\P_{\theta^*,\sigma}^{(3)}\left[V(r_l,\omega_l)- \E^{(3)}_{\theta^*,\sigma}[V(r_l,\omega_l)|(\|Y\|_2,\theta)]\leq -\sqrt{2lp^{1/2}t}\Big| (\|Y\|_2,\theta)\right]&\leq& e^{-t}  \ .\nonumber
\end{eqnarray} 
\end{lem}

\bigskip

\noindent 
\underline{Analysis of the tests under the null hypothesis}.
The assumptions of Lemma~\ref{lem:thresh} are fulfilled. As a consequence,  we have  $\|\theta^* - \widetilde{\theta}_{\bI}\|_{\infty} \leq \underline{c}^{(t)} \sigma\sqrt{\log(2p/\alpha)/n}$ with  probability larger than $1-\delta-\alpha/2$. Henceforth, we call this event $\cB$ and work conditionally to it. 
Thus, the support of $\overline{\theta}_{\bI}$ is included in that of $\theta^*$ which in turn implies that
$\sum_{j=1}^p \1_{\theta_j\neq 0}\1_{(\overline{\theta}_{\bI})_j=0}+ \1_{(\overline{\theta}_{\bI})_j\neq 0}\leq k_0$ which implies $\|\theta\|_0\leq k_0$. Besides, we also have $\|\theta\|_{\infty}\leq  \underline{c}^{(t)} \sigma\sqrt{\frac{\log(2p/\alpha)}{n}}$. 

Since $\max_{l\in \cL_0} \omega_l \leq  s = \sqrt{\log(ek_0/\sqrt{p})}\lor 1 \leq (\underline{c}^{(t)})^{-1}\sqrt{\frac{n}{\log(2p/\alpha)}}$ for $n \geq 9\underline{c}^{(t)2}\log^2(ep/\alpha)$, it also follows from Assumption $\bA[\alpha\wedge\delta]$ that 
\beq\label{eq:upper_theta_infty_h0}
\frac{\|\theta\|_{\infty}}{\sigma}\leq \frac{1}{s}\wedge \left(\min_{l\in \cL_0} \frac{1}{\omega_l}\right).
\eeq
Thus, we are in position to apply Lemma~\ref{lem:esperance_fourier_statistics}.  As explained in Section \ref{ss:fourier}, we have  $g(0)=0$ and $g(x)\in [0,1]$, it follows from that Lemma  that $\E^{(3)}_{\theta^*,\sigma}\big[Z_f\big|(\|Y\|_2,\theta)\big]\leq k_0 + s^2/5$. Also since $1-\Psi_{l}(0)=0$ and  $1-\Psi_{l}(x)\in [0,1+\frac{l}{k_0}]$ we have
\[
 \E^{(3)}_{\theta^*,\sigma}\big[V[r_l,\omega_l)\big|(\|Y\|_2,\theta)\big]\leq k_0 + l+ \frac{e^{1/2}}{2}\omega_l^2\ . 
\]
Then, we apply the deviation inequalities of Lemma \ref{lem:concentration_fourier_statistics} and integrate them with respect to $\|Y\|_2$ to conclude that 
\[\P_{\theta^*,\sigma}^{(3)}\big[Z_f \geq k_0 + v_{\alpha}^{(f)}|\theta\big]\leq \alpha/2\  ;\]
\[
 \sum_{l\in \cL_0}\P_{\theta^*,\sigma}^{(3)}\big[V[r_l,\omega_l] \geq k_0 + l+ v_{\alpha,l}^{(i)}|\theta\big]\leq \alpha/2\  . 
\]
 Taking the  probability of the event $\cB$ into account, we conclude that the type I error probability of both tests  is bounded by $\alpha+\delta$.

\bigskip 

\noindent
\underline{Analysis of the tests under the alternative hypothesis}.  Since $\|\theta^*\|_0$ is not too large, the assumptions of Lemma~\ref{lem:thresh} are fulfilled. As under the null hypothesis,  we have  $\|\theta^* - \widetilde{\theta}_{\bI}\|_{\infty} \leq \underline{c}^{(t)} \sigma\sqrt{\log(2p/\alpha)/n}$ with  probability higher than $1-\delta-\alpha/2$ and we still work conditionally to this event called $\cB$. 
 If $(\widetilde{\theta}_{\bI})_{(k_0+1)}\geq  \underline{c}^{(t)} \sigma  \sqrt{\log(2p/\alpha)/n}$, then both tests reject the null hypothesis, so that we can assume henceforth that $(\overline{\theta}_{\bI})_{k_0+1}=0$.

Since \eqref{eq:upper_theta_infty_h0} is still valid, we are in position to apply again Lemmas \ref{lem:esperance_fourier_statistics} and \ref{lem:concentration_fourier_statistics}. Hence, conditionally on $\theta$ and $\|Y\|_2$, we have 
\[
Z_f \geq \|\overline{\theta}_{\bI}\|_0+ \sum_{i=1}^p g(sv_i)\1_{(\overline{\theta}_{\bI})_i=0} - \frac{s^2}{5} - s e^{s^2/2}\sqrt{2p\log(\frac{2}{\alpha})}\ , 
\]
with probability higher than $1-\alpha/2$.  Define $\tilde{v}$ by 
$$
\tilde{v}_i = +\infty~~\mathrm{if}~~~(\overline{\theta}_I)_i \neq 0~~~~~\mathrm{and}~~~~~\tilde v_i =v_i ~~~~\mathrm{if}~~~(\overline{\theta}_I)_i= 0\ .
$$
Recall that $\lim_{x\rightarrow +\infty}g(x)= 1$ and $\forall x \in \mathbb R$, $0 \leq g(x) \leq 1$ (see~\cite{carpentier2017adaptive}). So it holds that
\beq \label{eq:lower_Z_f}
Z_f \geq \sum_{i=1}^p g(s \tilde v_i) - \frac{s^2}{5} - s e^{s^2/2}\sqrt{2p\log(\frac{2}{\alpha})}\ . 
\eeq
 Also, for any $l\in \cL_0$, we have  
\[
 V[r_l,\omega_l]\geq \|\overline{\theta}_{\bI}\|_0+ \sum_{i=1}^{p}\left(1-\Psi_{l}(\tilde v_i)\right)\1_{(\overline{\theta}_{\bI})_i=0}  - 
\frac{\sqrt{e}}{2}\omega_l^2 -\sqrt{2p^{1/2}l\log\left(\frac{2}{\alpha}\right)}\ ,
\]
with probability larger than $1-\alpha/2$. As above, we have $\lim_{x\rightarrow \infty}\Psi_l(x)=0$, and so 
\beq\label{eq:lower_V_r}
 V[r_l,\omega_l]\geq\sum_{i=1}^{p}\left(1-\Psi_{l}(\tilde v_i)\right)  - 
\frac{\sqrt{e}}{2}\omega_l^2 -\sqrt{2p^{1/2}l\log\left(\frac{2}{\alpha}\right)}\ ,
\eeq

In the sequel, we show that \eqref{eq:lower_Z_f} and \eqref{eq:lower_V_r} imply the desired type II error probability bounds.

\noindent
{\bf Case 1}: Analysis of \eqref{eq:lower_Z_f} for $\phi^{(f)}$. Write $\overline{s}= \sqrt{\log(e\frac{k_0^2}{p})}\lor 1$ the tuning parameter used in \cite{carpentier2017adaptive} for the corresponding test in the Gaussian sequence model. Note 
that $s\geq \overline{s}/\sqrt{2}$, $se^{s^2/2}\leq 2\overline{s}^{-1}e^{\overline{s}^2/2}$, and $s^2/5 \leq 
se^{s^2/2}$. We have shown in the proof of Proposition 2 in \cite{carpentier2017adaptive} that for a vector $x\in \mathbb{R}^p$ and any $\alpha \in (0,1)$
\[
 \sum_{i=1}^p g(\overline{s}x_i)\geq k_0 + \overline{s}^{-1}e^{\overline{s}^2/2}2\sqrt{8p\log(2/\alpha)}\ ,
\]
as soon as one of the two following condition holds for constants $c_{\alpha},c'_{\alpha},c''_{\alpha}$ positive and large enough, depending only on $\alpha$
\beqn 
\big|x_{(k_0+q)}\big|&\geq &c_{\alpha} \sqrt{\frac{k_0}{q\log\left(1+k_0/\sqrt{p}\right)}}\ , \text{ for some }q\geq c'_{\alpha} \frac{k_0}{\sqrt{\log\left(1+\frac{k_0^2}{p}\right)}}\ ;\\
\sum_{i=1}^p \left[x_i^2\wedge \overline{s}^{-1}\right]&\geq& c''_{\alpha}\frac{1}{\log\left(1+k_0/\sqrt{p}\right)}\ .
\eeqn 
It then follows from \eqref{eq:lower_Z_f}, that, given $\theta$ and $\|Y\|_2^2$ satisfying $\cB$, the test rejects the null with probability higher than $1-\alpha/2$ if  
\begin{eqnarray}
 \big| \tilde v_{(k_0+q)}\big|&\geq &c_{\alpha} \sqrt{\frac{k_0}{q\log\left(1+k_0/\sqrt{p}\right)}}\ , \text{ for some }q\geq c'_{\alpha} \frac{k_0}{\sqrt{\log\left(1+\frac{k_0^2}{p}\right)}}\ ;\label{eq:rejection1}\\
\sum_{i=1}^p \left[\tilde  v_i^2\wedge \overline{s}^{-1}\right]&\geq& c''_{\alpha}\frac{1}{\log\left(1+k_0/\sqrt{p}\right)}\ .\label{eq:rejection2}
\end{eqnarray}
Recall that, for $i\notin \cS(\overline{\theta}_{\bI})$,  $\tilde v_i = v_i = \theta_i \|Y\|_2/(\sigma^2+\|\theta\|_2^2)= \theta^*_i \|Y\|_2/(\sigma^2+\|\theta\|_2^2)$ . Since $\|Y\|_2^2/(\sigma^2+\|\theta\|_2^2)$ follows a $\chi^2$ distribution with $n/3$ degrees of freedom, we have $\|Y\|_2^2 \geq n(\sigma^2+\|\theta\|_2^2)/6$ with probability higher than $1-e^{-n/27}$ (see Lemma \ref{lem:conchi}). This implies that for any $i=1,\ldots, p$, we have 
\[
 |\tilde v_i| \geq |\theta^*_i| \sqrt{\frac{n}{6(\sigma^2+\|\theta\|_2^2)} }.
\]
\begin{lem}\label{lem:lower_theta}
Assume that the event $\cB$ holds, that $d_2^2(\theta^*;\bbB_0[k_0])\leq \sigma^2$, and that $(\overline{\theta}_{\bI})_{k_0+1}=0$. We have $\|\theta\|_2^2\leq c'\sigma^2$ .
\end{lem}
As a consequence, on the intersection of $\cB$ and an event of probability higher than $1-e^{-n/27}$, we have 
\[
 | \tilde v_i|  \geq c \frac{\sqrt{n} }{\sigma }|\theta_i^*|.
\]
Together with \eqref{eq:rejection1} and \eqref{eq:rejection2}, we have characterized the type II error probability of $\phi^{(f)}$.

\medskip 

\noindent 
{\bf Case 2}: Analysis of \eqref{eq:lower_V_r} for  $\phi^{(i)}$.
Observe $\sqrt{e}\omega_l^2/2$ is at most of the order of $\log(p)$ and is therefore negligible compared to $\sqrt{p^{1/2}l}$. 
We have shown in the proof of Proposition 3 in \cite{carpentier2017adaptive} that, for a vector $x\in \mathbb{R}^p$, and for any $\alpha$ in $(0,1)$ we have 
\[
 \sum_{i=1}^n \Psi_{l}(x_i)\geq k_0+ 2l + 2\sqrt{2lp^{1/2}}\sqrt{\log\left(\frac{\pi^2[1+ \log_2\left(l/l_0\right)]}{3\alpha}\right)},
\]
for some $l\in \cL_0$, if for constants $c_{\alpha},c'_{\alpha}$ positive and large enough, depending only on $\alpha$
\beq\label{eq:condition_x}
 |x_{(k_0+q)}|\geq  c_{\alpha} \frac{1+\log\left(\frac{k_0}{q\wedge k_0}\right)}{\sqrt{\log\left(1+\frac{k_0}{\sqrt{p}}\right)}}\ , \text{ for  some }q\geq c'_{\alpha}k_0^{4/5}p^{1/10}\ . 
\eeq
Actually, in Proposition 3 in \cite{carpentier2017adaptive}, we had considered a wider range of $q$'s as the collection $\cL_0$ was slightly larger, but this does not change the arguments here. 
In our setting, Condition \eqref{eq:condition_x} and \eqref{eq:lower_V_r} imply that $V[r_l,\omega_l] \geq k_0 + l+ v_{\alpha,l}^{(i)}$ for some $l\in \cL_0$ if  
\beq\label{eq:condition_v3}
|\tilde v_{(k_0+q)}|\geq  c_{\alpha} \frac{1+\log\left(\frac{k_0}{q\wedge k_0}\right)}{\sqrt{\log\left(1+\frac{k_0}{\sqrt{p}}\right)}}\ , \text{ for  some }q\geq c'_{\alpha}k_0^{4/5}p^{1/10}\ . 
\eeq
Then, arguing as in Case 1, we have $|\tilde v_i| \geq  c'|\theta_i|/\sigma$ on the intersection of $\cB$ and an event of probability higher than $1-e^{-n/27}$. Putting everything together, we have controlled the type II error probability of $\phi^{(i)}$.

\begin{proof}[Proof of Lemma \ref{lem:esperance_fourier_statistics}]

In view of the conditional distribution of $W_j$ given $Y$, one has
\beqn 
\E_{\theta^*,\sigma}^{(3)}[\varphi(s;W_j/\|\ol{Y}^{(3)}\|_2)|(\|Y\|_2,\theta)]&=&\int_{-1}^1 (1-|\xi|)\cos\Big[\xi \frac{s \theta_j \|Y\|_2}{\sigma^2+ \|\theta\|_2^2}  \Big] \exp\Big[\xi^2  \frac{s^2 \theta_j^2}{2(\sigma^2 + \|\theta\|_2^2)}\Big]d\xi\\ & =& g(s v_j) +  \int_{-1}^1 (1-|\xi|)\cos\Big[\xi \frac{s \theta_j \|Y\|_2}{\sigma^2+ \|\theta\|_2^2}  \Big] \left(\exp\Big[\xi^2  \frac{s^2 \theta_j^2}{2(\sigma^2 + \|\theta\|_2^2)}\Big]-1 \right)d\xi\ .
\eeqn 
 Since $s\|\theta\|_{\infty}\leq \sigma$, the remainder term is (in absolute value) less than 
\[
2\int_{0}^{1}(1-\xi)(e^{\tfrac{\xi s^2\theta_j^2}{2\sigma^2}}-1)d\xi\leq 2e^{1/2}\int_{0}^{1}(1-\xi)\xi^2\frac{s^2\theta_j^2}{2(\sigma^2+\|\theta\|_2^2)}d\xi\leq \frac{s^2\theta_j^2}{5(\sigma^2+\|\theta\|_2^2)} \ . 
\]
 Summing over all $j=1,\ldots, p$ such that $(\overline{\theta}_{\bI})_j=0$, we obtain the first result of Lemma \ref{lem:esperance_fourier_statistics}. 
 Turning to $V[r_l,\omega_l]$, we have 
 \beqn 
\E_{\theta^*,\sigma}^{(3)}\left[\eta_{r_l,\omega_l}(\frac{W_j}{\|Y\|_2})|(\|Y\|_2,\theta)\right]&=&\frac{r_l}{1-2\overline{\Phi}(r_l)}\int_{-1}^1 \frac{e^{-r_l^2\xi^2/2}}{\sqrt{2\pi}}\cos\left(\xi \omega_l \frac{ \theta_j \|Y\|_2}{\sigma^2+ \|\theta\|_2^2}\right)\exp\Big[\xi^2  \frac{\omega_l^2 \theta_j^2}{2(\sigma^2 + \|\theta\|_2^2)}\Big]d\xi
\eeqn 
As a consequence, 
\beqn 
\left|\E_{\theta^*,\sigma}^{(3)}\left[\eta_{r_l,\omega_l}(\frac{W_j}{\|Y\|_2})|(\|Y\|_2,\theta)\right]- \Psi_l(v_j)\right| &\leq &\frac{r_l}{1-2\overline{\Phi}(r_l)} \int_{-1}^1 \frac{e^{-r_l^2\xi^2/2}}{\sqrt{2\pi}} \left|e^{\xi^2 \omega_l^2 \frac{\theta_j^2}{2(\sigma^2+\|\theta\|_2^2)}} - 1\right|d\xi\\
&\leq & \frac{e^{1/2}r_l}{1-2\overline{\Phi}(r_l)} \int_{-1}^1 \frac{e^{-r_l^2\xi^2/2}}{\sqrt{2\pi}}  \xi^2 \omega_l^2 \frac{\theta_j^2}{2(\sigma^2+\|\theta\|_2^2)} d\xi \\
&\leq  & \frac{e^{1/2}\omega_l^2 \theta_j^2}{2(\sigma^2+\|\theta\|_2^2)}\ ,
\eeqn 
where we used the condition $\omega_l \|\theta\|_\infty\leq \sigma$ in the second line.  Summing this bound over all $j$ such that $(\overline{\theta}_{\bI})_j=0$ yields the desired result.

\end{proof}

\begin{proof}[Proof of Lemma \ref{lem:concentration_fourier_statistics}]
We shall apply the Gaussian concentration theorem~(see e.g. \cite{book_concentration}) to both $Z_f$ and $V[r_l,\omega_l]$. The covariance matrix $\bGamma$ associated to the conditional distribution of $W/\|Y\|_2$ decomposes as $\bI_p -  a \frac{\theta}{\|\theta\|_2} \frac{\theta^T}{\|\theta\|_2}$ with $a=\|\theta\|_2^2/[\sigma^2 +\|\theta\|_2^2]\in [0,1)$ and in particular its operators norm is less than one. Write $\bGamma^{1/2}$ for a square-root of this matrix and let $U$ denote a standard Gaussian vector. Conditionally to $Y$, $W/\|Y\|_2$ is distributed as $v  + \bGamma^{1/2} U$. For any $u\in \mathbb{R}^p$, define 
\[\zeta(u)= \sum_{j=1}^p\1_{(\overline{\theta}_{\bI})_j=0}\int_{-1}^{+1}(1-|\xi|) \cos\big(\xi s [\mu_j + (\bGamma^{1/2}  u)_j]\big)e^{ \xi^2 s^2/2}d\xi\ .\]
Given two vectors $u$ and $u'$, one has 
\beqn 
\big|\zeta(u)-\zeta(u')\big|&\leq& \sum_{j=1}^p\1_{(\overline{\theta}_{\bI})_j=0}\int_{-1}^{+1}(1-|\xi|) \Big|\cos\big(\xi s [\mu_j + (\bGamma^{1/2} u)_j]-\cos\big(\xi s [\mu_j + (\bGamma^{1/2}  u')_j]\big)\Big|e^{ \xi^2 s^2/2}d\xi\\
&\leq & s e^{s^2/2}\sum_{j=1}^p |\bGamma^{1/2}(u-u')_j|= se^{s^2/2}\|\bGamma^{1/2}(u-u')\|_1\\
&\leq& se^{s^2/2}\sqrt{p}\|(u-u')\|_2\  ,
\eeqn 
since the cosinus function is $1$-Lipschitz. As a consequence, the function $u\mapsto Z(u)$ is $se^{s^2/2}\sqrt{p}$-Lipschitz. The deviation inequalities \eqref{eq:concentation_Z_f} then follow from the Gaussian concentration theorem~(see e.g. \cite{book_concentration}).

As for $V[r_l,\omega_l]$, we argue similarly that, for $\omega_l > r_l$, it is conditionally distributed as a Lipschitz function of a standard Gaussian vector with Lipschitz constant
\[
 \frac{r_l\omega_l}{1-2\overline{\Phi}(r_l)}\int_{-1}^{1}e^{\xi^2 (\omega_l^2-r_l^2)/2}\frac{\sqrt{p}}{\sqrt{2\pi}}d\xi\leq 
 \sqrt{\frac{8p}{\pi}}~\frac{r_l\omega_l}{(1-2\overline{\Phi}(r_l))(\omega_l^2-r_l^2)}e^{ (\omega_l^2-r_l^2)/2}\ .
\]
Since $l \geq k_0^{4/5}p^{1/10}$, we have $\omega_l^2-r_l^2\geq 2\omega_l$ for any $l\in \cL_0$ and the Lipschitz constant is therefore less than 
\[
 \frac{r_l}{(1-2\overline{\Phi}(r_l))}\sqrt{\frac{2p}{\pi}}~e^{ (\omega_l^2-r_l^2)/2}\leq p^{1/4}l^{1/2}\ ,
\] 
where the last inequality is a consequence of the definition of $r_l$ and $\omega_l$ and is detailed in the proof of Lemma 6 in \cite{carpentier2017adaptive}.
 
\end{proof}

\begin{proof}[Proof of Lemma \ref{lem:lower_theta}]
Under $\cB$, we have $\|\theta^* - \widetilde{\theta}_{\bI}\|_{\infty} \leq \underline{c}^{(t)} \sigma\sqrt{\log(2p/\alpha)/n}$. Hence, 
\beqn 
\|\theta\|_2^2& =  &\sum_{i:\overline{\theta}_{\bI}\neq 0} (\theta^*_i - (\overline{\theta}_{\bI})_i)^2 + \sum_{i:\ \overline{\theta}_{\bI}= 0} \theta_i^{*2}\\
&\leq & k_0 \underline{c}^{(t)2} \sigma^2 \frac{\log(2p/\alpha)}{n} +  \sum \theta_i^{*2}\1_{|\theta_i^{*}|\leq 2 \underline{c}^{(t)} \sigma\sqrt{\log(2p/\alpha)/n}}\\
&\leq & 5 k_0 \underline{c}^{(t)2} \sigma^2 \frac{\log(2p/\alpha)}{n}+ \sigma^2 \leq c'\sigma^2\ ,
\eeqn 
where we used in the second line the definition of $\overline{\theta}_{\bI}$ and $(\overline{\theta}_{\bI})_{k_0+1}=0$ and we used $d_2^2(\theta^*;\bbB_0[k_0])\leq \sigma^2$ together with Assumption $(\bA[\alpha\wedge\delta])$ in the last line. 

\end{proof}

\subsection{Proof of Theorem~\ref{thm:ubkvkd}}

Consider any $\theta\in \bbB_0[k_0]$. In view of  Propositions~\ref{prp:chi}--\ref{prp:analyse_Z_i}, we can bound the rejection probability as follows
\beq \label{eq:first_type_I_error_bound}
\mathbb{P}_{\theta,\sigma}[\phi^{(ag)}=1]\leq 4(\delta+ \alpha)+ \mathbb{P}_{\theta,\sigma}[d_2^2(\widehat{\theta}_{SL};\bbB_0[k_0])\geq \sigma^2/2] 
\eeq 
Since, under the null hypothesis, $\theta^*$ is $k_0$-sparse, we have 
\[ d_2^2(\widehat{\theta}_{SL};\bbB_0[k_0]) \leq  \|\theta^*-\widehat{\theta}_{SL}\|_2^2\ .\]
Applying Lemma \ref{lem:square_root_Lasso}, we derive that, with probability higher than $1-\delta$, 
$d_2^2(\widehat{\theta}_{SL};\bbB_0[k_0])\leq \underline{c}_{1}^{SL}\sigma^2 \frac{k_0}{n}\log(p/\delta)$. Thus, by Condition $\bA[\alpha\wedge \delta]$, we have $\mathbb{P}_{\theta,\sigma}[d_2^2(\widehat{\theta}_{SL};\bbB_0[k_0]\geq \sigma^2/2] \leq \delta$. From \eqref{eq:first_type_I_error_bound}, we derive that 
$\mathbb{P}_{\theta,\sigma}[\phi^{(ag)}=1]\leq 5\delta+4\alpha$. Looking more closely at the proof of Propositions~\ref{prp:chi}--\ref{prp:analyse_Z_i}, we observe that each occurrence of the probability $\delta$ corresponds to the same control of 
the square-root Lasso estimator $\widehat{\theta}_{SL}$.  As a consequence $\phi^{(ag)}$ satisfies (${\bf P}_1[\delta+4\alpha]$). 
Turning to the Type II error, we fix $\Delta\leq p-k_0$ and assume that $\theta^*\in \bbB_0[k_0+\Delta]$. 

\medskip 

\noindent 
{\bf Case 1}: $\Delta \leq c n/\log(p/\delta)$. If $k_0\leq p^{1/2 - \varsigma}$, then the squared separation distance 
$\min[\Delta\log(p)/n, 1/\sqrt{n}+ k_0\log(p)/n]$ in \eqref{eq:upper_lbkvkd} is a consequence of Propositions \ref{prp:t} and \ref{prp:chi} and is achieved by the combination of  $\phi^{(t)}$ and $\phi^{(\chi)}$. If $k_0\geq p^{1/2+\varsigma}$, the squared separation distance $\Delta\log(p)/n$ is still achieved by $\phi^{(t)}$. To prove the last part of the result,
 let us assume that $\theta^*$ is such that  $\max(\phi^{(t)}, \phi^{(\chi)}, \phi^{(f)}, \phi^{(i)})$ does not reject the null with high probability. We shall prove that this implies $d_2^2\left(\theta^*; \bbB_0[k_0]\right)\leq c_{\alpha,\varsigma} \sigma^2 k_0/[\log(p)n]$. From Proposition \ref{prp:chi}, we have $d_2^2(\theta^*; \bbB_0[k_0])\leq \sigma^2$. In view of Proposition \ref{prp:analyse_Z_f}, we have 
\[
 \sum_{i}\theta_i^{*2}\wedge \frac{\sigma^2 }{n\log(p)}\leq c_{\alpha,\varsigma}\sigma^2 \frac{k_0}{n\log(p)}\ .
\]
In view of Proposition \ref{prp:analyse_Z_i}, we have 
\[
 |\theta_{(k_0+q)}^*| \leq c_{\alpha,\varsigma} \sigma \frac{1+\log(\frac{k_0}{q\wedge k_0})}{\sqrt{n\log(p)}}\ ,
\]
for all $q\geq c'_{\alpha} k_{0}^{4/5}p^{1/10}$. Finally, Proposition \ref{prp:t} enforces that 
\[
 |\theta_{(k_0+q)}^*| \leq c_{\alpha,\varsigma} \sigma \sqrt{\frac{\log(p)}{n}}\ ,
\]
for all $q < c'_{\alpha} k_{0}^{4/5}p^{1/10}$. Putting everything together, we obtain
\beqn 
d_2^2\left(\theta^*; \bbB_0[k_0]\right)&\leq& \sum_{q=1}^{\Delta} \theta^{*2}_{(k_0+q)}\leq c_{\alpha,\varsigma} \sigma^2 \left[\frac{k_0}{n\log(p)}+ \sum_{q=k_0/\log^2(p)}^{k_0}\frac{1+\log^2(\frac{k_0}{q\wedge k_0})}{n\log(p)}+ \sum_{q=k_0}^{\Delta} \theta_i^{*2} \right]\\ 
&\leq& c'_{\alpha,\varsigma} \sigma^2 \frac{k_0}{n\log(p)}\ ,
\eeqn 
where we used that, for $q\geq k_0$, $|\theta_{(k_0+q)}^*|$ is small compared to $\sigma/\sqrt{n\log(p)}$. This concludes the proof for Case 1.

\medskip

\noindent 
{\bf Case 2}: $\Delta \geq c n/\log(p/\delta)$. In that case, $\Delta \log(p)/n$ is larger than $k_0\log(p)/n+n^{-1/2}$ and the first result in \eqref{eq:upper_lbkvkd} is a consequence of the analysis of $\phi^{(\chi)}$ in  Proposition \ref{prp:chi}. We now turn to the case  $k_0 \geq p^{1/2+\varsigma}$ and we need to prove that the squared separation distance is less than $c_{\alpha,\delta} \sigma^2 k_0/[n\log(p)]$. If $d_2^2(\theta^*;\bbB_0[k_0])\geq \sigma^2$, then $\phi^{(\chi)}$ rejects the null hypothesis with high probability. Thus, we can assume that $d_2^2(\theta^*;\bbB_0[k_0])\leq \sigma^2$. Also, we can assume that $\|\widehat{\theta}_{SL}-\widetilde{\theta}_{SL,k_0}\|_2^2 \leq \sigma^2/2$, otherwise the test $\phi^{(ag)}$ rejects the null. Finally, we can assume that $\|\theta^*-\widetilde{\theta}_{SL,k_0}\|_2^2 \leq \sigma^2/2$, otherwise the test $\phi^{(\chi)}$ also rejects the null with high probability. By triangular inequality, $\theta^*$ therefore satisfies  $\|\theta^*-\widehat{\theta}_{SL}\|_2^2 \leq 2\sigma^2$ and we are in position to apply Lemma \ref{lem:thresh}, which implies 
\beq \label{eq:upper_linfty_1}
 \|\theta^*-\widetilde{\theta}_{\bI}\|_\infty \leq c_1 \sigma \sqrt{\frac{\log(2p/\alpha)}{n}} \ , 
\eeq
with probability higher than $1-\alpha/2$ conditionally to $\widehat{\theta}_{SL}$. As a consequence, the event $\cB$ involved in the proof of Propositions \ref{prp:analyse_Z_f} and \ref{prp:analyse_Z_i} is true. As ensuring this event is the only occurrence in the proof of these propositions where the restrictions $\|\theta^*\|_0\leq cn/\log(p/\delta)$ is needed, we conclude that, given $\cB$, $\max(\phi^{(f)},\phi^{(i)})$ rejects the null with probability higher than $1-\alpha/2$ if any of the conditions  \eqref{eq:separation_log}, \eqref{eq:separation_log_l2}, or  \eqref{eq:separation_log_intermediairy} is satisfied. Similarly, Condition \eqref{eq:upper_linfty_1} (with $\alpha/2$ replaced by $\alpha)$ allows to adapt the proof of Proposition~\ref{prp:t} without the restriction on $\|\theta^*\|_0$. Thus,   $\phi^{(t)}$ rejects the null with conditional probability higher than $1-\alpha$ under \eqref{eq:cond_P_separation}. 

Arguing as Case 1, we conclude that the aggregated test rejects the null with high probability if $d_2^2(\theta^*;\bbB_0[k_0])$ is large compared to $\sigma^2 k_0/[n\log(p)]$.

\subsection{Proof of Theorem \ref{thm:puissance_thres}}

Let $d$ denote any positive integer.  Let $S \subset \{1, \ldots, d\}$. 
For $u\in \mathbb{R}^d$, we write $u_{S}= (u_i \1_{i\in S})_i$ for the vector in $\mathbb{R}^d$ whose values outside $S$ have been set to $0$.

These notation are also extended to matrices.  Given $r$ a positive integer and a $r \times d$ matrix $\mathbf M$, we write $\mathbf M_{S}$ for the $r\times d$  matrix defined by $(\mathbf M_{S})_{i \leq r,j \leq d} = (\mathbf M_{i,j} \1_{\{j \in S\}})_{i \leq r,j \leq d}$. 
For $R \subset \{1, \ldots, r\}$, we also write $\mathbf M_{R, S}$ for the $r \times d$-dimensional matrix such that $(\mathbf M_{R,S})_{i\leq r,j \leq d} = (\mathbf M_{i,j} \mathbf 1\{i\in R, j \in S\})_{i\leq r,j \leq d}$.

\subsubsection{Proof of Theorem \ref{thm:puissance_thres}}

Let $\delta > 0$ and consider any subset $S$ satisfying the property ({${\bf S}[\ba_1,\ba_2,\ba_3]$).

\begin{lem}\label{lem:eigv} The exists a constant $c$ such that the following holds for all $\delta>0$. If  
\[
c\left[ \ba_2\|\theta^*\|_0+ \log\left(\frac{4}{\delta}\right)\right] \leq m\ ,
\]
there exists an event $\cB_1$ of probability higher than $1-\delta/2$ such that 
\begin{eqnarray}
\label{eq:upper:restricted_eigen_A0}
(2\eta)^{-1}\leq \frac{1}{m}\lambda_{\min,S}(\bX^{(0)T}_{S}\bX^{(0)}_{S}) &\leq &\frac{1}{m}\lambda_{\max,S}(\bX^{(0)T}_{S}\bX^{(0)}_{S}) \leq 2\eta\ \ ;\\
\label{eq:upper:restricted_eigen_A02}
\Big\|\frac{1}{m}\bX^{(0)T}_{S}\bX^{(0)}_{\overline{S}}\theta^*_{\overline{S}}\Big\|_2 &\leq &2\eta\sqrt{\log(4/\delta)} \|\theta^*_{\overline{S}}\|_2\ ,
\end{eqnarray}
where $\lambda_{\min, S}$ and $\lambda_{\max, S}$ respectively refer to the smallest and largest eigenvalue of a matrix restricted to its coordinates in $S\times S$. 
\end{lem}

So on the event $\cB_1$ defined above, the matrix $\bX^{(0)T}_{S}\bX^{(0)}_{S}$ restricted to its coordinates in $S\times S$ is non-singular. Recall that the matrix $\bX^{(0)T}_{S}\bX^{(0)}_{S}$  is $0$ outside $S\times S$. Nevertheless, we can define its pseudo-inverse  $(\bX^{(0)T}_{S}\bX^{(0)}_{S})^{-1}$  by considering its inverse when restricted to $S\times S$ and fixing all its remaining entries to 0. 
The restricted least-squares estimator $\widehat{\theta}_{ls,S}$ is then conditionally distributed as follows 
\beq\label{eq:distribution_Theta_RLS}
\big(\widehat{\theta}_{ls,S}|\bX^{(0)}\big)\sim \mathcal N(\theta_{S}^* + (\bX^{(0)T}_{S}\bX^{(0)}_{S})^{-1}\bX^{(0)T}_{S} \bX^{(0)}_{\overline{S}} \theta^*_{\overline{S}}, \sigma^2 (\bX^{(0)T}_{S}\bX^{(0)}_{S})^{-1}).
\eeq
Define the bias $B= \theta^*-\E^{(0)}[\widehat{\theta}_{ls,S}|\bX^{(0)}]$.
On the event $\cB_1$, it follows from the definition in Equation~\eqref{eq:upper_CS} of $\bS[\ba_1,\ba_2,\ba_3]$ and Lemma~\ref{lem:eigv} that 
\beq\label{eq:imp2}
\|B\|_2^2 \leq 16\eta^4\log(4/\delta) \|\theta^*_{\overline{S}}\|_2^2+\|\theta_{\overline{S}}^*\|_2^2 \leq 17\eta^4\log(4/\delta) \|\theta_{\overline{S}}^*\|_2^2\leq 17\eta^4\log(4/\delta) \ba^2_3 M(\ba_1,\frac{\theta^*}{\sigma})\frac{ \log(p)}{m}\sigma^2\ .
\eeq
Next, since $\widehat{\theta}_{ls,S}$ follows a normal distribution~\eqref{eq:distribution_Theta_RLS}, we can easily bound its deviations. In particular, we deduce from \eqref{eq:upper:restricted_eigen_A0}  that there exists an event  $\cB_2$ of probability higher than $1-\delta/3$ such that on $\cB_1\cap \cB_2$, one has 
\beq \label{eq:imp1_1}
\big|(\widehat{\theta}_{ls,S})_i - \theta^*_i\big| \leq \sigma\sqrt{2\eta\frac{\log(6p/\delta)}{m}} + |B_i|,\quad  \text{for}~~i \in S~~~~\text{and}~~~~\big|(\widehat{\theta}_{ls,S})_i - \theta^*_i\big|= |\theta^{*}_i| = |B_i|~~~\text{otherwise}\ .
\eeq

\begin{lem}\label{lem:variance_estimation}
 Assume that $\log(6/\delta)\leq c n$. There exists an event $\cB_3$ of probability higher than $1-\delta/6$ such that on $\cB=\cap_{i=1}^3\cB_i$, we have 
\[
   \frac{\sigma^2}{2} \leq \widehat{\sigma}^2_S \leq 2 \sigma^2 \left[1+\eta  \ba_3^2 M(\ba_1,\frac{\theta^*}{\sigma})\frac{\log(p)}{m}\right] 
\]
\end{lem}
Putting everything together, we derive that, under $\cB$, one has
\[
\frac{|(\widehat{\theta}_{ls,S})_i - \theta^*_i|}{\widehat{\sigma}_S} \leq 2\sqrt{\eta\frac{\log(6p/\delta)}{m}} +  \sqrt{2}\frac{|B_i|}{\sigma} \ , \quad i=1,\ldots, p\ . 
\]
 This implies that, for all $i=1,\ldots, p$, 
\beq \label{eq:imp1''}
\frac{|\theta^*_i|}{\sqrt{2}\sigma[1+ \eta  \ba_3^2 M(\ba_1,\frac{\theta^*}{\sigma})\frac{\log(p)}{m}]^{1/2}} -   2\sqrt{ \eta}\frac{\log(6p/\delta)}{m} - \sqrt{2}\frac{|B_i|}{\sigma} \leq \frac{|(\widehat{\theta}_{ls,S})_i| }{\widehat{\sigma}_S}  \leq  \sqrt{ 2}\frac{|\theta^*_i|}{\sigma} +    2\sqrt{\eta\frac{\log(6p/\delta)}{m}} + \sqrt{2} \frac{|B_i|}{\sigma}\ .
\eeq

\bigskip

\noindent 
\underline{Under the null hypothesis}. Suppose that $\theta^*\in \bbB_0[k_0]$. 
Note that \eqref{eq:definition_c_*} implies that 
\[\underline{c}_*\geq   \sqrt{2}\ba_1+ t_1+ 2\sqrt{\eta \frac{\log(6p/\delta)}{\log(p)}}\text{ with }t_1= 2\eta^2\ba_3\sqrt{17 \log(\frac{4}{\delta})}\ .\]

Assume that $\theta^*$ belongs to $\bbB_0[k_0]$. From \eqref{eq:imp1''}, we deduce that, 
conditionally on the event $\cB$, one has 
\beqn
 N\left[\underline{c}_*;\widehat{\theta}_{ls,S}/\widehat{\sigma}_S\right]&\leq& \left|\left\{i:\, \frac{ \sqrt{2} |\theta^*_i\big|}{\sigma}\geq \sqrt{2}\ba_1\sqrt{\frac{\log(p)}{m}}\right\}\right|+ \left|\left\{i:\, \frac{\sqrt{2}|B_i|}{\sigma}\geq t_1\sqrt{\frac{\log(p)}{m}}\right\}\right| \\
 &\leq & k_0 - M[\ba_1,\frac{\theta^*}{\sigma}]+ \frac{ 2\|B\|_2^2m}{\sigma^2 t_1^2\log(p)} \\
 &\stackrel{\eqref{eq:imp2}}{\leq }& k_0 -  M[\ba_1,\frac{\theta^*}{\sigma}]+ \frac{34\eta^4}{t_1^2}\log(\frac{4}{\delta}) \ba_3^2 M(\ba_1,\frac{\theta^*}{\sigma})  \leq k_0\ .
 \eeqn 
As a consequence, the test accepts the null hypothesis under the event $\cB$. 
 
 \medskip 
 
\noindent 
\underline{Under the alternative hypothesis}. We now assume that $\theta^*$ belongs to $\bbB_0[k_0+\Delta]$ and satisfies 
\beq\label{eq:condition_alternative}
d_2^2(\theta^* ; \bbB_0[k_0]) \geq \underline{t}^2\sigma^2\frac{\Delta \log(p)}{m}\, \,\quad \text{ with  }\quad \underline{t} =  4\sqrt{2}\left[1+ \eta  \ba_3^2 (k_0+\Delta)\frac{\log(p)}{m}\right]^{1/2} \left[\underline{c}_* \lor \Big( 5\eta^2 \sqrt{\log(4/\delta)}\ba_3\right]\ .
\eeq
Consider the set  $T= \big\{i, |\theta^*_i|\geq \sigma \underline{t} \sqrt{\frac{\log(p)}{2m}}\big\}$ of large  coordinates of $\theta^*$. In view of $d_2^2(\theta^* ; \bbB_0[k_0])$, we have 
\begin{equation}\label{eq:setset2}
d_2^2(\theta^*_{T} , \bbB_0[k_0]) \geq \underline{t}^2\sigma^2\frac{\Delta \log(p)}{2m}\ .
\end{equation}
On the event $\cB$, it follows from~\eqref{eq:imp1''} and the definition of $\underline{c}_*$ that $|(\widehat{\theta}_{ls,S})_i|/\widehat{\sigma}_S\geq \underline{c}_* \sqrt{ \log(p)/m}$ if 
\[
\frac{|\theta_i^*|}{\sigma}\geq 4\sqrt{2}[1+ \eta  \ba_3^2 M(\ba_1,\frac{\theta^*}{\sigma})\frac{\log(p)}{m}]^{1/2} \underline{c}_* \sqrt{\frac{\log(p)}{m}}\quad \text{ and }\quad |B_i|\leq \frac{|\theta_i^*|}{4[1+ \eta  \ba_3^2 M(\ba_1,\frac{\theta^*}{\sigma})\frac{\log(p)}{m}]^{1/2}}\ .  
\]
Observe that $\underline{t}\geq 4\sqrt{2}[1+ \eta  \ba_3^2 M(\ba_1,\frac{\theta^*}{\sigma})\frac{\log(p)}{m}]^{1/2}\underline{c}_*$.
 Denoting $T_0= T\cap \{i:\ 4|B_i| \geq |\theta^*_i|[1+ \eta  \ba_3^2 M(\ba_1,\frac{\theta^*}{\sigma})\frac{\log(p)}{m}]^{-1/2} \}     $, we obtain that $N[\underline{c}_*;\widehat{\theta}_{ls,S}/\widehat{\sigma}_S] \geq |T| -|T_0|$. We can bound $\|\theta^*_{T_0}\|^2_2$ in terms of the bias $\|B\|_2^2$ and then use~\eqref{eq:imp2} and \eqref{eq:condition_alternative}.
\beqn
\|\theta^*_{T_0}\|^2_2&\leq& 16\big[1+ \eta  \ba_3^2 M(\ba_1,\frac{\theta^*}{\sigma})\frac{\log(p)}{m}\big]\|B\|_2^2\\
&\leq& 272 \eta^4 \log(4/\delta)\big[1+ \eta  \ba_3^2 M(\ba_1,\frac{\theta^*}{\sigma})\frac{\log(p)}{m}\big] \ba_3^2\frac{ M[\ba_1,\frac{\theta^*}{\sigma}]\log(p)}{m} \sigma^2 <  \underline{t}^2\sigma^2\frac{ \Delta \log(p)}{2m}\ ,
\eeqn 
where the inequality $ M[\ba_1,\frac{\theta^*}{\sigma}]\leq \Delta$ is a consequence of \eqref{eq:condition_alternative} and  $\underline{c}_*\geq \sqrt{2}\ba_1$.    In view of Equation~\eqref{eq:setset2},
we have $\|\theta^*_{T_0}\|^2_2 <d_2^2(\theta^*_{T} , \bbB_0[k_0])$, which implies $|T_0|<|T|-k_0$ and therefore  $N[ \underline{c}_*;\widehat{\theta}_{ls,S}/\widehat{\sigma}]>k_0$. The test therefore rejects the null hypothesis under the event $\cB$, which concludes the proof.

\begin{proof}[Proof of Lemma~\ref{lem:eigv}] 
We first show~\eqref{eq:upper:restricted_eigen_A0}. 
\noindent
 Recall that $\bX^{(0)}$ is independent of $S$ and that the restriction of $\bN= \bSigma_{S,S}^{-1/2}\bX^{(0)T}_{S} \bX^{(0)}_{S}\bSigma_{S,S}^{-1/2}$ to $S\times S$ follows a standard Wishart distribution - all coordinates outside $S\times S$ being $0$. by $\bS[\ba_1,\ba_2,\ba_3]$, the size of  the corresponding  covariance matrix is less than $|S|\leq \ba_2\|\theta^*\|_0$. From e.g.~\cite{Davidson2001}, we deduce, on an event $\cB_{1-1}$ of probability larger than $1-\delta/4$, we have 
\begin{align*}
&\eta^{-1}\Big(1 - c_{R} \sqrt{\frac{|S|\log(4/\delta)}{m}} - c_{R}\frac{\log(4/\delta)+|S|}{m}\Big)\leq \frac{1}{m}\lambda_{\min,S}(\bN)\\ 
&\leq \frac{1}{m}\lambda_{\max,S}(\bN) \leq \eta \Big(1 + c_{R}\sqrt{\frac{|S|\log(4/\delta)}{n}} +  c_{R} \frac{|S|+ \log(4/\delta)}{m}\Big)\ ,
\end{align*}
where $c_R$ is an universal constant. Assuming that $\ba_3|\theta^*_0|+\log(4/\delta)$ is small compared to $m$, we deduce that the spectrum of $\bN$ lies in $[1/2,2]$. Thus, under $\bB_{1-1}$, the spectrum of $m^{-1}(\bX^{(0)T}_{S}\bX^{(0)}_{S})$ restricted to its coordinates $S\times S$ lies in $[\eta^{-1}/2, 2\eta]$.

\noindent
Turning to~\eqref{eq:upper:restricted_eigen_A02}, we observe that $\bX^{(0)}_{\overline{S}}\theta^*_{\overline{S}}$ follows a mean zero normal distribution. 
 Using a deviation inequality for $\chi^2$ distribution (Lemma \ref{lem:conchi}), we deduce the existence
of  an event $\cB_{1-2}$ of probability larger higher $1-\delta/4$ such that 
$ \|\frac{1}{\sqrt{m}}\bX^{(0)}_{\overline{S}}\theta^*_{\overline{S}}\|_2 \leq \|\theta^*_{\overline{S}}\|_2\sqrt{\eta}  [1+  \sqrt{2\log(4/\delta)/m}]\leq  \|\theta^*_{\overline{S}}\|_2   \sqrt{2\eta\log(4/\delta)}$, since $m$ is large enough  as assumed in Theorem \ref{thm:puissance_thres}. 
So from Equation~\eqref{eq:upper:restricted_eigen_A0}, we deduce that, on $\cB_1= \cB_{1-1} \cap \cB_{1-2}$, we have 
$$\Big\|\frac{1}{m}\bX^{(0)T}_{S}\bX^{(0)}_{\overline{S}}\theta^*_{\overline{S}}\Big\|_2 \leq \lambda^{1/2}_{\max}\Big(\frac{1}{\sqrt{m}}\bX^{(0)T}_{S}\bX^{(0)}_{S}\Big) \|\frac{1}{\sqrt{m}}\bX^{(0)}_{\overline{S}}\theta^*_{\overline{S}}\|_2 \leq 2\eta\sqrt{\log(4/\delta)} \|\theta^*_{\overline{S}}\|_2\ .$$

\end{proof}

\begin{proof}[Proof of Lemma \ref{lem:variance_estimation}]

$\widehat{\sigma}^2_S/\Var(Y|\bX_S)$ follows a $\chi^2$ distribution with $m$ degrees of freedom. Using a deviation inequality for $\chi^2$ distribution (Lemma \ref{lem:conchi}), we derive that $\widehat{\sigma}^2_S/\Var(Y|\bX_S)\in (1/2, 2)$, with probability higher than $1-e^{-cm}\geq 1-\delta/6$. 
Thus, it remains to bound $\Var(Y|\bX_S)$. From the definition of the property $\bS[\ba_1,\ba_2,\ba_3]$, we deduce that 
\[
\sigma^2\leq \Var(Y|\bX_S)\leq   \sigma^2 +\var{\bX \theta^*_{\overline{S}}}\leq  \eta \|\theta^*_{\overline{S}}\|_2^2 \leq \sigma^2 \left[1+ \eta \ba_3^2 M(\ba_1,\frac{\theta^*}{\sigma})\frac{\log(p)}{m}\right] \ . 
\]

\end{proof}

\subsubsection{Proof of Proposition~\ref{prp:estimator_MCP}}
Write $\widehat{\theta}_{MCP,N}$ for a stationary point of the MCP criterion and let $\widehat{S}_{MCP}$ denote its support. Since, we used the normalized design, we are more interested in the rescaled estimator $\widehat{\theta}_{MCP}$ defined by $(\widehat{\theta}_{MCP})_i= (\widehat{\theta}_{MCP,N})_i/\|\bX^{(1)}_{.,i}\|_2$. 
As explained in the proof of Lemma~\ref{lem:variance_estimation}, the design matrix $\bT$ satisfies, with probability higher than $1-1/p$  the compatibility  property (see~\cite{kolt11,2012_Sun}) with any set of size less than $n/[c^{(1)}_{\eta}\log(p)]$, see~\cite{10:RWG_restricted}. Besides, the restricted eigenvalue condition satisfied for sparsities of size less than $n/[c^{(1)}_{\eta}\log(p)]$ are bounded by some constants depending on $\eta$, see~\cite{Davidson2001,zhang2010nearly} with probability higher than $1-1/p$. From Lemma~\ref{lem:square_root_Lasso}, we deduce that $\widehat{\sigma}_{SL}/\sigma\in (3/4,5/4)$ with probability higher than $1-1/p$. 
We are therefore in position to apply Theorem 6 in~\cite{zhang2010nearly} and Corollary 1 in  Feng and Zhang~\cite{feng2017sorted}\footnote{Actually, our definition of MCP uses a different normalization from that in~\cite{zhang2010nearly} and~\cite{feng2017sorted} and one has therefore to translate their results in our setting. } provided that we chose the constant $\underline{c}^{(MCP)}_{\eta}$ large enough and $\underline{c}^{'(MCP)}_{\eta}$ small enough. From Theorem 6 in~\cite{zhang2010nearly} with $B= S^*$ (the support of $\theta^*$), we deduce that, with probability higher than $1-1/p$, 
\[
 |\widehat{S}_{MCP}|\leq |\widehat{S}_{MCP}\setminus S^*|+ |S^*|\leq c_{\eta}\|\theta^*\|_0\ . 
\]
Write $\widehat{\theta}^{(1)}_{ls,S^*}$ for the least-square estimator of $\theta^*$ restricted to $S^*$: 
$\widehat{\theta}^{(1)}_{ls,S^*} = \argmin_{\theta \ : \, \cS(\theta)\subset S^*} \| Y^{(1)} - \bX^{(1)}\theta\|_2^2$
as defined in Equation~\eqref{eq:definition_restricted_least_squares}. We deduce from Corollary 1 in~\cite{feng2017sorted} that 
\[
\|(\widehat{\theta}^{(1)}_{ls,S^*})_{\overline{\widehat{S}}_{MCP}}\|_2^2 \leq \|\widehat{\theta}_{MCP}-\widehat{\theta}^{(1)}_{ls,S^*}\|_2^2 \leq c''_{\eta}\sigma^2 \frac{\log(p)}{m}M(c'_{\eta},\frac{\theta^*}{\sigma})\ ,
\]
The restricted least-square estimator $\widehat{\theta}^{(1)}_{ls,S^*}$  follows a normal distributions with mean $\theta^*$ and covariance $(\bX_{S^*}^{(1)T}\bX_{S^*}^{(1)})^{-1}$ where we consider here the pseudo-inverse. The eigenvalues of $m(\bX_{S^*}^{(1)T}\bX_{S^*}^{(1)})^{-1}$ are bounded by the restricted eigenvalue condition on the design $\bX^{(1)}$. Hence, we obtain $\|\widehat{\theta}^{(1)}_{ls,S^*}-\theta^*\|_{\infty}\leq c'''_{\eta}\sigma \sqrt{\log(p)/m}$  with probability higher than $1-c/p$, from some $c'''_{\eta}>0$. This implies that $|\theta^*_i|\leq 2| (\widehat{\theta}^{(1)}_{ls,S^*})_i|$ if 
$|\theta^*_i|\geq 2c'''_{\eta}\sigma \sqrt{\log(p)/m}$. We obtain 
\beqn 
 \|\theta^*_{\overline{\widehat{S}}_{MCP}}\|_{2}^{2}&\leq& \sum_{i\in \overline{\widehat{S}}_{MCP}} \1_{|\theta^*_i|\geq 2c'''_{\eta}\sigma \sqrt{\log(p)/m}}4|(\widehat{\theta}^{(1)}_{ls,S^*})_i|^2+ \1_{|\theta^*_i|\leq 2c'''_{\eta}\sigma \sqrt{\log(p)/m}} 2\left([ (\widehat{\theta}^{(1)}_{ls,S^*})_i]^2+(c'''_{\eta})^2\sigma^2\frac{\log(p)}{m} \right)\\
 &\leq & 4 \|(\widehat{\theta}^{(1)}_{ls,S^*})_{\overline{\widehat{S}}_{MCP}}\|_2^2+ 2(c'''_{\eta})^2\sigma^2\frac{\log(p)}{m} M(c'''_{\eta},\frac{\theta^*}{\sigma})\\
 &\leq & 4[c''_{\eta}+ (c'''_{\eta})^2]\sigma^2 \frac{\log(p)}{m}M(c'_{\eta}\vee c'''_{\eta},\frac{\theta^*}{\sigma})\ .
\eeqn 
The result follows.

\subsubsection{Proof of Theorem~\ref{thm:support} }

To alleviate the notation, we simply write $S_t$ for $\widehat{S}_t$ in this proof. 
For a random vector $X\sim \cN(0,\bSigma)$, we write $\bSigma^{(t)}$ for the conditional variance of $X$ given $(X_{S_{t-1}},S_{t-1})$. Standard computations for conditional variance based on Schur complement lead to 
\beq\label{eq:definition_Sigma(t)}
\bSigma^{(t)}= \bSigma_{\overline{S}_{t-1},\overline{S}_{t-1}} - \bSigma_{\overline{S}_{t-1},S_{t-1}} (\bSigma_{S_{t-1},S_{t-1}})^{-1}\bSigma_{S_{t-1},\overline{S}_{t-1}},
\eeq 
where $(\bSigma_{S_{t-1},S_{t-1}})^{-1}$ is the pseudo-inverse of $\bSigma_{S_{t-1},S_{t-1}}$ obtained by considering its inverse when restricted to $S_{t-1} \times S_{t-1}$ and setting all its remaining entries to $0$.

In the sequel, we denote $ Y_{\perp}^{(t)}= \underline{\bPi}^{\perp}_{t, S_{t-1}}   Y^{(t)}$ and  $\bX_{\perp}^{(t)}= \underline{\bPi}^{\perp}_{t, S_{t-1}} \bX^{(t)}$. 
The following lemma ensures that the linear regression of $Y^{(t)}$ on $\bX_{\perp}^{(t)}$ involves the restriction of $\theta^*$ to $\overline{S}_{t-1}$. 

\begin{lem}\label{lem:mattransf}
Fix any $t\in [1;T]$ and consider the event such that $|S_{t-1}|< m/T$. Then, given $S_{t-1}$, the rows of $\bX_{\perp}^{(t)}$ are independent and follow a centered normal distribution with covariance matrix $\bSigma^{(t)}$. Besides, we have 
\[
 \Big(Y_{\perp}^{(t)}\big|\bX_{\perp}^{(t)}, S_{t-1}  \Big) \sim \cN\big(\bX_{\perp}^{(t)} \theta^*_{\overline{S}_{t-1}}, \sigma^2 \bI_{m/T-|S_{t-1}|}\big)\ .
\]
\end{lem}

The next lemma ensures that the population covariance matrix of the projected design still belongs to $\cU[\eta]$. 
\begin{lem}\label{lem:mistic}
For any  $\bSigma \in \mathcal U(\eta)$ and any set $S_{t-1}$, The restriction of  $\bSigma^{(t)}$ to $\overline{S}_{t-1}\times  \overline{S}_{t-1}$ belongs to  
$\mathcal U(\eta)$.
\end{lem}

\noindent 
Denote $\delta= p^{-2}$. For $1\leq t\leq T$, {\bf Property ($\bQ_t$)} is said to be satisfied if there exists an event $\xi_t$ measurable with respect to $(( \underline{Y}^{(1)},\underline{\bX}^{(1)}), \ldots, ( \underline{Y}^{(t)},\underline{\bX}^{(t)}))$ of probability higher than $(1-\delta)^{t}$  such that the three  following inequalities hold:
\beq\label{eq:condition_H_t}
|\cS(\theta^*) \setminus S_{t}| \leq \frac{\|\theta^*\|_0}{2^{t}} \vee  2M\Big[2\sqrt{c_{\eta}T\frac{\log(p/\delta)}{\log(p)}}, \frac{\theta^*}{\sigma}\Big]\ ;\quad \, |S_{t}| \leq 2t\|\theta^*\|_0\ ;\quad \,
\|\theta^*_{\overline{S}_{t}}\|_2^2\leq \sigma^2 c_{\eta}\|\theta^*_{\overline{S}_{t-1}}\|_0\frac{T\log(p/\delta)}{m}\ .
\eeq

Assume that the property $\bQ_T$ holds and recall that $\widehat{S}=S_T$. Since $|\cS(\theta^*) \setminus S_T|$ is an integer, $T \geq \log_2(n)$, and $\|\theta^*\|_0 \leq n/4$, there exists an event of probability larger than $(1-\delta)^{T}$ such that 
\beqn 
|\cS(\theta^*) \setminus \widehat{S}| &\leq& 2M\Big[2\sqrt{c_{\eta}T\frac{\log(p/\delta)}{\log(p)}}, \frac{\theta^*}{\sigma}\Big]\ ; \quad|\widehat{S}| \leq 2T\|\theta^*\|_0\ ;\\  \|\theta_{\overline{\widehat{S}}}^*\|_2^2&\leq &\sigma^2c_{\eta}2M\Big[2\sqrt{c_{\eta}T\frac{\log(p/\delta)}{\log(p)}}, \frac{\theta^*}{\sigma}\Big]\frac{T\log(p/\delta)}{m}\ ,
\eeqn
which, with $\delta=p^{-2}$, is the result of Theorem \ref{thm:support}. Thus, it suffices to prove $(\bQ_t)$ by induction.

\begin{lem}\label{lem:estbon} 
Assume that $\bSigma \in \mathcal U(\eta)$ with $\eta>0$. Assume that $|S_{t}|$  is such that 
\[
\underline{c}_{\eta}^{(SL),2}\left[\|\theta^*\|_0\log(p/\delta) + \log(1/\delta)\log(p)\right]\leq m/T-|S_t|\ . 
\]
(Recall that $\underline{c}_{\eta}^{(SL),2}$ is introduced in Lemma \ref{lem:square_root_Lasso}).
Then, given $S_{t}$, there exists an event $\cF_{t+1}$ measurable with respect to $(\underline{Y}^{(t+1)},\underline{\bX}^{(t+1)})$ of probability higher than $1-\delta$ such that 
\[
 \|\theta^*_{\overline{S}_{t+1}}\|_2^2 \leq c_\eta \sigma^2\frac{T\|\theta^*_{\overline{S}_t}\|_0}{m}\log\left(\frac{p}{\delta}\right)\text{ and } |S_{t+1}|\leq |S_t|+ 2\|\theta^*_{\overline{S}_{t}}\|_0\ . 
\]

\end{lem}

\paragraph{Step $(\bQ_1)$:} Recall that $S_0=\emptyset$. By Lemma~\ref{lem:estbon} and Equation~\eqref{eq:boundspa}, there exists an event $\cE_1$ with probability higher than $1-\delta$  such that
$$\|\theta^*_{\overline{S_1}}\|_2^2\leq c_\eta \sigma^2 \|\theta^*\|_0\frac{T\log(p/\delta)}{m}~~~\text{and}~~~|S_1| \leq 2 \|\theta^*\|_0\ .$$
Counting the components of $\theta^*_{\overline{S}_1}$ that are larger (in absolute value) than $2\sigma \sqrt{c_{\eta}T\log(p/\delta)/m}$, we derive that 
\[
\|\theta^*_{\overline{S_1}}\|_2^2\geq \left[|\cS(\theta^*)\setminus S_1|-M\left[2\sqrt{c_{\eta}T\frac{\log(p/\delta)}{\log(p)}}, \frac{\theta^*}{\sigma}\right]\right]  4c_{\eta}\sigma^2 \frac{T}{m}\log\left(\frac{p}{\delta}\right)\ ,
\]
which, together with the previous bound  implies 
$$|\cS(\theta^*) \setminus S_1| \leq M\big[2\sqrt{c_{\eta}T\frac{\log(p/\delta)}{\log(p)}}, \frac{\theta^*}{\sigma}\big]+ \frac{\|\theta^*\|_0}{4}\leq \frac{\|\theta^*\|_0}{2} \vee  \left(2M\big[2\sqrt{c_{\eta}T\frac{\log(p/\delta)}{\log(p)}}, \frac{\theta^*}{\sigma}\big]\right)\ .$$
So $(\bQ_1)$ holds.

\paragraph{Induction step:} Assume that $(\bQ_{t-1})$ holds for some $T-1 \geq t \geq 1$. By $(\bQ_{t-1})$ and on $\xi_{t-1}$, we have that $|S_{t-1}| \leq 2(t-1)\|\theta^*\|_0 \leq  m/(2T)$ by Condition~\eqref{eq:boundspa}. Thus, $m/T-|S_{t-1}|$ is large enough and we can apply Lemma~\ref{lem:estbon}. As a consequence, there exists an event $\cE_t$ of probability higher than $(1- \eta)^{t}$ such that 
\[
 \|\theta^*_{\overline{S}_{t}}\|_2^2 \leq c_\eta \sigma^2\frac{T\|\theta^*_{\overline{S}_{t-1}}\|_0}{m}\log\left(\frac{p}{\delta}\right)\text{ and } |S_{t}|\leq |S_{t-1}|+ 2\|\theta^*_{\overline{S}_{t-1}}\|_0\ . 
\]
Together with $(\bQ_{(t-1)})$, this implies $|S_{t}|\leq 2(t-1)\|\theta^*\|_0+ 2\|\theta^*\|_0= 2t\|\theta^*\|_0$. As for the proof of $(\bQ_1)$, we lower bound  $\|\theta^*_{\overline{S}_{t}}\|_2^2$ by considering separately the entries larger than (in absolute value) than $2\sigma\sqrt{c_{\eta}T\log(p/\delta)/m}$. This leads us to 
\beqn  
|\cS(\theta^*) \setminus S_t| &\leq& M\big[2\sqrt{c_{\eta}T\frac{\log(p/\delta)}{\log(p)}}, \frac{\theta^*_{\overline{S}_{t-1}}}{\sigma}\big]+ \frac{\|\theta^*_{\overline{S}_{t-1}}\|_0}{4}\leq \frac{\|\theta^*_{\overline{S}_{t-1}}\|_0}{2} \vee  \left(2M\big[2\sqrt{c_{\eta}T\frac{\log(p/\delta)}{\log(p)}}, \frac{\theta^*}{\sigma}\big]\right)\\
&\leq & \frac{\|\theta^*\|_0}{2^t}  \vee  \left(2M\big[2\sqrt{c_{\eta}T\frac{\log(p/\delta)}{\log(p)}}, \frac{\theta^*}{\sigma}\big]\right)\ ,
\eeqn 
where we used $(\bQ_{t-1})$ in the second line. We have proved $(\bQ_t)$. This concludes the proof.

\begin{proof}[Proof of Lemma~\ref{lem:mattransf}]

To alleviate the notation, we simply write $S$ for $S_{t-1}$, $\hat S$ for $\hat S^{(ith)}$, $\bX$ (resp. $Y$) for $\underline{\bX}^{(t)}$ (resp. $\underline{Y}^{(t)}$), $\bE$ for the  expectation $\bE^{(t)}$, and $\underline{\bPi}_S^{\perp}$ for $\underline{\bPi}_{t,S}^{\perp}$ (in the proof of this lemma only). Besides, since $S$ has been built based on independent samples, we consider it as fixed. Also, without loss of the generality, we assume that $S=\{1,\ldots, |S|\}$.

Define $\bZ = \bX-\mathbb{E}[\bX|\bX_{S}]$. Since $\bX$ follows a normal distribution, $\bZ$ is independent of $\bX_{S}$. Besides, the rows of $\bZ$ are i.i.d. distributed according to centered normal distribution with covariance matrix $\bSigma^{(t)}$. Since the rows of $\bX$ are i.i.d., each column of  $\mathbb{E}[\bX|\bX_{S}]$ is a linear combination of the columns of $\bX_S$. As a consequence, there exists a $|S|\times p$ matrix $\bR$ such that 
$\mathbb{E}[\bX|\bX_{S}]= \bX_S \bR$.

Since $T|S|<m$ and since $\bSigma$ is invertible, the rank of $V[S,\bX]$ equals $|S|$ almost surely. 
As a consequence, applying the orthogonal projection along $V[S,\bX]$ to $\bX$ leads to 
\[
\underline{\bPi}_S^{\perp}\bX= \underline{\bPi}_S^{\perp} \bZ + \underline{\bPi}_S^{\perp} \bX_S\bR= \underline{\bPi}_S^{\perp} \bZ \ . 
\]
Since the rows of $\bZ$ are i.i.d. with covariance $\bSigma^{(t)}$, there exists a matrix $\bU$ with i.i.d. standard normal entries such that $\bZ= \bU\bGamma^{(t)}$ where $\bGamma$ is a square root of $\bSigma^{(t)}$. As a consequence, $\underline{\bPi}_S^{\perp}\bX= \underline{\bPi}_S^{\perp} \bU\bGamma$. Since $\bX_S$ is independent of $\bU$ it follows that, given $\bX_S$, the $(m/T-|S|)\times p$ matrix $\underline{\bPi}_S^{\perp} \bU$ is made of independent standard normal entries and the rows of $\bX_{\perp}^{(t)}= \underline{\bPi}_S^{\perp}\bX$ therefore follow independent normal distributions with covariance matrix $\bSigma^{(t)}$.

Also we have
\[
Y_{\perp}^{(t)} =  \underline{\bPi}_S^{\perp} Y =\bX_{\perp}^{(t)} \theta^* + \sigma  \underline{\bPi}_S^{\perp} \epsilon =\bX_{\perp}^{(t)} \theta^*_{\overline{S}} + \sigma \underline{\bPi}_S^{\perp} \epsilon \ ,
\]
since the columns of $\bX_{\perp}^{(t)}$ in $S$ are equal to zero. 
Given $\bX_S$, $\underline{\bPi}_S^{\perp} \epsilon$  is projection of a standard normal vector onto a subspace of dimension $m/T-|S|$. As a consequence, $\underline{\bPi}_S^{\perp} \epsilon$ follows a normal distribution with covariance matrix $ \bI_{m/T-|S|}$ and is independent of $\bX$. The result follows. 
\end{proof}

\begin{proof}[Proof of Lemma \ref{lem:mistic}]
For simplicity, we write $S$ for $S_{t-1}$. 
Let $u$ be a normed vector supported  in $\overline{S}$. We shall prove that $u^T \bSigma^{(t)}u$ belongs to $(1/\eta,\eta)$. Consider a random vector $X\sim \cN(0,\bSigma)$ so that  $u^T \bSigma^{(t)}u=\var{u^TX|X_S}$. Consider the $|S|+1$ size covariance matrix $\bGamma$ of $((X_i)_{i\in S},u^TX)$. Then, $\bGamma\in \cU(\eta)$ and $\var{u^TX|\bX_S}=1/(\bGamma^{-1}_{|S|+1,|S|+1})$, which therefore lies in $(1/\eta,\eta)$. 
\end{proof}

\begin{proof}[Proof of Lemma~\ref{lem:estbon}] 
To alleviate the notation, we simply write $\widehat{\theta}_{(SL)}$ and 
$\widehat{\theta}_{(SL,t)}$ for $\widehat{\theta}_{(SL)}\big[\underline{Y}_{\perp}^{(t+1)},\underline{\bX}_{\perp}^{(t+1)} \big]$ and $\widehat{\theta}_{(SL,t)}\big[\underline{Y}_{\perp}^{(t+1)},\underline{\bX}_{\perp}^{(t+1)} \big]$ respectively. Recall that $S_{t+1}= S_t\cup \cS(\widehat{\theta}_{(SL,t)})$. The rows of $\underline{\bX}_{\perp}^{(t+1)}$ corresponding to indices in $S_t$ are null. Therefore, 
$\widehat{\theta}_{(SL)}\big[\underline{Y}_{\perp}^{(t+1)},\underline{\bX}_{\perp}^{(t+1)} \big]$ is a square-root Lasso estimator of $\underline{Y}_{\perp}^{(t+1)}$ given the restriction of $\underline{\bX}_{\perp}^{(t+1)}$ to the rows in $\overline{S_t}$. 
In view of Lemmas \ref{lem:mattransf} and \ref{lem:mistic}, we can apply  Lemma~\ref{lem:square_root_Lasso}. Thus, given $S_t$, there exists an event $\cF_t$ of probability higher than $1-\delta$ such that 
\begin{align}\label{eq:SL2}
\|\theta^*_{\overline{S}_t} - \widehat{\theta}_{(SL)}\|_2^2\leq \underline c_\eta^{(SL)}\sigma^2\frac{\|\theta^*_{\overline{S}_t}\|_0}{  m/T - |S_{t-1}|}\log\left(\frac{p}{\delta}\right)\ \, \quad \text{ and } \big|\frac{\widehat{\sigma}^{(t)}}{\sigma}-1\big|\leq \frac{1}{4}\ ,
\end{align}
By assumption,   $m/T\geq 2 |S_{t-1}|$.
Since $\widehat{\theta}_{(SL,t)}$ is a hard thresholded modification of 
$\widehat{\theta}_{(SL)}$ at level 
\[
\frac{8}{3}\widehat{\sigma}^{(t)} \sqrt{\underline{c}_\eta^{(SL)}\frac{T}{m} \log(\frac{p}{\delta})}\geq 2\sigma  \sqrt{\underline{c}_\eta^{(SL)}\frac{T}{m} \log(\frac{p}{\delta})}\ ,
\]
 its entry-wise error increases only at the non-zero entries of $\theta^*_{\overline{S}_t}$ and at most by $10/3\sigma\sqrt{\underline{c}_\eta^{(SL)}\frac{T}{m} \log(\tfrac{p}{\delta})}$.   This implies that 
\[
 \|\theta^*_{\overline{S}_{t+1}}\|_2^2 \leq \|\theta^*_{\overline{S}_t} - \widehat{\theta}_{(SL,t)}\|_2^2 \leq 2\big(2+ \tfrac{100}{9}\big)\sigma^2\underline{c}_\eta^{(SL)}\frac{T \|\theta^*_{\overline{S}_t}\|_0}{m}\log\left(\frac{p}{\delta}\right)\ . 
\]
Recall that $(S_{t+1}\setminus S_t)$ is the support of $\widehat{\theta}_{(SL,t)}$.
Each non-zero entry of $\widehat{\theta}_{(SL,t)}$ is equal to that of $ \widehat{\theta}_{(SL)}$. As a consequence,  each index in the support of $ \widehat{\theta}_{(SL,t)}$ and outside the support of $\theta^*$ contributes at least by $2\sigma^2  \underline{c}_\eta^{(SL)}\frac{T}{m} \log(\frac{p}{\delta})$ in the loss $\|\theta^*_{\overline{S}_t} - \widehat{\theta}_{(SL)}\|_2^2$. This implies 
\[
 \|\theta^*_{\overline{S}_t} - \widehat{\theta}_{(SL)}\|_2^2\geq 2 \sigma^2  \underline{c}_\eta^{(SL)}\frac{T}{m} \log(\frac{p}{\delta})\big|S_{t+1}\setminus \cS(\theta^*_{\overline{S}_t})\big|\ ,
\]
which in view of~\eqref{eq:SL2} leads us to 
$
|S_{t+1}\setminus \cS(\theta^*_{\overline{S}_t})|\leq \|\theta^*_{\overline{S}_t}\|_0$ and 
$$|S_{t+1}\setminus S_t|\leq \|\theta^*_{\overline{S}_t}\|_0+ |S_{t+1}\setminus \cS(\theta^*_{\overline{S}_t})|\leq 2\|\theta^*_{\overline{S}_t}\|_0\ ,$$
which concludes the proof. 

\end{proof}

\section{Proofs of the minimax lower bounds}

We first state the following classical lemma that links the total variation distance with the performance of a test with composite hypotheses. Some  variants of it may be found in textbooks such as~\cite{tsybakovbook}. For a sake of completeness, we provide a proof below. 
\begin{lem}\label{lem:minimax_lower_bound_general}
Consider a parametric model $\{\P_{\theta},\, \theta\in \Theta\}$ and two subsets $\Theta_0\subset \Theta ,\Theta_1\subset \Theta$. Let $\mu_0$ and $\mu_1$ be any probability measures on $\Theta$. Denote ${\bf P}_{\mu_i}=\int \P_{\theta}\mu_i(d\theta)$ for $i=0,1$.  Any test $\phi$ of $\Theta_0$ against $\Theta_1$ satisfies
\beq \label{eq:lower_bound_risk_test}
 \sup_{\theta\in \Theta_0}\P_{\theta}[\phi=1]+ \sup_{\theta\in \Theta_1}\P_{\theta}[\phi=0]\geq 1- \mu_0[\theta\notin \Theta_0] - \mu_1[\theta\notin \Theta_1] - \|{\bf P}_{\mu_0}- {\bf P}_{\mu_1}\|_{TV}\ .
 \eeq
 
\end{lem}

\begin{proof}[Proof of Lemma \ref{lem:minimax_lower_bound_general}]
For $i=0,1$, define the probability measure $\mu'_i$ by $\mu'_i[A]= \mu_i[A\cap \Theta_i]/\mu_i[\Theta_i]$ for any event $A$. Given $\mu'_i$, let ${\bf P}'_{\mu_i}=\int \P_{\theta}\mu'_i(d\theta)$ . It follows from Le Cam's arguments that 
\beq\label{eq:ll}
\sup_{\theta\in \Theta_0}\P_{\theta}[\phi=1]+ \sup_{\theta\in \Theta_1}\P_{\theta}[\phi=0]\geq 1- \|{\bf P}'_{\mu_0}- {\bf P}'_{\mu_1}\|_{TV}\ .
\eeq
By triangular inequality, one has
\[
 \|{\bf P}'_{\mu_0}- {\bf P}'_{\mu_1}\|_{TV}\leq \|{\bf P}_{\mu_0}- {\bf P}_{\mu_1}\|_{TV}+ \|{\bf P}'_{\mu_0}- {\bf P}_{\mu_0}\|_{TV}+ \|{\bf P}'_{\mu_1}- {\bf P}_{\mu_1}\|_{TV}
\]
Obviously, the total variation distance $\|\mu'_0 - \mu_0\|_{TV}$ equals $\mu'_0[\Theta_0]- \mu_0[\Theta_0]= \mu_0[\overline{\Theta_0}]$.
\[\|{\bf P}'_{\mu_0}- {\bf P}_{\mu_0}\|_{TV}=\sup_{\cA}\Big|\int \P_{\theta}(\cA)[\mu'_0(d\theta)- \mu_0(d\theta)]\Big|\leq \|\mu'_0 - \mu_0\|_{TV}\ .\]
Arguing similarly for $\|{\bf P}'_{\mu_1}- {\bf P}_{\mu_1}\|_{TV}$ and plugging these bound into \eqref{eq:ll} concludes the proof.

\end{proof}

\subsection{Proof of Proposition~\ref{prp:k0large}}

\begin{proof}[Proof of Proposition~\ref{prp:k0large}]
Intuitively, testing the sparsity for $k_0\geq n$ is impossible because  $\theta^*$ cannot be even recovered in noiseless setting ($\sigma=0$) when it contains more than $n$ non-zero entries. As the design matrix $\bX$ is random, this argument needs to be slightly refined.
Without loss of generality, we consider the case $p=n+1$, $k_0=n$ and $\Delta=1$. 
Let us write $\underline{\bX}$ the submatrix of $\bX$ made of its $n$ first  columns. In order to apply Lemma~\ref{lem:minimax_lower_bound_general}, we shall build two suitable prior distributions on the set of $n$ and $n+1$ sparse vectors.

With probability one, the square matrix $\underline{\bX}$ is invertible. Also denote $s_{\min}$ (resp. $s_{\max}$) the smallest (resp. highest) singular values of $\underline{\bX}$. Fix any $\delta \in (0,1)$.
As stated for instance in~\cite{MR2827856},  there exist $c_-(n,\delta)=c_->0$, $c_+(n,\delta)=c_+>0$ such that the following holds  
\beq\label{eq:control_c-}
\mathbb P_{\bX}\big[s_{min}>c_-\ ; s_{max}<c_+ \ ; c_-<\|\bX_{.,p}\|_2<c_+\big]\geq 1-\delta\ ,
\eeq
where $\mathbb P_{\bX}$ stands for the distribution of $\bX$. Here, $\bX_{.,p}$ stands for the $p$-th column of $\bX$. Although the exact expression of $c_-$ and $c_+$ is not relevant in this proof, these two quantities are of the order $n^{-1/2}$ and $n^{1/2}$.We call $\cA$ the event defined in the above probability bound.

Let $\mu_0$ stand for the centered Gaussian measure in $\mathbb{R}^{n+1}$ with covariance matrix $\big(\begin{array}{cc}
                                                                                                    \bI_n & 0 \\ 0 & 0 
                                                                                                   \end{array}\big)$
. We write ${\bf P}_{0,0}= \int_{\mathbb{R}^{n+1}}\P_{\theta,0}\mu_0(d\theta)$. Given any $r>0$, define the vector $v_r= (0,\ldots,0,r)^{T}$. We fix ${\bf P}_{1,r,0}= \int_{\mathbb{R}^{n+1}}\P_{\theta+v_r,0}\mu_0(d\theta)$. We argue that, for $r$ small enough, the total variation distance $\|{\bf P}_{0,0} - {\bf P}_{1,r,0}\|_{TV}$ is smaller than $2\delta$.  

Under $\mathbf{P}_{0,0}$, for a fixed  $\bX$,  it holds that $Y\sim \cN(0,\underline{\bX}\underline{\bX}^T)$ whereas, under $\mathbf{P}_{1,r,0}$, it holds that  $Y\sim \cN(\bX v_r,\underline{\bX}\underline{\bX}^T)$. When $\underline{\bX}$  satisfies $\cA$, these two covariance matrices are invertible with eigenvalues in $(c_-^2, c_{+}^2)$ and $\|\bX v_r\|_2\leq r c_+$. Thus, for $r$ going to zero,  the total variation distance between these conditional distributions   goes to zero uniformly over all $\bX$ satisfying $\cA$. In particular, there exists some $r_0$ such that these distances are uniformly smaller than $\delta$. Since $\P(\cA)\geq 1-\delta$, it follows that 
\[
\| {\bf P}_{0,0} - {\bf P}_{1,r_0,0}\|_{TV}\leq 2\delta.
\]
Consider $\sigma_0>0$ whose value will be fixed later. Define ${\bf P}_{0,\sigma_0}= \int_{\mathbb{R}^{n+1}}\P_{\theta,\sigma_0}\mu_0(d\theta)$ and ${\bf P}_{1,r_0,\sigma_0}= \int_{\mathbb{R}^{n+1}}\P_{\theta+v_r,\sigma_0}\mu_0(d\theta)$ the distributions associated to the linear regression models. By contraction properties of the total variation distances, one has 
\[
 \| {\bf P}_{0,\sigma_0} - {\bf P}_{1,r_0,\sigma_0}\|_{TV}\leq \| {\bf P}_{0,0} - {\bf P}_{1,r_0,0}\|_{TV}\leq 2\delta.
\]
When $\theta$ is sampled according to $\mu_0$, then the smallest (in absolute value) entry of $\theta$  among the $n$ first entries  is larger than some positive quantity $\underline{c}_{-}$, with probability larger than $1-\delta$.  Let us call $\cB$ the corresponding event. Define $\underline{\mu}$ as the measure $\mu_0$ conditioned to the event $\cB$, i.e. $\underline{\mu}(\cC)= {\mu}_{0}(\cC\cap \cB)/ {\mu}_{0}(\cB)$ for any measurable event $\cC$. Then, we introduce $\underline{\bf P}_{1,r_0,\sigma_0}= \int_{\mathbb{R}^{n+1}}\P_{\theta+v_r,\sigma_0}\underline{\mu}(d\theta)$. By triangular inequality, we obtain
\[
 \| {\bf P}_{0,\sigma_0} - \underline{\bf P}_{1,r_0,\sigma_0}\|_{TV}\leq 3\delta\ .
\]

When $\theta$ is sampled according to $\underline{\mu}$, $(\theta+v_{r_0})$ satisfies $d_2(\theta+v_{r_0},\bbB_0[n])\geq \underline{c_-}\wedge r_0$. As a consequence of Lemma \ref{lem:minimax_lower_bound_general}, any test of 
$\{\|\theta\|_0\leq n, \sigma=\sigma_0\}$ versus $\{\|\theta\|_0\leq n+1, d_2(\theta+v_{r_0},\bbB_0[n])\geq \underline{c}_-\wedge r_0, \sigma=\sigma_0\}$ has a risk higher than $1-3\delta$. We have 
\[
 \rho^*_{3\delta}[k_0,\Delta]\geq \frac{\underline{c}_-\wedge r_0}{\sigma_0}\ ,
\]
where $\underline{c}_{-}\wedge r_0$ does not depend on $\sigma_0$. Taking $\sigma_{0}$ arbitrarily small leads to the desired result.

\end{proof}

\subsection{Proof of Theorem~\ref{thm:lbkvkd}}

Given integers $k_0$ and $\Delta\leq p-k_0$, and $\rho>0$, we define the collection
\[
 \bbB_0[k_0,\Delta,\rho]= \bbB_0[k_0+\Delta]\ \bigcap\ \big\{\theta: d_2(\theta;\bbB_0[k_0]\geq\rho\big\}\ .
\]

We start by a simple reduction result to narrow the range of parameters. Its proof is postponed to the end of the section. 
\begin{lem}\label{lem:reduction}
For any $ \Delta'\leq \Delta\leq p-k_0$, we have
 \beq\label{eq:monotonic}
  \rho_{\gamma}^{*}[k_0,\Delta]\geq \rho_{\gamma}^{*}[k_0,\Delta']\ .
 \eeq
For the sake of the following bound, we explicit the dependency of $\rho_{\gamma}^{*}[k_0,\Delta]$ with respect to $p$ by denoting it  $\rho_{\gamma}^{*}[p,k_0,\Delta]$. For any $k'_0< k_0< p$ and $\Delta\leq p-k_0$, we have
 \beq\label{eq:subproblem}
 \rho_{\gamma}^{*}[p,k_0,\Delta] \geq \rho_{\gamma}^{*}[p-k_0+k'_0,k'_0,\Delta]\ .
 \eeq
 \end{lem}
In other words, the minimax separation distance in non-decreasing with respect to $\Delta$ and, up to a change in the number $p$ of covariates, it is also nondecreasing with respect to $k_0$. Next, we state three lemmas whose combination implies Theorem \ref{thm:lbkvkd}.

\begin{lem}\label{lem:detection}
Assume that $p\geq 2n$. There exists a numerical constant $c>0$ such that 
\beq \label{eq:lower_detection}
\rho_{\gamma}^{*2}[k_0,\Delta]  \geq c \Big[\frac{1}{\sqrt{n}} \bigwedge \frac{\Delta  \log(1+\sqrt{p}/\Delta)}{n}\Big]\  ,
\eeq
for any $\gamma\leq 0.53$,  all $k_0\leq n$ and $1\leq \Delta\leq p-k_0$.
\end{lem}
\begin{proof}[Proof of Lemma \ref{lem:detection}]
This lemma is a consequence of known signal detection lower bounds ($k_0=0$). For instance, it is proved in~\cite[][Sect.9.1]{verzelen_minimax} in 
\[\rho_{\gamma}^{*2}[0,\Delta]  \geq c \Big[\frac{1}{\sqrt{n}} \bigwedge \frac{\Delta \log(1+\sqrt{p}/\Delta)}{n}\Big]\ ,\]
for all $1\leq \Delta\leq p$. Since $p\geq 2n$ and $k_0\leq n$, Lemma~\ref{lem:reduction} entails that 
\[\rho_{\gamma}^{*2}[k_0,\Delta]  \geq c \Big[\frac{1}{\sqrt{n}} \bigwedge \frac{\Delta \log(1+\sqrt{p/2}/\Delta)}{n}\Big]\ ,\]
 which concludes the proof. 
\end{proof}

\begin{lem}\label{lem:tierce}
Assume that $p\geq 2n$. There exist constants $c_1$--$c_5$ such that the following holds for all $\gamma\leq 0.06$ :
\beq\label{eq:lower_tierce}
\rho_{\gamma}^{*2}[k_0,\Delta]  \geq c_1  \frac{\Delta\wedge k_0}{n}  \log\big(1+\frac{\sqrt{p}}{\Delta\wedge k_0}\big)\ ,
\eeq
for all $k_0\leq n$ and $\Delta>0$.
Furthermore, if $p\geq c_2 n^2$ and $k_0\in (c_3 n /\log(\sqrt{p}/n), \sqrt{p}/e^4)$ and $\Delta\leq k_0$, then
\beq\label{eq:lower_tierce2}
\rho_{\gamma}^{*2}[k_0,\Delta] \geq c_4\frac{\Delta }{n} \log\left(2\vee  \frac{\sqrt{p}}{k_0}\right)e^{c_5\frac{k_0}{n}\log(1+\frac{\sqrt{p}}{k_0})}\ .
\eeq

\end{lem}

\begin{proof}[Proof of Lemma \ref{lem:tierce}]
 In the above lemma, the minimax lower bounds both depend on the size $k_0$ of the null hypothesis and on the size $\Delta$ of the alternative hypothesis. As a consequence, we cannot directly rely anymore on signal detection results as in the previous lemma. Nevertheless, we will introduce a third party hypothesis and make  make use of previous signal detection lower bounds for unknown $\sigma$~\cite{2010_AS_Verzelen,verzelen_minimax}. 
 
 By Lemma~\ref{lem:reduction}, we assume without loss of generality that $\Delta\leq k_0$ . Given $\rho>0$ and $1\leq k\leq p$, we define $\mu_{\rho,k}$ as the uniform measure over the set 
 \[
\Big\{\theta\in \mathbb{R}^p,\, \|\theta\|_{0}=k,\quad \forall i\in\{1,\ldots,p\},\quad \, |\theta_i|= 0\text{ or }\rho/\sqrt{k}\Big\}\ ,
 \]
and the mixture measure $\mathbf{P}_{\rho,k}= \int \bP_{\theta,1}\mu_{\rho,k}(d\theta)$. As a way to derive minimax lower bounds for signal detection with unknown noise level, it is proved in \cite[][Theorem 4.3]{2010_AS_Verzelen} and \cite[][Lemma 9.3]{verzelen_minimax}\footnote{Actually, the results in \cite{2010_AS_Verzelen,verzelen_minimax} are expressed in terms of minimax separation distance, the total variation distance control being stated in their respective proof.} that 
$\|\mathbf{P}_{\rho,k} - \bP_{0,1+\rho^2}\|_{TV}\leq 0.47$ if $\rho\leq \rho_{k}$ or if $\rho\leq \rho'_{k}$ with 
\beqn 
 \rho_k^2&= &\frac{k}{2n}\log\left(1+\frac{\sqrt{p}}{k} \right)\ ;\\
 \rho^{'2}_{k}&=& -1 + \left(\frac{p}{2ek} \right)^{k/n}(4k)^{-2/n}\quad \text{ for }\quad  k\log\left(\frac{\sqrt{p}}{e^{3/2}k}\right)\geq n\ . 
\eeqn 
Let us now deduce \eqref{eq:lower_tierce}. Since $\rho_k$ is increasing with respect to $k$, we have
\[
 \|\mathbf{P}_{\rho_{k_0},k_{0}}- \mathbf{P}_{\rho_{k_0},k_{0}+\Delta}\|_{TV}\leq \|\mathbf{P}_{\rho_{k_0},k_{0}}- \bP_{0,1+\rho_{k_0}^2}\|_{TV}+ \| \bP_{0,1+\rho_{k_0}^2}- \mathbf{P}_{\rho_{k_0},k_{0}+\Delta}\|_{TV}\leq 0.94
\]
Under $\mu_{\rho_{k_0},k_{0}}$, $\theta$ is $k_0$-sparse, whereas under $\mu_{\rho_{k_0},k_{0}+\Delta}$, $\theta$ is $k_0+\Delta$-sparse and  its square distance to $\bbB_0[k_0]$ is $\Delta \rho^2_{k_0}/(k_0+\Delta)$. From Lemma \ref{lem:minimax_lower_bound_general}, we deduce that, for $\gamma\leq 0.06$, one has 
\[
\rho_{\gamma}^{*2}[k_0,\Delta]\geq \Delta \rho^2_{k_0} \frac{\Delta}{\Delta+k_0}\geq \rho^2_{k_0}\frac{\Delta}{2k_0}\ ,
\]
which enforces \eqref{eq:lower_tierce} since we have $\Delta\leq k_0$. Turning to \eqref{eq:lower_tierce2}, we observe that, under the assumptions of the lemma (and with a suitable choice of $c_2$), 
$k\log\big(\tfrac{\sqrt{p}}{e^{3/2}k}\big)\geq n$ both for $k=k_0$ and $k=k_0+\Delta\leq 2k_0$. Arguing as above, we deduce that 
\[
\rho_{\gamma}^{*2}[k_0,\Delta]\geq \rho^{'2}_{k_0}\frac{\Delta}{2k_0}\geq c\frac{\Delta}{k_0}e^{c'\frac{k_0}{n}\log\left(\frac{p}{k_0}\right)}\geq c_2\frac{\Delta}{n}\log\left(\frac{p}{k_0}\right)e^{c_3\frac{k_0}{n}\log\left(\frac{p}{k_0}\right)}\ , 
\]
since the expression inside the exponential is bounded away from zero and since $e^{x}\geq 1+x$ for $x>0$. We have proved~\eqref{eq:lower_tierce2}.
\end{proof}

The following lemma provides the key new lower bound. It corresponds to the regime where both $k_0$ and $\Delta$ are large. Its proof relies on  more advanced arguments than the other regimes.

\begin{lem} \label{lem:lower_moments}  
There exists positive numerical constant $c$ and $c_2$ such that the following holds for any $\gamma\leq 0.5$ and all $p\geq c_2$. 
For any $p^{1/4}\leq k_0 \leq n$ and $\Delta \geq k_0^{2/3}\vee p^{1/4}$, one has 
\beq\label{eq:lower_moments_minimax}
\rho^{*2}_{\gamma}[k_0,\Delta] \geq  c \frac{\Delta}{n } \frac{\log^2\big[1+ \sqrt{\frac{k_0}{\Delta}}\big]}{\log(p)}\ .
\eeq
\end{lem}

\begin{proof}[Proof of Theorem~\ref{thm:lbkvkd}]

First we prove \eqref{eq:lower_lbkvkd}. 
The case $k_0\leq \sqrt{p}$ is a consequence of Lemmas \ref{lem:detection} and \ref{lem:tierce}. As for the case $k_0\in (\sqrt{p}, n)$, we divide the analysis into several subcases. If $\Delta \leq p^{1/4}$, it follows from Lemma \ref{lem:tierce} that $\rho_{\gamma}^{*2}[k_0,\Delta]$  is at least of the  order of $\Delta\log(p)/n$  which is larger than the lower bound in \eqref{eq:lower_lbkvkd}. For $\Delta \geq p^{1/4}\vee k_0^{2/3}$ we rely on Lemma \ref{lem:lower_moments}. For $\Delta \in (p^{1/4},k_0^{2/3})$, we define $k'_0=\lfloor \Delta^{3/2}\rfloor$. From the reduction \eqref{eq:subproblem} and  Lemma \ref{lem:lower_moments}, we derive that 
\[
 \rho_{\gamma}^{*2}[k_0,\Delta]\geq c\frac{\Delta}{n } \frac{\log^2\big[1+ \sqrt{\frac{k'_0}{\Delta}}\big]}{\log(p-k_0+k'_0)}\geq c' \frac{\Delta}{n}\log(p) \geq c''\frac{\Delta}{n } \frac{\log^2\big[1+ \sqrt{\frac{k_0}{\Delta}}\big]}{\log(p)}\ . 
\]
Finally, the lower bound \eqref{eq:lower_ultra_high_known_variance} is a consequence of the second part of Lemma \ref{lem:tierce} together with the reduction lemma~\ref{lem:reduction}. 

\end{proof}

\begin{proof}[Proof of Lemma \ref{lem:reduction}]
The first bound is a simple consequence of the inclusion $\bbB_0[k_0,\Delta,\rho]\subset \bbB_0[k_0,\Delta',\rho]$. Let us turn to \eqref{eq:subproblem}. Take any $\zeta>0$ arbitrarily small and define $r= \rho_{\gamma}^{*}[k_0,\Delta]+\zeta$. There exists a test $\phi$ satisfying $R[\phi;k_0,\Delta,r]\leq \gamma$. For any linear regression problem with $p-k_0+k'_0$ covariates and response $Y$, we sample $k_0-k'_0$ new independent covariates, write $\underline{\bX}$ the corresponding new design matrix of size $n\times (k_0-k'_0)$, and define $\underline{Y}= Y + r\underline{\bX}1$ where $1$ is the constant vector of size $k_0-k'_0$. Since $R[\phi;k_0,\Delta,r]\leq \gamma$, we have
\[\sup_{\theta,\ \|\theta\|_0\leq k'_0} \mathbb P_{\theta}[\phi(\underline{Y})=1]+ \sup_{\theta,\ \|\theta\|_0\leq k'_0+ \Delta,\ d_2(\theta,\bbB_0[k'_0])\geq r}\mathbb P_{\theta}[\phi(\underline{Y})=0]\leq \gamma,\]
implying that $\rho_{\gamma}^{*}[p-k_0+k'_0,k'_0,\Delta]\leq r$. Taking the infimum over all $\zeta>0$, we obtain \eqref{eq:subproblem}.
\end{proof}

\begin{proof}[Proof  of Lemma \ref{lem:lower_moments}]

Without loss of generality we assume that the noise level $\sigma$ is equal to one and we write $\P_{\theta}$ for $\P_{\theta,1}$. Since the minimax separation distance $\rho_{\gamma}^{*}[k_0,\Delta]$ is a nondecreasing function of $\Delta$, we have $\rho_{\gamma}^{*}[k_0,\Delta]\geq \rho_{\gamma}^{*}[k_0,k_0]$ for any $\Delta>k_0$. In view of \eqref{eq:lower_moments_minimax} and since $\log(1+x)\geq x/2$ for any $x\in [0,1]$,  we only need to prove \eqref{eq:lower_moments_minimax} for $\Delta\leq k_0$.

Define $\overline{k}_0= k_0-\Delta/2$ and $\overline{k}_1= k_0+\Delta/2$. We introduce two priors $\mu_0^{\otimes p}$ and $\mu_{1}^{\otimes p}$ that are almost supported on $\bbB_0[k_0]$ and $\bbB_0[k_0+\Delta]$ respectively and such that the first moments of $\mu_0$ and $\mu_1$ are matching. In Step 3 below, we show that this moment matching property ensures that the corresponding mixture distributions of $(Y,\bX)$ are close in total variation distance.

\bigskip

\noindent 
\paragraph{Step 1. Construction of the priors.} As in \cite{carpentier2017adaptive}, we build prior measures $\mu_0$ and $\mu_1$ in such a way that their first moments are matching. Define the two quantities where $m$ is redefined only in this proof as follows)
\beq \label{eq:definition_m_M}
m= 2\lfloor 2\log(p)\rfloor\ , \quad \quad M= c \sqrt{\log(p)/n}\ , 
\eeq
for some universal constant $c$ whose value will be fixed later.
The following result is borrowed from~\cite[Lemma 3]{carpentier2017adaptive}.
\begin{lem}\label{lem:nemirovski2}
Given any positive and even integer $m$ and $q\in (0,1)$, define 
\beq\label{eq:defintion_am}
a_{m}= \tanh\Big[\frac{1}{m}\ \arg\cosh\big(\frac{1+q}{1-q}\big)\Big]\ .
\eeq
There exists two positive and symmetric measures $\nu_0$ and $\nu_1$ whose support lie in $[-1,-a_{m}]\cup [a_{m},1]$ satisfying:
\begin{eqnarray}
 \int \nu_0(dt)&=&q\ ;\quad \quad \int \nu_1(dt)=1 \ ;\label{eq:condition_moment0} \\
 \int t^d \nu_0(dt)&= &\int t^d \nu_1(dt) ;\quad \quad d=1,\ldots, m\ .\label{eq:condition_momentq}
\end{eqnarray}
\end{lem}

Fix $q=\overline{k}_0/\overline{k}_1$. Then, given $m=  2\lfloor 2\log(p)\rfloor$, we consider the measures $\nu_0$ and $\nu_1$ as in Lemma \ref{lem:nemirovski2}. Given any measurable event $A$, we define $\mu_0$  and $\mu_1$ by 
\beq
\mu_0(A)= (1-\frac{\overline{k}_0}{p}) \delta_0[M . A] +\frac{\overline{k}_1}{p}\nu_0[M  . A] \ , \quad    \mu_1(A)= (1-\frac{\overline{k}_1}{p})\delta_0[M . A] + \frac{\overline{k}_1}{p}\nu_1[M .A] \ .
\eeq
 Here,  $M.A$ stands for $\{Mx : x\in A\}$ and $\delta_0$ is the Dirac measure at $0$.  In view of this definition, the first $m$  moments of $\mu_0$ and $\mu_1$ are matching.
\bigskip

\noindent 
\paragraph{Step 2. Properties of the priors.} We consider the prior measures $\mu_0^{\otimes p}$ and $\mu_1^{\otimes p}$. In view of Lemma \ref{lem:minimax_lower_bound_general}, we need to show that $\mu_0^{\otimes p}$ is concentrated on $\bbB_0[k_0]$ and that $\mu_1^{\otimes p}$ is concentrated on $\bbB_0[k_0,\Delta,\rho]$ for some large $\rho$.

Under $\mu_0^{\otimes p}$, $\|\theta\|_0$ follows a binomial distribution with parameter $(p, (k_0-\Delta/2)/p)$. By Chebychev's inequality, 
\beq \label{eq:lower_theta_0}
\mu_0^{\otimes p}[\|\theta\|_0> k_0]\leq 4\frac{k_0}{\Delta^2}\leq 4 p^{-1/8}\ ,
\eeq
since $\Delta \geq p^{1/4}\vee k_0^{2/3}$. 
Similarly,
\beq \label{eq:lower_theta_1} 
\mu_0^{\otimes p}[\|\theta\|_0\in (k_0+\Delta/4, k_0+ 3\Delta/4)]\geq 1 - \frac{32k_0}{\Delta^2}\geq  1 -32 p^{-1/8}\ .
\eeq
Under the event $\|\theta\|_0\in (k_0+\Delta/4, k_0+ 3\Delta/4)$, the corresponding parameter $\theta$ satisfies
\beqn
d^2_2(\theta, \bbB_0[k_0])&\geq &\frac{\Delta}{4} a^2_{m}M^2
=  \frac{c^2}{4} \Delta \frac{\log(p)}{n} \tanh^2\big[\frac{1}{m} \arg\cosh[1+ \frac{2\overline{k}_0}{\Delta}]\big]\ . 
\eeqn 
Since $ \arg\cosh[1+ 2\overline{k}_0/\Delta]\leq \arg\cosh[1+ 2p]\leq 4\log(p)$ for $p\geq 2$ and since $\tanh(t)\geq 0.4 t$ for any $t\in (0,1)$, we deduce that 
\beq \label{eq:lower_theta_2} 
d^2_2(\theta, \bbB_0[k_0]) \geq  c' \frac{\Delta}{n \log(p)} \arg\cosh^2[1+ \frac{2\overline{k}_0}{\Delta}] 
\geq  c' \frac{\Delta}{n \log(p)} \log^2[1+ \frac{k_0}{\Delta}]\ .
\eeq
$\arg\cosh(x)\geq \log(x)$. As a consequence, with probability $\mu_{1}^{\otimes p}$ larger than $1- 32 p^{-1/8}$, $\theta$ belongs to $\bbB_0[k_0,\Delta, \rho]$ with $\rho^2=  c' \frac{\Delta}{n \log(p)} \log^2[1+ \frac{k_0}{\Delta}]$. To apply Lemma \ref{lem:minimax_lower_bound_general}, it remains to bound the total variation distance between 
\[\mathbf{P}_0= \int \P_{\theta}\mu_0^{\otimes p}(d\theta)\  \quad\text{ and } \quad \mathbf{P}_1= \int \P_{\theta}\mu_1^{\otimes p}(d\theta)\ . \]

\bigskip 

\noindent
\paragraph{Step 3. Control of $\|\mathbf{P}_0-\mathbf{P}_1\|_{TV}$.} 

For $j=0, \ldots, p$, define the distribution $\mathbf{P}^{(j)}_0= \int \P_{\theta}\mu_1^{\otimes j}\otimes \mu_0^{\otimes p-j}(d\theta)$ with $\mathbf{P}^{(0)}_0=\mathbf{P}_0$ 
and $\mathbf{P}^{(p)}_0=\mathbf{P}_1$. By triangular inequality, one has 
\beq \label{eq:decomposition_TV}
 \|\mathbf{P}_0-\mathbf{P}_1\|_{TV}\leq  \sum_{j=0}^p \|\mathbf{P}^{(j)}_0-\mathbf{P}^{(j+1)}_0\|_{TV}.\
\eeq
This upper bound greatly simplifies the following computations as the distributions $\mathbf{P}^{(j)}_0$ and $\mathbf{P}^{(j+1)}_0$ only differ by one coordinate. Unfortunately, we conjecture that our minimax lower bound in Theorem~\ref{thm:lbkvkd} is suboptimal in the regime where $k_0$ is close to $\sqrt{p}$ precisely because of the upper bound~\eqref{eq:decomposition_TV}. In the arguably simpler Gaussian sequence model~\cite{carpentier2017adaptive}, we have directly computed the $\chi^2$ distances between the corresponding distributions $\mathbf{P}_0$ and $\mathbf{P}_1$ to obtain the sharp separation distance in all regimes. If we use instead the decomposition \eqref{eq:decomposition_TV} for the  Gaussian sequence model, this leads to a suboptimal lower bound for $k_0$ close to $\sqrt{p}$. To close this gap in the linear regression model, one would therefore need to directly handle the $\chi^2$ distance between $\mathbf{P}_0$ and $\mathbf{P}_1$ but we were not able to do it. 

In the following, we shall bound independently each of these $p$ distances  $\|\mathbf{P}^{(j)}_0-\mathbf{P}^{(j+1)}_0\|_{TV}$.  Interestingly, $\mathbf{P}^{(j)}_0$ and $\mathbf{P}^{(j+1)}_0$ only differ by the distribution of the $j+1$-th coordinate of $\theta$. The general idea is to condition with respect to all the coordinates except the $j+1$-th one so that we consider a linear regression model with only one covariate.

Let us write $g^{(j)}_0(Y|\bX)$ the conditional density of $Y$ given $\bX$ under $\mathbf{P}^{(j)}_0$. 
\[
g^{(j)}_0(Y|\bX) = (2\pi)^{-n/2}\int  \exp\Big(-\frac{\|Y - \bX\theta\|_2^2}{2}\Big)\mu_1^{\otimes j}\otimes \mu_0^{\otimes p-j}(d\theta).
\]
Writing down $\E_{\bX}$ the expectation with respect to $\bX$, we have
\beqn
2\|\mathbf{P}^{(j)}_0-\mathbf{P}^{(j+1)}_0\|_{TV}& = & \E_{\bX}\Big[\int \Big| g^{(j)}_0(Y|\bX) - g^{(j+1)}_0(Y|\bX)\Big|dY \Big]\nonumber\\
& =& (2\pi)^{-n/2}\E_{\bX}\Big[\int \Big|\int e^{-\|Y - \bX\theta\|_2^2/2}\mu_{0}^{\otimes j}\otimes [\mu_0-\mu_1]\otimes \mu_{1}^{\otimes p-j-1}(d\theta)  \Big|dY\Big] \nonumber \\
 & = &  (2\pi)^{-n/2}\E_{\bX}\Big[\int\Big|\int e^{-\|Y - \bX\theta\|_2^2/2}\mu_{0}^{\otimes j}\otimes \mu_{1}^{\otimes p-j-1}\otimes [\mu_0-\mu_1](d\theta)  \Big|dY\Big]\ ,  
\eeqn
by permutation invariance. We call  $A_j$ this last quantity. 

Given a $p$-dimensional vector $\theta$, let  $\theta^{(-p)}$ be such that $\theta^{(-p)}_p=0$ and $\theta^{(-p)}_j=\theta_j$ for all $j < p$. Write $\boldsymbol{\mu}_j= \mu_{0}^{\otimes j}\otimes \mu_{1}^{\otimes p-j-1}$ and $\mu_{\Delta}= \mu_0-\mu_1$.

\[
(2\pi)^{n/2}A_j= \E_{\bX}\Big[\int\Big|\int e^{-\|Y - \bX\theta^{(-p)}\|_2^2/2}  \Big[ \int e^{\omega_p \theta_p + \xi_p \theta_p^2}  \mu_{\Delta}(d\theta_p) \Big]\boldsymbol{\mu}_j(d\theta^{-(p)})\Big|dY\Big]\ , 
\]
where the quantities $\omega_p$ and $\xi_p$ are defined by 
$$\omega_p = (Y -\bX\theta^{(-p)})^T  \bX_{.,p}\ ; ~~~~ \xi_p = \sum_{i} \bX_{i,p}^2.$$
Let $\Omega$ be the event such that  $\Omega= \{ |\omega_p| \leq 5\sqrt{n\log(p)},~\xi_p\leq 2n\}$.

Fix any $\theta\in \mathbb{R}^p$ such that $\|\theta\|_{\infty}\leq M$. Then, under $\P_{\theta}$, $\xi_p$ follows a $\chi^2$ distribution with $n$ degrees of freedom. As a consequence of deviation inequalities for $\chi^2$ distributions (Lemma \ref{lem:conchi}), its probability to be larger than $2n$ is smaller than $e^{-n/16}$. Besides, conditionally to $\bX$, $\omega_p$ follows a normal distribution with mean $\theta_p\|\bX_{.p}\|_2^2$ and variance $\|\bX_{.p}\|_2^2$. As a consequence, under the event $\{\xi_p\leq 2n\}$, the probability that $|\omega_p|\geq 2Mn + 2\sqrt{2n\log(p)}$ is smaller than $1/p^2$. In view of the definition \eqref{eq:definition_m_M} of $M$ and by taking the constant $c$ in that definition small enough, we conclude that 
\begin{align}\label{eq:probaomeg}
\mathbb P_{\theta}(\Omega) \geq 1-1/p^2- e^{-n/16}\ ,
\end{align}
for all $\theta$ such that $\|\theta\|_{\infty}\leq M$.
We set
\[A_{j,\Omega}= (2\pi)^{-n/2}\E_{\bX}\Big[\int\Big|\int e^{-\|Y - \bX\theta^{(-p)}\|_2^2/2}  \Big[ \int {\bf 1}_{\Omega}e^{\omega_p \theta_p + \xi_p \theta_p^2}  \mu_{\Delta}(d\theta_p) \Big]\boldsymbol{\mu}_j(d\theta^{-(p)})\Big|dY\Big]\ , 
\]
It follows from this definition and from Equation~\eqref{eq:probaomeg} that
\begin{eqnarray}
A_j &\leq & A_{j,\Omega} + (2\pi)^{-n/2}\E_{\bX} \Big[\int \Big|\int{\bf 1}_{\overline{\Omega}}  e^{-\|Y - \bX\theta\|_2^2/2} \boldsymbol{\mu}_j(d\theta^{-(p)}) \mu_{\Delta}(d\theta_p) \Big|dY\Big] \nonumber \\
&\leq & A_{j,\Omega} + 
\sup_{\theta \in [-M, M]^p}\mathbb P_{\theta} [ \overline{\Omega}] \leq A_{j,\Omega} + \frac{1}{p^2} + e^{-n/16} . \label{eq:probaAd}
\end{eqnarray}
.

In order to work out the term $A_{j,\Omega}$, we rely on the power expansion of  $e^{\omega_p \theta_p + \xi_p \theta_p^2}$ together with the nullity of the $m$ first moments of  $\mu_{\Delta}$. 
\beqn 
\Big|\int {\bf 1}_{\Omega}e^{\omega_p \theta_p + \xi_p \theta_p^2}  \mu_{\Delta}(d\theta_p) \Big|& =& \Big|\sum_{d=m/2+1}^{\infty}\int  {\bf 1}_{\Omega}\frac{\Big(\omega_p\theta_p + \xi_p \theta_p^2/2  \Big)^d}{d!} \mu_{\Delta}(d\theta_p)  \Big|
\\ 
& \leq & \frac{\overline{k}_1}{p}\sum_{d=m/2+1}^{\infty} \frac{\Big(5 M\sqrt{n\log(p)}   + n M^2  \Big)^d}{d!}\\
& \leq & \frac{\overline{k}_1}{p}\sum_{d=m/2+1}^{\infty} \Big[\frac{2e(5M\sqrt{n\log(p)} +  nM^2) }{m} \Big]^d \leq  \frac{\overline{k}_1}{p2^{m/2}} \ ,
\eeqn 
since, by \eqref{eq:definition_m_M}, $m\geq 10e (5M\sqrt{n\log(p)} +  nM^2)$ if we fix $c \leq 1/(10e)$.
Plugging this bound into the definition of $A_{j,\Omega}$, we obtain

\beq
A_{j,\Omega} \leq  \frac{\overline{k}_1}{p2^m}(2\pi)^{-n/2}\E_{\bX}\Big[\int\int e^{-\|Y - \bX\theta^{(-p)}\|_2^2/2}  \boldsymbol{\mu}_j(d\theta^{-(p)})dY\Big] \leq \frac{\overline{k}_1}{p2^m}\leq  \frac{1}{p^2}\ ,
\eeq
by definition  \eqref{eq:definition_m_M} of $m$. 
Together with~\eqref{eq:probaAd}, this implies that $A_{j}\leq 2/p^2+ e^{-n/16}$. Then, we use the definition of $A_j$ and~\eqref{eq:decomposition_TV} to conclude that 
\[
 \|\mathbf P_0 - \mathbf P_1\|_{TV} \leq \frac{2}{p}+ pe^{-n/16}\ ,
\]
which is smaller than $1/4$ for $p$ large enough since the assumptions of Lemma~\ref{lem:lower_moments} enforce that $p^{1/4}\leq n$. 
In view of the above bound, \eqref{eq:lower_theta_0}, \eqref{eq:lower_theta_1}, and \eqref{eq:lower_theta_1}, we are in position to apply Lemma \ref{lem:minimax_lower_bound_general}. Thus, for $p$ large enough, we conclude that 
\[
\rho^{*2}_{\gamma}[k_0,\Delta] \geq c \frac{\Delta}{n \log(p)} \log^2\big[1+ \frac{k_0}{\Delta}\big]\ . 
\]

\end{proof}

\subsection{Proof of Proposition \ref{prp:lbuvkd}}

Since $\boldsymbol{\rho}_{g,\gamma}^{*}[k_0,\Delta]\geq \rho_{\gamma}^*[k_0,\Delta]$, the second part of \eqref{eq:lower_lbuvkd} comes from Theorem \ref{thm:lbkvkd}. Turning to the first part of \eqref{eq:lower_lbuvkd}, we have already pointed out in the proof of Lemma \ref{lem:tierce} that it is proved in~\cite{2010_AS_Verzelen} and \cite{verzelen_minimax} that, for all $(k,n,p)$ one has 
\[
\boldsymbol{\rho}_{g,\gamma}^{*2}[0,k] \geq c \left\{\begin{array}{cc}
\frac{k}{2n}\log\left(1+\frac{\sqrt{p}}{k} \right)&\\
-1 + \left(\frac{p}{2ek} \right)^{k/n}(4k)^{-2/n}&\text{ if } k\log\left(\frac{\sqrt{p}}{e^{3/2}k}\right)\geq n.
                                                                     \end{array}\right.
\]
These two bounds imply that, for $p$ large enough and for all $\Delta\leq \sqrt{p}/e^3$, one has
\[
\boldsymbol{\rho}_{g,\gamma}^{*2}[0,\Delta] \geq c \frac{\Delta}{n}\log\left(1+\frac{\sqrt{p}}{\Delta}\right)\exp\left[c'\frac{\Delta}{n}\log\left(1+\frac{\sqrt{p}}{\Delta}\right)  \right].
\]
Since $\boldsymbol{\rho}_{g,\gamma}^{*}[0,\Delta]$ is nondecreasing with respect to $\Delta$ (Lemma \ref{lem:reduction}), the above bound is also valid for all $\Delta\leq p$ at the price of worse constants. Finally, we apply Lemma \ref{lem:reduction} together with the assumption $k_0\leq n\leq p/2$ to obtain the first part of \eqref{eq:lower_lbuvkd}.

\section*{Acknowledgements.} The work of A. Carpentier is partially supported by the Deutsche Forschungsgemeinschaft (DFG) Emmy Noether grant MuSyAD (CA 1488/1-1), by the DFG - 314838170, GRK 2297 MathCoRe, by the DFG GRK 2433 DAEDALUS, by the DFG CRC 1294 'Data Assimilation', Project A03, and by the UFA-DFH through the French-German Doktorandenkolleg CDFA 01-18. The authors thank anonymous reviewers for their helpful suggestions that improved the manuscript. The authors are also grateful to Alexandre Tsybakov and Cun-Hui Zhang for bringing to our knowledge some recent work on MCP.

\bibliography{biblio}
\bibliographystyle{plain}

\end{document}